\documentclass[11pt,leqno]{amsart}
\usepackage{amsthm,amsfonts,amssymb,amsmath,oldgerm,mathrsfs,mathabx,amsbsy}
\usepackage[scr=boondoxo]{mathalfa}
\numberwithin{equation}{section}
\usepackage{fullpage,setspace,fancyhdr,bm}
\usepackage{graphicx,psfrag,listings} 
\usepackage{color}

\usepackage[top=30mm,bottom=30mm,left=25mm,right=25mm,a4paper]{geometry}

\usepackage{hyperref}

\renewcommand\d{\partial}
\DeclareMathOperator{\beD}{\mathbf{e}}
\newcommand\ii{\iD}
\newcommand\dd{\dD}
\DeclareMathOperator{\iD}{i}
\DeclareMathOperator{\eD}{e}
\DeclareMathOperator{\dD}{d}

\def\eps{\varepsilon }

\newcommand{\I}{\text{I}}
\DeclareMathOperator{\Tr}{Tr}

\DeclareMathOperator{\sign}{sgn}
\DeclareMathOperator{\diag}{diag}
\DeclareMathOperator{\Hess}{Hess}
\DeclareMathOperator{\Jac}{Jac}
\DeclareMathOperator{\Div}{div}
\DeclareMathOperator{\Curl}{curl}
\def\transp#1{{#1}^{{\sf T}}}

\newcommand{\Ham}{\mathscr{H}}
\newcommand{\Hamp}{\mathscr{H}_{\rm u}}
\newcommand{\Hamkp}{\mathscr{H}_{\tkp}}
\newcommand{\Hp}{H_{\rm u}}
\newcommand{\HEK}{H_{\rm EK}}

\newcommand{\bImpulse}{\pmb{\mathscr{Q}}}
\newcommand{\Impulse}{\mathscr{Q}}
\newcommand{\impulse}{\mathscr{q}}
\newcommand{\bimpulse}{\pmb{\mathscr{q}}}
\newcommand{\uimpulse}{{\underline \impulse}}

\newcommand{\Mass}{\mathscr{M}}
\newcommand{\mass}{\mathscr{m}}
\newcommand{\umass}{{\underline \mass}}
\newcommand{\uxip}{\uxi_{\phi}}
\newcommand{\ubkp}{\ubk_{\phi}}
\newcommand{\ubkx}{\ubk_{x}}
\newcommand{\utkp}{\widetilde{\ubk}_{\phi}}
\newcommand{\uex}{\underline{\beD}_{x}}
\newcommand{\ukp}{\uk_{\phi}}
\newcommand{\ukx}{\uk_{x}}

\newcommand{\uXx}{\uX_{x}}

\newcommand{\ucx}{\uc_{x}}
\newcommand{\uomp}{\uom_{\phi}}
\newcommand{\uomx}{\uom_{x}}
\newcommand{\umup}{\umu_{\phi}}
\newcommand{\umux}{\umu_{x}}
\newcommand{\xip}{\xi_{\phi}}
\newcommand{\bkp}{\bfk_{\phi}}
\newcommand{\bkx}{\bfk_{x}}
\newcommand{\tkp}{\widetilde{\bfk}_{\phi}}
\newcommand{\ex}{\beD_{x}}
\newcommand{\kp}{k_{\phi}}
\newcommand{\kx}{k_{x}}

\newcommand{\Xx}{X_{x}}

\newcommand{\cx}{c_{x}}
\newcommand{\omp}{\omega_{\phi}}
\newcommand{\omx}{\omega_{x}}
\newcommand{\mup}{\mu_{\phi}}
\newcommand{\mux}{\mu_{x}}
\newcommand{\vphip}{\varphi_{\phi}}
\newcommand{\vphix}{\varphi_{x}}
\newcommand{\Sp}{\mathscr{S}_{\phi}}
\newcommand{\Sx}{\mathscr{S}_{x}}
\newcommand{\rhomin}{\rho_{\rm min}}
\newcommand{\rhomax}{\rho_{\rm max}}
\newcommand{\rhodual}{\rho_{\rm dual}}

\newcommand\br{\begin{remark}}
\newcommand\er{\end{remark}}
\newcommand\bp{\begin{pmatrix}}
\newcommand\ep{\end{pmatrix}}
\newcommand{\be}{\begin{equation}}
\newcommand{\ee}{\end{equation}}
\newcommand\ds{\displaystyle}
\newcommand\nn{\nonumber}
\newcommand{\beg}{\begin{example}}
\newcommand{\eeg}{\end{exaplem}}
\newcommand{\bpr}{\begin{proposition}}
\newcommand{\epr}{\end{proposition}}
\newcommand{\bt}{\begin{theorem}}
\newcommand{\et}{\end{theorem}}
\newcommand{\bc}{\begin{corollary}}
\newcommand{\ec}{\end{corollary}}
\newcommand{\bl}{\begin{lemma}}
\newcommand{\el}{\end{lemma}}
\newcommand{\bd}{\begin{definition}}
\newcommand{\ed}{\end{definition}}
\newcommand{\brs}{\begin{remarks}}
\newcommand{\ers}{\end{remarks}}

\newtheorem{theorem}{Theorem}[section]
\newtheorem{proposition}[theorem]{Proposition}
\newtheorem{corollary}[theorem]{Corollary}
\newtheorem{lemma}[theorem]{Lemma}

\theoremstyle{remark}
\newtheorem{remark}[theorem]{Remark}
\theoremstyle{definition}
\newtheorem{definition}[theorem]{Definition}

\newtheorem{example}[theorem]{Example}

\newcommand\R{\mathbf R}
\newcommand\C{\mathbf C}
\newcommand{\N}{\mathbf N}
\newcommand{\Z}{\mathbf Z}

\newcommand{\per}{\text{per}}

\newcommand{\di}{\displaystyle}

\newcommand{\SSp}{S_{{\rm p}}}
\newcommand{\spp}{s^{\rm p}}

\newcommand{\SigW}{\Sigma^{W}}

\newcommand{\OO}{{\mathbb O}}
\newcommand{\PP}{{\mathbb P}}

\newcommand{\RR}{{\mathbb R}}
\renewcommand{\SS}{{\mathbb S}}

\newcommand\bA{{\mathbf A}}
\newcommand\bB{{\mathbf B}}
\newcommand\bC{{\mathbf C}}
\newcommand\bD{{\mathbf D}}

\newcommand\bJ{{\mathbf J}}

\newcommand\bP{{\mathbf P}}

\newcommand\bT{{\mathbf T}}
\newcommand\bU{{\mathbf U}}
\newcommand\bV{{\mathbf V}}
\newcommand\bW{{\mathbf W}}
\newcommand\bX{{\mathbf X}}

\newcommand\bZ{{\mathbf Z}}

\newcommand\bfa{{\mathbf a}}
\newcommand\bfb{{\mathbf b}}

\newcommand\bfk{{\mathbf k}}

\newcommand\bfu{{\mathbf u}}
\newcommand\bfv{{\mathbf v}}

\newcommand\bfx{{\mathbf x}}
\newcommand\bfy{{\mathbf y}}

\newcommand\bfeta{{\boldsymbol \eta}}

\newcommand\bfpsi{{\boldsymbol \psi}}
\newcommand\bfxi{{\boldsymbol \xi}}

\newcommand\ubD{{\underline \bD}}

\newcommand\ubU{{\underline \bU}}

\newcommand\ubW{{\underline \bW}}

\newcommand\ubk{{\underline \bfk}}

\newcommand\uD{{\underline D}}

\newcommand\uU{{\underline U}}

\newcommand\uX{{\underline X}}

\newcommand\uc{{\underline c}}

\newcommand\uf{{\underline f}}

\newcommand\uk{{\underline k}}

\newcommand\uu{{\underline u}}

\newcommand\utheta{{\underline \theta}}

\newcommand\umu{{\underline \mu}}

\newcommand\uxi{{\underline \xi}}

\newcommand\urho{{\underline \rho}}

\newcommand\utau{{\underline \tau}}

\newcommand\uom{{\underline \omega}}

\newcommand{\falpha}{\mathfrak{a}}
\newcommand{\fbeta}{\mathfrak{b}}
\newcommand{\fgamma}{\mathfrak{c}}
\newcommand{\feta}{\mathfrak{h}}

\newcommand{\fdelta}{\mathfrak{d}}
\newcommand{\fepsilon}{\mathfrak{e}}
\newcommand{\fq}{\mathfrak{q}}

\newcommand\utv{{\underline \tv}}

\newcommand\cA{{\mathcal A}}
\newcommand\cB{{\mathcal B}}
\newcommand\cC{{\mathcal C}}

\newcommand\cF{{\mathcal F}}
\newcommand\cG{{\mathcal G}}
\newcommand\cH{{\mathcal H}}
\newcommand\cI{{\mathcal I}}
\newcommand\cJ{{\mathcal J}}

\newcommand\cL{{\mathcal L}}
\newcommand\cM{{\mathcal M}}

\newcommand\cO{{\mathcal O}}
\newcommand\cP{{\mathcal P}}

\newcommand\cU{{\mathcal U}}
\newcommand\cV{{\mathcal V}}
\newcommand\cW{{\mathcal W}}

\newcommand\cZ{{\mathcal Z}}

\newcommand\ucF{{\underline{\mathcal F}}}

\newcommand\ucU{{\underline{\mathcal U}}}
\newcommand\ucV{{\underline{\mathcal V}}}

\newcommand\tf{{\widetilde f}}

\newcommand\tq{{\widetilde q}}

\newcommand\tv{{\widetilde v}}

\newcommand\talpha{{\widetilde \alpha}}
\newcommand\tbeta{{\widetilde \beta}}

\newcommand\tkappa{{\widetilde \kappa}}

\newcommand\tchi{{\widetilde \chi}}
\newcommand\tpsi{{\widetilde \psi}}

\newcommand\tLambda{\widetilde{\Lambda}}

\newcommand\tbU{\widetilde{\bU}}
\newcommand\tbV{\widetilde{\bV}}

\usepackage{cancel}
\newcommand{\n}{\nabla}

\title{
About plane periodic waves\\of the nonlinear Schr\"odinger equations
}

\author{Corentin Audiard}
\address{Sorbonne Université, CNRS, Université de Paris, Laboratoire Jacques-Louis Lions (LJLL), F-75005 Paris, France }
\email{{\tt corentin.audiard@upmc.fr}}
\thanks{Research of C.A. was partially supported by the French ANR Project NABUCO ANR-17-CE40-0025.}

\author{L.~Miguel Rodrigues}
\address{
Univ Rennes \& IUF, CNRS, IRMAR - UMR 6625, F-35000 Rennes, France}
\email{{\tt luis-miguel.rodrigues@univ-rennes1.fr}}
\thanks{}

\begin{document}

\begin{abstract}
The present contribution contains a quite extensive theory for the stability analysis of plane periodic waves of general Schr\"odinger equations. On one hand, we put the one-dimensional theory, or in other words the stability theory for longitudinal perturbations, on a par with the one available for systems of Korteweg type, including results on co-periodic spectral instability, nonlinear co-periodic orbital stability, side-band spectral instability and linearized large-time dynamics in relation with modulation theory, and resolutions of all the involved assumptions in both the small-amplitude and large-period regimes. On the other hand, we provide extensions of the spectral part of the latter to the multi-dimensional context. Notably, we provide suitable multi-dimensional modulation formal asymptotics, validate those at the spectral level and use them to prove that waves are always spectrally unstable in both the small-amplitude and the large-period regimes.

\vspace{1em}


\vspace{1em}

\end{abstract}

\maketitle



{\bf 
\contentsline {section}{\tocsection {}{1}{\large Introduction}}{2}{section.1}
}%
\vspace{0.5em}
\contentsline {subsection}{\tocsubsection {}{\hspace{2em}1.1}{Longitudinal perturbations}}{4}{subsection.1.1}%
\contentsline {subsection}{\tocsubsection {}{\hspace{2em}1.2}{General perturbations}}{7}{subsection.1.2}%
\vspace{1em}
{\bf 
\contentsline {section}{\tocsection {}{2}{\large Structure of periodic wave profiles}}{12}{section.2}
}%
\vspace{0.5em}
\contentsline {subsection}{\tocsubsection {}{\hspace{2em}2.1}{Radius equation}}{12}{subsection.2.1}%
\contentsline {subsection}{\tocsubsection {}{\hspace{2em}2.2}{Jump map}}{14}{subsection.2.2}%
\contentsline {subsection}{\tocsubsection {}{\hspace{2em}2.3}{Madelung's transformation}}{15}{subsection.2.3}%
\contentsline {subsection}{\tocsubsection {}{\hspace{2em}2.4}{Action integral}}{18}{subsection.2.4}%
\contentsline {subsection}{\tocsubsection {}{\hspace{2em}2.5}{Asymptotic regimes}}{19}{subsection.2.5}%
\contentsline {subsection}{\tocsubsection {}{\hspace{2em}2.6}{General plane waves}}{23}{subsection.2.6}%
\vspace{1em}
{\bf 
\contentsline {section}{\tocsection {}{3}{\large Structure of the spectrum}}{24}{section.3}
}%
\vspace{0.5em}
\contentsline {subsection}{\tocsubsection {}{\hspace{2em}3.1}{The Bloch transform}}{24}{subsection.3.1}%
\contentsline {subsection}{\tocsubsection {}{\hspace{2em}3.2}{Linearizing Madelung's transformation}}{25}{subsection.3.2}%
\contentsline {subsection}{\tocsubsection {}{\hspace{2em}3.3}{The Evans function}}{26}{subsection.3.3}%
\contentsline {subsection}{\tocsubsection {}{\hspace{2em}3.4}{High-frequency analysis}}{27}{subsection.3.4}%
\contentsline {subsection}{\tocsubsection {}{\hspace{2em}3.5}{Low-frequency analysis}}{27}{subsection.3.5}%
\vspace{1em}
{\bf 
\contentsline {section}{\tocsection {}{4}{\large Longitudinal perturbations}}{30}{section.4}
}%
\vspace{0.5em}
\contentsline {subsection}{\tocsubsection {}{\hspace{2em}4.1}{Co-periodic perturbations}}{30}{subsection.4.1}%
\contentsline {subsection}{\tocsubsection {}{\hspace{2em}4.2}{Side-band perturbations}}{33}{subsection.4.2}%
\contentsline {subsection}{\tocsubsection {}{\hspace{2em}4.3}{Large-time dynamics}}{37}{subsection.4.3}%
\vspace{1em}
{\bf 
\contentsline {section}{\tocsection {}{5}{\large General perturbations}}{46}{section.5}
}%
\vspace{0.5em}
\contentsline {subsection}{\tocsubsection {}{\hspace{2em}5.1}{Geometrical optics}}{48}{subsection.5.1}%
\contentsline {subsection}{\tocsubsection {}{\hspace{2em}5.2}{Instability criteria}}{50}{subsection.5.2}%
\contentsline {subsection}{\tocsubsection {}{\hspace{2em}5.3}{Large-period regime}}{53}{subsection.5.3}%
\contentsline {subsection}{\tocsubsection {}{\hspace{2em}5.4}{Small-amplitude regime}}{58}{subsection.5.4}%
\vspace{1em}
\contentsline {section}{\tocsection {{\bf Appendix}}{{\bf A}}{\large Symmetries and conservation laws}}{61}{appendix.A}%
\vspace{0.5em}
\contentsline {section}{\tocsection {{\bf Appendix}}{{\bf B}}{\large Spectral stability of constant states}}{62}{appendix.B}%
\vspace{0.5em}
\contentsline {section}{\tocsection {{\bf Appendix}}{{\bf C}}{\large Anisotropic equations}}{64}{appendix.C}%
\vspace{0.5em}
\contentsline {section}{\tocsection {{\bf Appendix}}{{\bf D}}{\large General plane waves}}{66}{appendix.D}%
\vspace{0.5em}
\contentsline {section}{\tocsection {{\bf Appendix}}{{\bf E}}{\large Table of symbols}}{66}{appendix.E}%
\vspace{0.5em}
\contentsline {section}{\tocsection {}{}{{\bf \large References}}}{67}{section*.3}%

\vspace{2em}


\section{Introduction}\label{s:introduction}

We consider Schr\"odinger equations in the form
\be\label{e:nls}
\iD\d_tf\,=\,
-\Div_\bfx\left(\kappa(|f|^2)\,\nabla_\bfx f\right)
\,+\,\kappa'(|f|^2)\,\|\nabla_\bfx f\|^2\,f
\,+\,2\,W'(|f|^2)\,f ,
\ee
(or some anisotropic generalizations) with $W$ real-valued and $\kappa$ positive-valued, bounded away from zero, where the unknown $f$ is complex-valued, $f(t,\bfx)\in\C$, $(t,\bfx)\in\R^d$. Note that the sign assumption on $\kappa$ may be replaced with the assumption that $\kappa$ is real-valued and far from zero since one may change the sign of $\kappa$ by replacing $(f,\kappa,W)$ with $(\overline{f},-\kappa,-W)$.

Since the nonlinearity is not holomorphic in $f$, it is convenient to adopt a real point of view and introduce real 
and imaginary parts $f=a\,+\,\iD\,b$, $\bU=\bp a\\b\ep$. Multiplication by $-\iD$ is thus encoded in
\be\label{def:J}
\bJ\,=\,\bp0&1\\-1&0\ep ,
\ee
and Equation~\eqref{e:nls} takes the form
\be\label{e:ab}
\d_t\bU\,=\,
\bJ\left(
-\Div_\bfx\left(\kappa(\|\bU\|^2)\,\nabla_\bfx \bU\right)
\,+\,\kappa'(\|\bU\|^2)\,\|\nabla_\bfx \bU\|^2\,\bU
\,+\,2\,W'(\|\bU\|^2)\,\bU\right)\,.
\ee
The problem  has a Hamiltonian structure
\[
\d_t\bU=\bJ\,\delta \Ham_0[\bU]
\qquad\text{with}\qquad
\Ham_0\left[\bU\right]=\tfrac12\kappa(\|\bU\|^2)\|\nabla_\bfx \bU\|^2
+W(\|\bU\|^2) ,
\]
with $\delta$ denoting variational gradient\footnote{See the notational section at the end of the present introduction for a definition.}. Indeed our interest in \eqref{e:nls} originates in the fact that we regard the class of equations \eqref{e:nls} as the most natural class of isotropic quasilinear dispersive Hamiltonian equations including most classical semilinear Schr\"odinger equations. See \cite{Sulem-Sulem} for a comprehensive introduction to the latter. In Appendix~\ref{s:more-equations}, we also show how to treat some anisotropic versions of the equations.

Note that in the above form are embedded invariances with respect to rotations, time translations and space translations: if $f$ is a solution so is $\tf$ when
\begin{align*}
\tf(t,\bfx)&=\eD^{-\iD\phi_0}f(t,\bfx)\,,&\phi_0\in\R\,,\ \text{rotational invariance}\,,\\
\tf(t,\bfx)&=f(t-t_0,\bfx)\,,&t_0\in\R\,,\ \text{time translation invariance}\,,\\
\tf(t,\bfx)&=f(t,\bfx-\bfx_0)\,,&\bfx_0\in\R^d\,,\ \text{space translation invariance}\,.\\
\end{align*}
Actually rotations and time and space translations leave the Hamiltonian $\Ham_0$ essentially unchanged, in a sense made explicit in Appendix~\ref{s:Noether}. Thus, through a suitable version of Noether's principle, they are associated with conservation laws, respectively on mass $\Mass[\bU]=\tfrac12\|\bU\|^2$, Hamiltonian $\Ham_0[\bU]$ and momentum $\bImpulse[\bU]=(\Impulse_j[\bU])_j$, with $\Impulse_j[\bU]=\tfrac12\bJ\bU\cdot\d_j \bU$, $j=1,\cdots,d$. Namely invariance by rotation implies that any solution $\bU$ to \eqref{e:ab} satisfies mass conservation law
\be
\label{e:mass}
\d_t\Mass(\bU)=
\sum_j\d_j\left(\bJ\delta\Mass[\bU]\cdot\nabla_{\bU_{x_j}}\Ham_0[\bU]\right)\,.
\ee
Likewise invariance by time translation implies that \eqref{e:ab} contains the conservation law
\be\label{e:Ham}
\d_t\Ham_0[\bU]=\sum_j\d_j\left(\nabla_{\bU_{x_j}}\Ham_0[\bU]\cdot\bJ\delta\Ham_0[\bU]\right)\,.
\ee
At last, invariance by spatial translation implies that from \eqref{e:ab} stems 
\be
\label{e:momentum}
\d_t\left(\bImpulse[\bU]\right)\,=\,
\nabla_\bfx\left(\frac12\bJ\bU\cdot\bJ\delta\Ham_0[\bU]-\Ham_0[\bU]\right)
+\sum_{\ell}\d_\ell(\bJ\delta\,\bImpulse[\bU]\cdot\nabla_{\bU_{x_\ell}}\Ham_0[\bU])\,.
\ee
The reader is referred to Appendix~\ref{s:Noether} for a derivation of the latter.

We are interested in the analysis of the dynamics near plane periodic uniformly traveling waves of \eqref{e:nls}. Let us first recall that a (uniformly traveling) wave is a solution whose time evolution occurs through the action of symmetries. We say that the wave is a plane wave when in a suitable frame it is constant in all but one direction and that it is periodic if it is periodic up to symmetries. Given the foregoing set of symmetries, after choosing for sakes of concreteness the direction of propagation as $\beD_1$ and normalizing period to be $1$ through the introduction of wavenumbers, we are interested in solutions to \eqref{e:nls} of the form
\[
f(t,\bfx)
\,=\,\eD^{-\iD\left(\ukp\,(x-\ucx\,t)+\uomp\,t\right)}\uf(\ukx\,(x-\ucx\,t))
\,=\,\eD^{-\iD(\ukp\,x+(\uomp-\ukp\,\ucx)\,t)}\uf(\ukx\,x+\uomx\,t)\,,
\]
with profile $\uf$ $1$-periodic, wavenumbers $(\ukp,\ukx)\in\R^2$, $\ukx>0$, time-frequencies $(\uomp,\uomx)\in\R^2$, spatial speed $\ucx\in\R$, where
\[
\bfx=(x,\bfy)\,\qquad
\uomx=-\ukx\,\ucx\,.
\]
In other terms we consider solutions to \eqref{e:ab} in the form
\be \label{def:wave}
\bU(t,\bfx)\,=\,
\eD^{\left(\ukp\,(x-\ucx\,t)+\uomp\,t\right)\bJ}\,
\ucU(\ukx\,(x-\ucx\,t))\,,
\ee
with $\ucU$ $1$-periodic (and non-constant). More general periodic plane waves are also considered in Appendix~\ref{s:more-waves}. Beyond references to results involved in our analysis given along the text and comparison to the literature provided near each main statement, in order to place our contribution in a bigger picture, we refer the reader to \cite{KapitulaPromislow_book} for general background on nonlinear wave dynamics and to \cite{Angulo-Pava,Haragus-Kapitula-Hamiltonian,DeBievre-RotaNodari} for material more specific to Hamiltonian systems.

To set the frame for linearization, we observe that going to a frame adapted to the background wave in \eqref{def:wave} by
\[
\bU(t,\bfx)
\,=\,
\eD^{\left(\ukp\,(x-\ucx\,t)+\uomp\,t\right)\bJ}\,\bV(t,\ukx\,(x-\ucx\,t),\bfy)\,,
\]
changes \eqref{e:ab} into
\begin{align}\label{e:moving-nls}
\d_t\bV&\,=\,\bJ\delta\Ham[\bV]\,,\\
\Ham[\bV]&\,:=\,\Ham_0(\bV,(\ukx\d_x+\ukp\bJ)\bV,\nabla_\bfy\bV)-\uomp\Mass[\bV]
+\ucx\Impulse_1(\bV,(\ukx\d_x+\ukp\bJ)\bV)\nonumber\\
&\,\,=\,
\Ham_0(\bV,(\ukx\d_x+\ukp\bJ)\bV,\nabla_\bfy\bV)-\left(\uomp-\ukp\,\ucx\right)\Mass[\bV]-\uomx\Impulse_1[\bV]\,,\nonumber
\end{align}
and that $(t,x,\bfy)\mapsto\ucU(x)$ is a stationary solution to \eqref{e:moving-nls}. Direct linearization of \eqref{e:moving-nls} near this solution provides the linear equation $\partial_t\bV=\cL\,\bV$ with $\cL$ defined by
\be\label{def:L}
\cL\bV\,=\,\bJ\Hess\Ham[\ucU](\bV)
\ee
where $\Hess$ denotes the variational Hessian, that is, $\Hess=L\delta$ with $L$ denoting linearization. Incidentally we point out that the natural splitting
\[
\Ham_0=\Ham_0^x+\Ham^\bfy\,,\qquad
\Ham^\bfy\left[\bU\right]=\tfrac12\kappa(\|\bU\|^2)\|\nabla_\bfy\bU\|^2\,,
\]
may be followed all the way through frame change and linearization
\begin{align*}
\Ham&=\Ham^x+\Ham^\bfy\,,\\
\cL&\,=\,\bJ\Hess\Ham^x[\ucU]+\bJ\Hess\Ham^\bfy[\ucU]=:\cL^x+\cL^\bfy\,,
\end{align*}
with $\cL^\bfy\,=\,-\kappa(\|\ucU\|^2)\bJ\,\Delta_\bfy$.

As made explicit in Section~\ref{s:Bloch} at the spectral and linear level, to make the most of the spatial structure of periodic plane waves, it is convenient to introduce a suitable Bloch-Fourier integral transform. As a result one may analyze the action of $\cL$ defined on $L^2(\R)$ through\footnote{As, by using Fourier transforms on constant-coefficient operators one reduces their action on functions over the whole space to finite-dimensonial operators parametrized by Fourier frequencies.} the actions of $\cL_{\xi,\bfeta}$ defined on $L^2((0,1))$ with periodic boundary conditions, where $(\xi,\bfeta)\in [-\pi,\pi]\times\R^{d-1}$, $\xi$ being a longitudinal Floquet exponent, $\bfeta$ a transverse Fourier frequency. The operator $\cL_{\xi,\bfeta}$ encodes the action of $\bJ\Hess\Ham[\ucU]$ on functions of the form
\[
\bfx=(x,\bfy)\mapsto \eD^{\iD\xi x+\iD\bfeta\cdot\bfy}\,\bW(x)\,,\qquad
\bW(\cdot+1)\,=\,\bW\,,
\]
through 
\[
\bJ\Hess\Ham[\ucU]\left((x,\bfy)\mapsto \eD^{\iD\xi x+\iD\bfeta\cdot\bfy}\,\bW(x)\right)(\bfx)\,=\,\eD^{\iD\xi x+\iD\bfeta\cdot\bfy}\,(\cL_{\xi,\bfeta}\bW)(x)\,.
\]
In particular, the spectrum of $\cL$ coincides with the union over $(\xi,\bfeta)$ of the spectra of $\cL_{\xi,\bfeta}$. In turn, as recalled in Section~\ref{s:Evans-def}, generalizing the analysis of Gardner \cite{Gardner-structure-periodic}, the spectrum of each $\cL_{\xi,\bfeta}$ may be conveniently analyzed with the help of an Evans' function $D_\xi(\cdot,\bfeta)$, an analytic function whose zeroes agree in location and algebraic multiplicity with the spectrum of $\cL_{\xi,\bfeta}$. A large part of our spectral analysis hinges on the derivation of an expansion of $D_\xi(\lambda,\bfeta)$ when $(\lambda,\xi,\bfeta)$ is small (Theorem~\ref{th:low-freq}).

As derived in Section~\ref{s:profile}, families of plane periodic profiles in a fixed direction --- here taken to be $\beD_1$ --- form four-dimensional manifolds when identified up to rotational and spatial translations, parametrized by $(\mux,\cx,\omp,\mup)$ where $(\mux,\mup)$ are constants of integration of profile equations associated with conservation laws \eqref{e:mass} and \eqref{e:momentum} (or more precisely its first component since we consider waves propagating along $\beD_1$). The averages along wave profiles of quantities of interest are expressed in terms of an action integral $\Theta(\mux,\cx,\omp,\mup)$ and its derivatives. This action integral plays a prominent role in our analysis. A significant part of our analysis indeed aims at reducing properties of operators acting on infinite-dimensional spaces to properties of this finite-dimensional function.

After these preliminary observations, we give here a brief account of each of our main results and provide only later in the text more specialized comments around precise statements. Our main achievements are essentially two-fold. On one hand, we provide counterparts to the main upshots of \cite{Benzoni-Noble-Rodrigues-note,Benzoni-Noble-Rodrigues,Benzoni-Mietka-Rodrigues,BMR2-I,BMR2-II,R_linKdV} --- derived for one-dimensional Hamiltonian equations of Korteweg type --- for one-dimensional Hamiltonian equations of Schr\"odinger type. On the other hand we extend parts of this analysis to the present multi-dimensional framework.

\subsection{Longitudinal perturbations}

To describe the former, we temporarily restrict to longitudinal perturbations or somewhat equivalently restrict to the case $d=1$. At the linear level, this amounts to setting $\bfeta=0$.

The first set of results we prove concerns perturbations that in the above adapted moving frame are spatially periodic with the same period as the background waves, so-called co-periodic perturbations. At the linear level, this amounts to restricting to $(\xi,\bfeta)=(0,0)$. In Theorem~\ref{th:co-periodic}, as in \cite{Benzoni-Mietka-Rodrigues}, we prove that a wave of parameters $(\umux,\ucx,\uomp,\umup)$ such that $\Hess(\Theta)(\umux,\ucx,\uomp,\umup)$ is invertible is
\begin{enumerate}
\item $H^1$ (conditionnally) nonlinearly (orbitally) stable under co-periodic longitudinal perturbations if $\Hess(\Theta)(\umux,\ucx,\uomp,\umup)$ has negative signature $2$ and $\d_{\mux}^2\Theta(\umux,\ucx,\uomp,\umup)\neq0$;
\item spectrally (exponentially) unstable under co-periodic longitudinal perturbations if this negative signature is either $1$ or $3$, or equivalently if $\Hess(\Theta)(\umux,\ucx,\uomp,\umup)$ has negative determinant.
\end{enumerate}
The main upshot here is that instead of the rather long list of assumptions that would be required by directly applying the abstract general theory \cite{GSS-II,DeBievre-RotaNodari}, assumptions are both simple and expressed in terms of the finite-dimensional $\Theta$.

Then, as in \cite{BMR2-I}, we elucidate these criteria in two limits of interest, the solitary-wave limit when the spatial period tends to infinity and the harmonic limit when the amplitude of the wave tends to zero. To describe the solitary-wave regime, let us point out that solitary wave profiles under consideration are naturally parametrized by $(\cx,\rho,\kp)$ where $\rho>0$ is the limiting value at spatial infinities of its mass and that families of solitary waves also come with an action integral $\Theta_{(s)}(\cx,\rho,\kp)$, known as the Boussinesq momentum of stability \cite{Boussinesq,Benjamin72,Benjamin} and associated for Schr\"odinger-like equations with the famous Vakhitov-Kolokolov slope condition \cite{VK}. The reader may consult  \cite{Zhidkov,Lin,DeBievre-RotaNodari} as entering gates in the quite extensive mathematical literature on the latter. In Theorem~\ref{th:coperiodic_asymptotics}, we prove that
\begin{enumerate}
\item in non-degenerate small-amplitude regimes, waves are nonlinearly stable to co-periodic perturbations;
\item in the large-period regime near a solitary wave of parameters $(\ucx^{(0)},\urho^{(0)},\ukp^{(0)})$, co-periodic spectral instability occurs when $\d_{\cx}^2\Theta_{(s)}(\ucx^{(0)},\urho^{(0)},\ukp^{(0)})<0$ whereas co-periodic nonlinear orbital stability holds when $\d_{\cx}^2\Theta_{(s)}(\ucx^{(0)},\urho^{(0)},\ukp^{(0)})>0$.
\end{enumerate}
Both results appear to be new in this context. Note in particular that our small-amplitude regime is disjoint from the cubic semilinear one considered in \cite{Gallay-Haragus-spectral-small} since there the constant asymptotic mass is taken to be zero. Yet, in the large-period regime, the spectral instability result could also be partly recovered by combining a spectral instability result for solitary waves available in the above-mentioned literature for some semilinear equations, with a non-trivial spectral perturbation argument from \cite{Gardner-large-period,Sandstede-Scheel-large-period,Yang-Zumbrun}.  

The rest of our results on longitudinal perturbations concerns side-band longitudinal perturbations, that is, perturbations corresponding to $(\xi,\bfeta)=(\xi,0)$ with $\xi$ small (but non-zero), and geometrical optics \emph{\`a la} Whitham \cite{Whitham}. 

The latter is derived by inserting in \eqref{e:ab} the two-phases slow/fastly-oscillatory \emph{ansatz}  
\be\label{e:ansatz-intro}
\bU^{(\eps)}(t,x)
\,=\,\eD^{\frac{1}{\eps}\vphip^{(\eps)}(\eps\,t,\eps\,x)\,\bJ}\cU^{(\eps)}\left(\eps\,t,\eps\,x;
\frac{\vphix^{(\eps)}(\eps\,t,\eps\,x)}{\eps}\right)
\ee 
with, for any $(T,X)$, $\zeta\mapsto\cU^{(\eps)}(T,X;\zeta)$ periodic of period $1$ and, as $\eps\to0$,
\begin{align*}
\cU^{(\eps)}(T,X;\zeta)&=
\cU_{0}(T,X;\zeta)+\eps\,\cU_{1}(T,X;\zeta)+o(\eps)\,,\\
\vphip^{(\eps)}(T,X)&=
(\vphip)_{0}(T,X)+\eps\,(\vphip)_{1}(T,X)+o(\eps)\,,\\
\vphix^{(\eps)}(T,X)&=
(\vphix)_{0}(T,X)+\eps\,(\vphix)_{1}(T,X)+o(\eps)\,.
\end{align*}
Arguing heuristically and identifying orders of $\eps$ as detailed in Section~\ref{s:side-band}, one obtains that the foregoing \emph{ansatz} may describe behavior of solutions to \eqref{e:ab} provided that the leading profile $\cU_{0}$ stems from a slow modulation of wave parameters
\be\label{e:slow-modulation}
\cU_{0}(T,X;\zeta)\,=\, \cU^{(\mux,\cx,\omp,\mup)(T,X)}(\zeta)
\ee
where here $\cU^{(\mux,\cx,\omp,\mup)}$ denotes a wave profile of parameters $(\mux,\cx,\omp,\mup)$, with local wavenumbers $(\kp,\kx)=(\d_X(\vphip)_{0},\d_X(\vphix)_{0})$ and the slow evolution of local parameters obeys 
\be\label{e:W-intro}
\kx\bA_0\Hess\Theta\,\left(\d_T+\cx\d_X\right)
\begin{pmatrix}
\mux\\\cx\\\omp\\\mup
\end{pmatrix}
\,=\,\bB_0\,\d_X
\begin{pmatrix}
\mux\\\cx\\\omp\\\mup
\end{pmatrix}
\ee
where
\begin{align*}
\bA_0&:=
\begin{pmatrix}
1&0&0&0\\
0&1&0&0\\
0&0&-1&0\\
0&0&0&-1
\end{pmatrix}\,,&
\bB_0&:=\,
\begin{pmatrix}
0&1&0&0\\
1&0&0&0\\
0&0&0&1\\
0&0&1&0
\end{pmatrix}\,.
\end{align*}
Let us point out that actually, in the derivation sketched above, System~\eqref{e:W-intro} is firstly obtained in the equivalent form
\be\label{e:W-intro-bis} 
\left\{
\begin{array}{rcl}
\d_T\kx&=&\d_X\omx\\
\d_T\impulse
&=&\d_X\left(\mux-\cx\impulse\right)\\
\d_T\mass
&=&\d_X\left(\mup-\cx\mass\right)\\
\d_T\kp&=&\d_X\left(\omp-\cx\,\kp\right)
\end{array}
\right.
\ee
with $\mass$ and $\impulse$ denoting averages over one period of respectively $\Mass(\cU)$ and \mbox{$\Impulse_1(\cU,(\kp\,\bJ+\kx\d_\zeta)\cU)$} with $\cU=\cU^{(\mux,\cx,\omp,\mup)}$. Note that two of the equations of \eqref{e:W-intro-bis} are so-called conservations of waves, whereas the two others arise as averaged equations. For a thorougher introduction to modulation systems such as \eqref{e:W-intro-bis} in the context of Hamiltonian systems, we refer the reader to the introduction of \cite{BMR2-I} and references therein.

Our second set of results concerns spectral validation of the foregoing formal arguments in the slow/side-band regime. More explicitly, as in \cite{Benzoni-Noble-Rodrigues}, we obtain in the specialization of Theorem~\ref{th:low-freq} to $\bfeta=0$ that 
\begin{align}\label{e:spec-W-intro}
D_\xi(\lambda,0)
&\stackrel{(\lambda,\xi)\to(0,0)}{=}
\det\left(\lambda\,\bA_0\,\Hess\,\Theta(\umux,\ucx,\uomp,\umup)-\iD\xi\,\bB_0\right)
+\cO\left((|\lambda|^4+|\xi|^4)\,|\lambda|\right)\,,
\end{align}
for a wave of parameter $(\umux,\ucx,\uomp,\umup)$. This connects slow/side-band Bloch spectral dispersion relation for the wave profile $\ucU$ as a stationary solution to \eqref{e:moving-nls} with slow/slow Fourier dispersion relation for $(\umux,\ucx,\uomp,\umup)$ as a solution to \eqref{e:W-intro}. Among direct consequences of \eqref{e:spec-W-intro} derived in Corollary~\ref{c:W-cor}, we point out that this implies that if $\Hess(\Theta)(\umux,\ucx,\uomp,\umup)$ is invertible and \eqref{e:W-intro} fails to be weakly hyperbolic at $(\umux,\ucx,\uomp,\umup)$, then the wave is spectrally exponentially unstable to side-band perturbations. Afterwards, as in \cite{BMR2-II}, in Theorem~\ref{th:side-band_asymptotics} we combine asymptotics for \eqref{e:W-intro} with the foregoing instability criterion to derive that waves are spectrally exponentially unstable to longitudinal side-band perturbations
\begin{enumerate}
\item in non-degenerate small-amplitude regimes near an harmonic wavetrain such that $\delta_{hyp}<0$ or $\delta_{BF}<0$, with indices $(\delta_{hyp},\delta_{BF})$ defined explicitly in \eqref{def:sideindex} and \eqref{def:sideindex_II} ;
\item in the large-period regime near a solitary wave of parameters $(\ucx^{(0)},\urho^{(0)},\ukp^{(0)})$ such that $\d_{\cx}^2\Theta_{(s)}(\ucx^{(0)},\urho^{(0)},\ukp^{(0)})<0$.
\end{enumerate}
Again, these results are new in this context, except for the corollary about weak hyperbolicity that overlaps with the recent preprint \cite{Clarke-Marangell} --- based on the recent \cite{LBJM} ---, appeared during the preparation of the present contribution. Note however that our proof of the corollary is different and our assumptions are considerably weaker.

Our third set of results concerning longitudinal perturbations shows that for spectrally stable waves in a suitable dispersive sense, by including higher-order corrections in \eqref{e:W-intro} one obtains a version of \eqref{e:slow-modulation}-\eqref{e:W-intro} that captures at any arbitrary order the large-time asymptotics for the slow/side-band part of the linearized dynamics. Besides the oscillatory-integral analysis directly borrowed from \cite{R_linKdV}, this hinges on a spectral validation of the formal asymptotics --- obtained in Theorem~\ref{th:W-Bloch} --- as predictors for expansions of spectral projectors (and not only of spectral curves) in the slow/side-band regime. The identified decay is inherently of dispersive type and we refer the curious reader to \cite{Linares-Ponce,Erdogan-Tzirakis} 
for comparisons with the well-known theory for constant-coefficient operators. Let us stress that deriving global-in-time dispersive estimates for non-constant non-normal operators is a considerably harder task and that the analysis in \cite{R_linKdV} has provided the first-ever dispersive estimates for the linearized dynamics about a periodic wave. We also point out that a large-time dynamical validation of modulation systems for general data --- as opposed to a spectral validation or a validation for well-prepared data --- requires the identification of effective initial data for modulation systems, a highly non trivial task that cannot be guessed from the formal arguments sketched above.

At this stage the reader could wonder how in a not-so-large number of pages could be obtained Schr\"odinger-like counterparts to Korteweg-like results originally requiring a quite massive body of literature \cite{Benzoni-Noble-Rodrigues-note,Benzoni-Noble-Rodrigues,Benzoni-Mietka-Rodrigues,BMR2-I,BMR2-II,R_linKdV}. There are at least two phenomena at work. On one hand, we have actually left without counterparts a significant part of \cite{BMR2-II,R_linKdV}. Results in \cite{BMR2-II} are mainly motivated by the study of dispersive shocks and the few stability results adapted here from \cite{BMR2-II} were obtained there almost in passing. The analysis in \cite{R_linKdV} studies the full linearized dynamics for the Korteweg-de Vries equation. Yet, the underlying arguments being technically demanding, we have chosen to adapt here only the part of the analysis directly related to modulation behavior, for the sake of both consistency and brevity. On the other hand, some of the results proved here are actually deduced from the results derived for some Korteweg-like systems rather than proved from scratch. 

The key to these deductions is a suitable study of Madelung's transformation \cite{Madelung}. As we develop in Section~\ref{s:EK}, even at the level of generality considered here, Madelung's transformation provides a convenient hydrodynamic formulation of \eqref{e:nls} of Korteweg type. A solution $\bU$ to \eqref{e:ab} is related to a solution $(\rho,\bfv)$, with curl-free velocity $\bfv$, of a Euler--Korteweg system through
\[
\bU=\sqrt{2\,\rho}\,\eD^{\theta\,\bJ}\bp 1\\0\ep\,,\qquad \bfv=\nabla_\bfx \theta\,.
\]
We refer the reader to \cite{Madelung-survey} for some background on the transformation and its mathematical use. Let us stress that the transformation dramatically changes the geometric structure of the equations, in both its group of symmetries and its conservation laws. A basic observation that makes the Madelung's transformation particularly efficient here is that non-constant periodic wave profiles stay away from zero. Consistently, the asymptotic regimes we consider also lie in the far-from-zero zone. Our co-periodic nonlinear orbital stability result is in particular proved here by studying in Lemma~\ref{diffeosobolev} correspondences through the Madelung's transformation. Even more efficiently, identification of respective action integrals also reduces the asymptotic expansions of $\Hess \Theta$ required here to those already obtained in \cite{BMR2-I,BMR2-II}. For the sake of completeness, in Section~\ref{s:linearized-EK} we also carry out a detailed study of spectral correspondences. Yet those fail to fully elucidate spectral behavior near $(\lambda,\xi,\bfeta)=(0,0,0)$ and, thus, they play no role in our spectral and linear analyses.

\subsection{General perturbations}

In the second part of our analysis we extend to genuinely multi-dimensional perturbations the spectral results of the longitudinal part. 

To begin with, we provide an instability criterion for perturbations that are longitudinally co-periodic, that is, that corresponds to $\xi=0$. The corresponding result, Corollary~\ref{c:transverse}-(1), is made somewhat more explicit in Lemma~\ref{l:xi=0}. Yet we do not investigate the corresponding asymptotics because in the multi-dimensional context we are more interested in determining whether waves may be stable against any perturbation and the present co-periodic instability criterion turns out to be weaker than the slow/side-band one contained in Corollary~\ref{c:transverse}-(2) and that we describe now. 

The second, and main, set of results of this second part focuses on slow/side-band perturbations, corresponding to the regime $(\lambda,\xi,\bfeta)$ small. In the latter regime, generalizing the longitudinal analysis, we derive an instability criterion, interpret it in terms of formal geometrical optics and elucidate it in both the small-amplitude and large-period asymptotics.

Concerning geometrical optics, a key observation is that even if one is merely interested on the stability of waves in the specific form \eqref{def:wave}, the relevant modulation theory involves more general waves in the form
\be \label{def:wave-general}
\bU(t,\bfx)\,=\,
\eD^{\left(\ubkp\cdot(\bfx-\ucx\,\uex\,t)+\uomp\,t\right)\bJ}\,
\ucU(\ubkx\cdot(\bfx-\ucx\,\uex\,t))\,,
\ee
with $\ucU$ $1$-periodic, $\ubkx$ non-zero of unitary direction $\uex$. The main departure in \eqref{def:wave-general} from \eqref{def:wave} is that $\ubkx$ and $\ubkp$ are non longer assumed to be colinear. To stress comparisons with \eqref{def:wave}, let us decompose $(\ubkx,\ubkp)$ as 
\begin{align*}
\ubkx&=\ukx\,\uex\,,&
\ubkp&=\ukp\,\uex+\utkp\,, 
\end{align*}
with $\utkp$ orthogonal to $\uex$. In Section~\ref{s:profile-general}, we show that this more general set of plane waves may be conveniently parametrized by $(\mux,\cx,\omp,\mup,\ex,\tkp)$, with $(\ex,\tkp)$ varying in the $2(d-1)$-dimensional manifold of vectors such that $\ex$ is unitary and $\tkp$ is orthogonal to $\ex$.

With this in hands, adding possible slow dependence on $\bfy$ in \eqref{e:ansatz-intro} through
\be\label{e:ansatz-more-intro}
\bU^{(\eps)}(t,\bfx)
\,=\,\eD^{\frac{1}{\eps}\vphip^{(\eps)}(\eps\,t,\eps\,\bfx)\,\bJ}\cU^{(\eps)}\left(\eps\,t,\eps\,\bfx;
\frac{\vphix^{(\eps)}(\eps\,t,\eps\,\bfx)}{\eps}\right)
\ee 
and arguing as before leads to the modulation behavior 
\be\label{e:slow-modulation-general}
\cU_{0}(T,\bX;\zeta)\,=\, \cU^{(\mux,\cx,\omp,\mup,\ex,\tkp)(T,\bX)}(\zeta)
\ee
with local wavevectors $(\bkp,\bkx)=(\nabla_\bX(\vphip)_{0},\nabla_\bX(\vphix)_{0})$ and the slow evolution of local parameters obeys 
\be\label{e:W-intro-more} 
\left\{
\begin{array}{rcl}
\d_T\bkx&=&\nabla_\bX\omx\\
\d_T\bimpulse
&=&\nabla_\bX\left(\mux-\cx\impulse+\frac12\,\tau_0\,\|\tkp\|^2\right)\\
&&+\Div_\bX\left(\tau_1\,\tkp\otimes\tkp
+\tau_2\,(\tkp\otimes\ex+\ex\otimes\tkp)
+\tau_3\,\left(\ex\otimes\ex-\I_d\right)\right)\\
\d_T\mass
&=&\Div_\bX\left((\mup-\cx\mass)\,\ex+\,\tau_1\,\tkp\right)\\
\d_T\bkp&=&\nabla_\bX\left(\omp-\cx\,\kp\right)
\end{array}
\right.
\ee
with extra constraints (propagated by the time-evolution) that $\bkx$ and $\bkp$ are curl-free. In  System~\eqref{e:W-intro-more}, $\bfa\otimes \bfb$ denotes the matrix of $(j,\ell)$-coordinate $\bfa_\ell\,\bfb_j$, $\Div_\bX$ acts on matrix-valued maps row-wise and $\bimpulse$, $\tau_0$, $\tau_1$, $\tau_2$ and $\tau_3$ denote the averages over one period of respectively 
\begin{align*}
\bImpulse(\cU,(\kp\,\bJ+\kx\d_\zeta)\cU),&& 
\kappa'(\|\cU\|^2)\,\|\cU\|^2,&& 
\kappa(\|\cU\|^2)\,\|\cU\|^2,\\
\kappa(\|\cU\|^2)\,\bJ\cU\cdot(\kp\,\bJ+\kx\d_\zeta)\cU&&
\textrm{and}&&
\kappa(\|\cU\|^2)\,\|(\kp\,\bJ+\kx\d_\zeta)\cU\|^2\,,
\end{align*}
with $\cU=\cU^{(\mux,\cx,\omp,\mup,\ex,\tkp)}$. Linearizing System~\eqref{e:W-intro-more} about the constant $(\umux,\ucx,\uomp,\umup,\beD_1,0)$ yields after a few manipulations 
\be\label{e:W-lin-intro}
\left\{
\begin{array}{rcl}
\ukx\bA_0\Hess\Theta\,\left(\d_T+\ucx\d_X\right)
\begin{pmatrix}
\mux\\\cx\\\omp\\\mup
\end{pmatrix}
&\,=\,&\bB_0\,\d_X
\begin{pmatrix}
\mux\\\cx\\\omp\\\mup
\end{pmatrix}
+\bp
0&0\\
\utau_3&\utau_2\\
\utau_2&\utau_1\\
0&0
\ep\bp \Div_\bX(\ex)\\\Div_\bX(\tkp) \ep\\
\left(\d_T+\ucx\d_X\right)\ex
&\,=\,&-\left(\nabla_\bX-\beD_1\,\d_X\right)\cx\\
\left(\d_T+\ucx\d_X\right)\tkp
&\,=\,&\left(\nabla_\bX-\beD_1\,\d_X\right)\omp
\end{array}\right.
\ee
with extra constraints that $\tkp$ and $\ex$ are orthogonal to $\uex=\beD_1$ and that $\kx\uex+\ukx\ex$ and $\kp\uex+\ukp\ex+\tkp$ are curl-free, where $(\kx,\kp)$ are deviations given explicitly as 
\begin{align*}
\kx&=-\ukx^2\dD\,(\d_{\mux}\Theta)(\mux,\cx,\omp,\mup)\,,&
\kp&=\frac{\kx}{\ukx}\,\ukp
-\ukx\dD\,(\d_{\mup}\Theta)(\mux,\cx,\omp,\mup)\,,&
\end{align*}
where total derivatives are taken with respect to $(\mux,\cx,\omp,\mup)$ and evaluation is at $(\umux,\ucx,\uomp,\umup,\beD_1,0)$. In System~\ref{e:W-lin-intro}, likewise $\Hess\Theta=\Hess_{(\mux,\cx,\omp,\mup)}\Theta$ is evaluated at $(\umux,\ucx,\uomp,\umup,\beD_1,0)$, and $\bA_0$ and $\bB_0$ are as in System~\eqref{e:W-intro}.

As made explicit in Section~\ref{s:WKB}, our Theorem~\ref{th:low-freq} provides a spectral validation of \eqref{e:W-intro-more} in the form 
\begin{align*}
&\lambda^{2(d-1)}\times D_\xi(\lambda,\bfeta)\stackrel{(\lambda,\xi,\bfeta)\to(0,0,0)}{=}\\
&\det\left(\lambda\,\bp \I_{2(d-1)}&0\\
0&\bA_0\,\Hess\,\Theta\ep
-\iD\xi\,
\bp 0&0\\
0&\bB_0\ep
+\left(\begin{array}{cc|cccc}
0&0&0&-\iD\bfeta&0&0\\
0&0&0&0&\iD\bfeta&0\\
\hline\\[-1.25em]
0&0&0&0&0&0\\
\frac{\utau_3}{\ukx}\,\iD\transp{\bfeta}&\frac{\utau_2}{\ukx}\,\iD\transp{\bfeta}
&0&0&0&0\\[0.5em]
\frac{\utau_2}{\ukx}\,\iD\transp{\bfeta}&\frac{\utau_1}{\ukx}\,\iD\transp{\bfeta}
&0&0&0&0\\
0&0&0&0&0&0
\end{array}\right)\right)\\
&+\cO\left(|\lambda|^{2(d-1)}\,(|\lambda|+|\xi|+\|\bfeta\|)^5\right)\,,
\end{align*}
or equivalently in the form
\begin{align}\label{e:spec-W-intro-more}
D_\xi(\lambda,\bfeta)\stackrel{(\lambda,\xi,\bfeta)\to(0,0,0)}{=}
&\det\left(\lambda\,\bA_0\,\Hess\,\Theta-\iD\xi\,\bB_0
+\frac{\|\bfeta\|^2}{\lambda}\,\bC_0\right)
+\cO\left((|\lambda|+|\xi|+\|\bfeta\|)^5\right)\,,
\end{align}
with
\begin{align*}
\bC_0
&:=\begin{pmatrix}0&0&0&0\\
0&-\sigma_3
&\sigma_2&0\\
0&-\sigma_2
&\sigma_1&0\\
0&0&0&0
\end{pmatrix}\,,& 
\sigma_j=\frac{\utau_j}{\ukx}\,,&\quad j\in\{1,2,3\}\,.
\end{align*}
In the foregoing, again $\Hess\Theta=\Hess_{(\mux,\cx,\omp,\mup)}\Theta(\umux,\ucx,\uomp,\umup,\beD_1,0)$. Note that, consistently with the equality, the structure of $\bB_0$ and $\bC_0$ implies that the apparent singularity in $\lambda$ of the left-hand side of \eqref{e:spec-W-intro-more} is indeed spurious, each factor $\|\bfeta\|^2/\lambda$ being necessarily paired with a factor $\lambda$ in the expansion of the determinant. We stress that we are not aware of any other rigorous spectral validation of a multi-dimensional modulation system, even for other classes of equations.

It follows directly from \eqref{e:spec-W-intro-more} that if \eqref{e:W-lin-intro} fails to be weakly hyperbolic at $(\umux,\ucx,\uomp,\umup,\beD_1,0)$ then the corresponding wave is spectrally exponentially unstable. In Section~\ref{s:criteria}, besides this most general instability criterion, we provide two instability criteria, more specific but easier to check, corresponding to the breaking of multiple roots near $\bfeta=0$ (Proposition~\ref{pr:splitting-eta0}) and near $\xi=0$ (Proposition~\ref{pr:splitting-xi0}) respectively. 

Afterwards we turn to the elucidation of the full instability criterion in the asymptotic regimes already studied in the longitudinal part. Our striking conclusion is that, when $d\geq2$, in non degenerate cases plane waves of the form \eqref{def:wave} are spectrally exponentially unstable in both the small-amplitude (Theorem~\ref{th:harmonic_asymptotics}) and the large-period (Theorem~\ref{th:homoclinic_asymptotics}) regimes. More explicitly we prove that such waves are spectrally exponentially unstable to slow/side-band perturbations
\begin{enumerate}
\item in non-degenerate small-amplitude regimes near an harmonic wavetrain such that $\delta_{hyp}\neq0$ and $\delta_{BF}\neq0$, with indices defined explicitly in \eqref{def:sideindex} and \eqref{def:sideindex_II} ;
\item in the large-period regime near a solitary wave of parameters $(\ucx^{(0)},\urho^{(0)},\ukp^{(0)})$ such that $\d_{\cx}^2\Theta_{(s)}(\ucx^{(0)},\urho^{(0)},\ukp^{(0)})\neq0$.
\end{enumerate}

Let us stress that to obtain the latter we derive various instability scenarios --- all hinging on expansion \eqref{e:spec-W-intro-more} thus occurring in the region $(\lambda,\xi,\bfeta)$ small --- corresponding to different instability criteria. The point is that the union of these criteria covers all possibilities. In particular, in the harmonic limit, the argument requires the full strength of the joint expansion in $(\lambda,\xi,\bfeta)$ and it is relatively elementary --- see Appendix~\ref{s:constant} --- to check that the instability is non trivial in the sense that it occurs even in cases when the limiting constant states is spectrally stable. We also stress that both asymptotic results are derived by extending to the multidimensional context some of the finest properties of longitudinal modulated systems proved in \cite{BMR2-II} from asymptotic expansions of $\Hess\Theta$ obtained in \cite{BMR2-I}.

All the results about general perturbations are new, including this form of the formal derivation of a modulation system. The only small overlap we are aware of is with \cite{LBJM} appeared during the preparation of the present contribution and studying to leading order the spectrum of $\cL_{(0,\bfeta)}$ near $\lambda=0$, when $\bfeta$ is small. Even for this partial result, our proof is different and our assumptions are considerably weaker. Let us also stress that \cite{LBJM} discusses neither modulation systems nor asymptotic regimes. At last, we point out that the operator $\cL_{(0,\bfeta)}$ depends on $\bfeta$ only through the scalar parameter $\|\bfeta\|^2$ so that the problem studied in \cite{LBJM} fits the frame of spectral analysis of analytic one-parameter perturbations, a subpart of general spectral perturbation theory that is considerably more regular and simpler, even compared to two-parameters perturbations as we consider here. Concerning the latter, we refer the reader to \cite{Kato,Davies} for general background on spectral theory. Besides \cite{LBJM}, in the large-period regime, we expect again that the spectral instability result could be partly recovered by combining a spectral instability result for solitary waves available in the literature for some specific semilinear equations \cite{Rousset-Tzvetkov-linear}, with a non-trivial spectral perturbation argument as mentioned above \cite{Gardner-large-period,Sandstede-Scheel-large-period,Yang-Zumbrun}.

\medskip

\noindent {\bf Extensions and open problems.} Since such plane waves play a role in the nearby modulation theory, the reader may wonder whether our main results extend to more general plane waves in the form \eqref{def:wave-general}. As pointed out in Section~\ref{s:profile-general}, it is straightforward to check that it is so for all results concerning longitudinal perturbations. Concerning instability under general perturbations, a first obvious answer is that instabilities persist under perturbations and thus extends to waves associated with small $\tkp$. In Appendix~\ref{s:more-waves}, we show how to extend the results to all waves in the semilinear case, that is, when $\kappa$ is constant, and in the high dimensional case, that is, when $d\geq3$.

At last, in Appendix~\ref{s:more-equations}, we show how to extend our results to anisotropic equations, even with dispersion of mixed signature, for waves propagating in a principal direction.

Though our results strongly hints at the multi-dimensional spectral instability of any periodic plane wave, it does leave this question unanswered, even for semilinear versions of \eqref{e:nls}. In the reverse direction of leaving some hope for stability, we stress that there are known natural examples of classes of one-dimensional equations for which both small-amplitude and large-period waves are unstable but there are bands of stable periodic waves. The reader is referred to \cite{BJNRZ-KS,JNRZ-KdV-KS,Barker} for examples on the Korteweg-de Vries/Kuramoto-Sivashinsky equation and to \cite{BJNRZ-KdV-SV,BJNRZ-KdV-SV-note} for examples on shallow-water Saint-Venant equations. We regard the elucidation of this possibility, even numerically, as an important open question. We point out as an intermediate issue whose resolution would already be interesting, and probably more tractable, the determination of whether there exist periodic waves of \eqref{e:nls} associated with wave parameters at which the modulation system~\eqref{e:W-intro-more} is weakly hyperbolic.

Let us conclude the global presentation of our main results by recalling that more specialized discussions, including more technical comparison to the literature, are provided along the text.

\medskip

\noindent {\bf Outline.} Next two sections contain general preliminary material, the first one on the structure of wave profile manifolds, the following one on adapted spectral theory. The latter contains however two highly non trivial results: spectral conjugations through linearized Madelung's transform (Section~\ref{s:linearized-EK}), and the slow/side-band expansion of the Evans' function (Theorem~\ref{th:low-freq}) --- a key block of our spectral analysis. After these two sections follow two sections devoted respectively to longitudinal perturbations and to general perturbations. Appendices contain key algebraic relations stemming from invariances and symmetries used throughout the text (Appendix~\ref{s:Noether}), the examination of constant-state spectral stability (Appendix~\ref{s:constant}), extensions to more general equations (Appendix~\ref{s:more-equations}) and more general profiles (Appendix~\ref{s:more-waves}) and a table of symbols (Appendix~\ref{s:index}).

Subsections of the two main sections are in clear correspondence with various sets of results described in the introduction so that the reader interested in some specific class of results may use the table of contents to jump at the relevant part of the analysis and meanwhile refer to the table of symbols to seek for involved definitions. 

\medskip

\phantomsection\label{notation}

\noindent {\bf Notation.} Before engaging ourselves in more concrete analysis, we make explicit here our conventions for vectorial, differential and variational notation.

Throughout we identify vectors with columns. The partial derivative with respect to a variable $a$ is denoted $\d_a$, or $\d_j$ when variables are numbered and $a$ is the $j$th one. The piece of notation $\dD$ stands for differentiation so that $\dD g(x)(h)$ denotes the derivative of $g$ at $x$ in the direction $h$. The Jacobian matrix $\Jac g(x)$ is the matrix associated with the linear map $\dD g(x)$ in the canonical basis. The gradient $\nabla g(x)$ is the adjoint matrix of $\Jac g(x)$ and we sometimes use suffix ${}_a$ to denote the gradient with respect to $a$. The Hessian operator $\Hess$ is given as the Jacobian of the gradient, $\Hess g=\Jac(\nabla g)$. The divergence operator $\Div$ is the opposite of the dual of the $\nabla$ operator. We say that a vector-field is curl-free if its Jacobian is valued in symmetric matrices. 

For any two vectors $\bV$ and $\bW$ in $\R^{d_0}$, thought of as column vectors, $\bV\otimes \bW$ stands for the rank-one, square matrix of  size $d_0$
\[
\bV\otimes \bW\,=\,\bV \,\transp{\bW}
\]
whatever $d_0$, where $\transp{}$ stands for matrix transposition. Acting on square-valued maps, $\Div$ acts row-wise. Dot $\cdot$ denotes the standard scalar product. Since, as a consequence of invariance by rotational changes, our differential operators act mostly component-wise, we believe that no confusion is possible and do not mark differences of meaning of $\cdot$ even when two vectorial structures coexist. The convention is that summation in scalar products is taken over compatible dimensions.  For instance,
\begin{align*}
\bV\cdot\nabla_{\bU}\Ham_0(\bU,\nabla_\bfx\bU)&=\sum_{j=1}^2 \bV_j\,\d_{U_j}\Ham_0(\bU,\nabla_\bfx\bU)\,,\\
\ex\cdot\nabla_{\nabla_{\bfx}\bU}\Ham_0(\bU,\nabla_\bfx\bU)
&=\sum_{j=1}^d (\ex)_j\,\nabla_{\d_j\bU}\Ham_0(\bU,\nabla_\bfx\bU)\,,\\
\nabla_{\bfx}\bU\cdot\nabla_{\nabla_{\bfx}\bU}\Ham_0(\bU,\nabla_\bfx\bU)
&=\sum_{j=1}^d\sum_{\ell=1}^2 \d_jU_\ell\,\nabla_{\d_jU_\ell}\Ham_0(\bU,\nabla_\bfx\bU)\,.
\end{align*}

We also use notation for differential calculus on functional spaces (thus in infinite dimensions), mostly in variational form. We use $L$ to denote linearization, analogously to $\dD$, so that $L(\cF)[\bU]\bV$ denotes the linearization of $\cF$ at $\bU$ in the direction $\bV$. Notation $\delta$ stands for variational derivative and plays a role analogous to gradient except that we use it on functional densities instead of functionals. With suitable boundary conditions, this would be the gradient for the $L^2$ structure of the functional associated with the given functional density at hand. We only consider functional densities depending of derivatives up to order $1$, so that this is explicitly given as
\[
\delta \cA[\bU]\,=\,
\nabla_\bU \cA(\bU,\nabla_\bfx\bU)\,-\Div_\bfx\,\left(\nabla_{\nabla_\bfx\bU}\cA(\bU,\nabla\bU)\right)\,.
\]
In this context, $\Hess$ denotes the linearization of the variational derivative, $\Hess=L\delta$, explicitly here 
\begin{align*}
\Hess \cA[\bU]\bV&\,=\,
\dD_{(\bU,\nabla_\bfx\bU)}(\nabla_\bU\cA)(\bU,\nabla_\bfx\bU)(\bV,\nabla_\bfx\bV)\\
&\quad\,-\Div_\bfx\,\left(\dD_{(\bU,\nabla_\bfx\bU)}(\nabla_{\nabla_\bfx\bU}\cA)(\bU,\nabla_\bfx\bU)(\bV,\nabla_\bfx\bV)\right)\,.
\end{align*}

Even when one is interested in a single wave, nearby waves enter in stability considerations. We use almost systematically underlining to denote quantities associated with the particular given background wave under study. In particular, when a wave parametrization is available, underlining denotes evaluation at the parameters of the wave under particular study.


\section{Structure of periodic wave profiles}\label{s:profile}

To begin with, we gather some facts about plane traveling wave manifolds. Until Section~\ref{s:profile-general}, we restrict to waves in the form \eqref{def:wave}. Consistently, here, for concision, we may set $\Impulse=\Impulse_1$.

\subsection{Radius equation}\label{s:radius}

To analyze the structure of the wave profiles, we step back 
from \eqref{def:wave} and look for profiles in the form
\be\label{def:wave-unscaled}
\bU(t,\bfx)\,=\,
\eD^{\omp\,t\bJ}\,\cV(x-\cx\,t)\,,
\ee
without normalizing to enforce $1$-periodicity. Profile equation 
becomes
\be\label{e:unscaled-profile}
0\,=\,\delta\Hamp[\cV]\,,\qquad\qquad\textrm{with}\qquad
\Hamp[\cV]\,=\,\Ham_0[\cV]-\omp\Mass[\cV]+\cx\Impulse[\cV]\,.
\ee
Moreover note that \eqref{e:unscaled-profile} also contains as a consequence of the rotational and spatial translation invariances of $\Hamp$ the following form of mass and momentum conservations
\begin{align}
\label{e:mass-sta}
0&=
-\frac{\dd}{\dd x}\left(\cV\cdot\bJ\nabla_{\bU_{x}}\Hamp[\cV]
\right)\,,\\
\label{e:mom-sta}
0&=
\frac{\dd}{\dd x}\left(-\Hamp[\cV]+\d_x\cV\cdot\nabla_{\bU_{x}}
\Hamp[\cV]\right)\,,
\end{align}
and introduce $\mup$ and $\mux$ corresponding constants of integration so that
\begin{align}\label{e:aug-mass}
\mup&=\bJ\cV\cdot\nabla_{\bU_{x}}\Hamp[\cV]\,,\\\label{e:aug-mom}
\mux&=\frac{\dd\cV}{\dd x}\cdot\nabla_{\bU_{x}}\Hamp[\cV]-\Hamp[\cV]\,.
\end{align}
Observe that reciprocally by differentiating 
\eqref{e:aug-mass}-\eqref{e:aug-mom} one obtains 
\[
\bJ\cV\cdot\delta\Hamp[\cV]=0
\qquad\textrm{and}\qquad
\bigg(\cV\cdot\frac{\dd\cV}{\dd x}\bigg)\ \cV\cdot\delta\Hamp[\cV]=0\,,
\]
which yields \eqref{e:unscaled-profile} provided the set where 
$\cV\cdot\frac{\dd\cV}{\dd x}$ vanishes has empty interior.

We check now that the above-mentioned condition on $\cV\cdot\frac{\dd\cV}{\dd x}$ excludes only solutions under the form \eqref{def:wave-unscaled} that have constant modulus and travel uniformly in phase. Since \eqref{e:unscaled-profile} is a 
differential equation, it is already clear that if $\cV$ vanishes 
on some nontrivial interval then $\cV\equiv0$ and from now on we 
exclude this case from our analysis. Then if 
$\cV\cdot\frac{\dd\cV}{\dd x}$ vanishes on some nontrivial interval 
it follows that on this interval $\|\cV\|$ is constant equal to some $r_0>0$ and from \eqref{e:aug-mass} that 
$$
\cV(x)\,=\,\eD^{\frac{2\mup-\cx\,r_0^2}{2\,\kappa(r_0^2)\,r_0^2}
\,x\,\bJ}\left(r_0\,\eD^{\vphip\bJ}\beD_1\right),
$$
for some $\vphip\in\R$. Since the formula provides a solution 
to \eqref{e:unscaled-profile} everywhere this holds everywhere and 
henceforth we also exclude this case. However these constant solutions are discussed further in Appendix~\ref{s:constant}.

Now to analyze \eqref{e:unscaled-profile} further we first recast 
\eqref{e:aug-mass}-\eqref{e:aug-mom} in a more explicit form, 
\begin{align*}
\mup&=\kappa(\|\cV\|^2)\,\bJ\cV\cdot\frac{\dd\cV}{\dd x}+\frac{\cx}{2}\|\cV\|^2\,,\\
\mux&=\frac12\kappa(\|\cV\|^2)\left\|\frac{\dd\cV}{\dd x}\right\|^2-W(\|\cV\|^2)+\frac{\omp}{2}\|\cV\|^2\,.
\end{align*}
Then we set $\alpha=\|\cV\|^2$ and observe that
\begin{align*}
\alpha\,\frac{\dd\cV}{\dd x}&= \frac12\frac{\dd \alpha}{\dd x}\,\cV
+\bJ\cV\cdot\frac{\dd\cV}{\dd x}\,\bJ\cV\,,\\
\alpha\left\|\frac{\dd\cV}{\dd x}\right\|^2&=
\frac{1}{4}\left(\frac{\dd \alpha}{\dd x}\right)^2+\,\left(\bJ\cV\cdot\frac{\dd\cV}{\dd x}\right)^2\,.
\end{align*}
In particular from \eqref{e:aug-mass}-\eqref{e:aug-mom} stems
\be\label{e:radius-square}
\frac{1}{8}\kappa(\alpha)\left(\frac{\dd \alpha}{\dd x}\right)^2
+\cW_\alpha(\alpha;\cx,\omp,\mup)\,=\,\mux\,\alpha\,,
\ee
with
\be\label{e:a-potential}
\cW_\alpha(\alpha;\cx,\omp,\mup)
:=-W(\alpha)\,\alpha+\frac{\omp}{2}\,\alpha^2
+\frac{1}{8}\frac{(2\mup-\cx\,\alpha)^2}{\kappa(\alpha)}\,.
\ee
Consistently going back to \eqref{e:unscaled-profile}, one derives
\be\label{e:eq-radius-square}
\frac{1}{4}\kappa(\alpha)\frac{\dd^2 \alpha}{\dd x^2}
+\frac{\kappa'(\alpha)}{\kappa(\alpha)}(\mux\,\alpha-\cW_\alpha(\alpha))+\d_\alpha\cW_\alpha(\alpha)\,=\,\mux\,.
\ee

As a consequence, since $\alpha\geq0$, if $\alpha$ vanishes at some 
point then its derivative also vanishes there and $\mup=0$. From 
this we deduce near the same point 
$$
\frac{\dd \alpha}{\dd x}=\cO(\alpha)\qquad\textrm{and}\qquad
\bJ\cV\cdot\frac{\dd\cV}{\dd x}=\cO(\alpha)
\qquad\qquad\textrm{hence}\qquad
\frac{\dd\cV}{\dd x}=\cO(\sqrt{\alpha})\,,
$$
This implies $\mux=-W(0)$ and corresponds to the trivial solution 
to \eqref{e:unscaled-profile} given by $\cV\equiv0$ that we have 
already ruled out. Note that this exclusion may be enforced by 
requiring $(\mux,\mup)\neq(-W(0),0)$. 

The foregoing discussion ensures that actually $\cV$ does not 
vanish so that in particular $r=\sqrt{\alpha}=\|\cV\|$ is a smooth 
function solving 
\be\label{e:radius}
\frac12\kappa(r^2)\left(\frac{\dd r}{\dd x}\right)^2
+\cW_r(r;\cx,\omp,\mup)\,=\,\mux,
\ee
where $\cW_r$ is defined by
\begin{align}\label{e:reduced-potential}
\cW_r(r;\cx,\omp,\mup)
&\,:=\,
\frac{1}{r^2}\cW_\alpha(r^2;\cx,\omp,\mup)\\\nonumber
&\,=\,
-W(r^2)+\frac{\omp}{2}\,r^2
+\frac18\frac{(2\mup-\cx\,r^2)^2}{\kappa(r^2)\,r^2}\,,
\end{align}
and
\be\label{e:eq-radius}
\kappa(r^2)\frac{\dd^2 r}{\dd x^2}
+2r\,\frac{\kappa'(r^2)}{\kappa(r^2)}(\mux\,-\cW_r(r))+\d_r\cW_r(r)\,=\,0\,.
\ee
Note that the excluded case where $r$ is constant equal to some $r_0$ happens only when
$$
\mux\,=\,-\cW_r(r_0;\cx,\omp,\mup)\qquad\textrm{and}\qquad 
0\,=\,\d_r\cW_r(r_0;\cx,\omp,\mup)\,.
$$

When coming back from \eqref{e:radius} to \eqref{e:unscaled-profile}, 
some care is needed when $\mup$ is zero since then $\cW_r$ may be extended to $\R$ but solutions to \eqref{e:radius} taking negative 
values must still be discarded. Except for that point, one readily 
obtains from \eqref{e:aug-mass} that with any solution $r$ to 
\eqref{e:radius} is associated the family of solutions to 
\eqref{e:unscaled-profile}-\eqref{e:aug-mass}-\eqref{e:aug-mom}
$$
\cV(x)\,=\,\eD^{\ds\left(\int_0^{x+\vphix}\frac{2\mup-\cx\,r(y)}
{2\,\kappa(r(y)^2)\,r(y)^2}\,\dD y\right)\,\bJ}\left(r(x+\vphix)\,
\eD^{\vphip\bJ}\beD_1\right),
$$
parametrized by rotational and spatial shifts $(\vphip,\vphix)\in\R^2$. 

Classical arguments show that if parameters are such that \eqref{e:radius} defines\footnote{The relation could define many connected components but implicitly we discuss them one by one. See Figures~\ref{phasesmall} and~\ref{phasesol}.} a non-trivial closed curve in phase-space that is included in the half-plane $r>0$, then the above construction yields a wave of the sought form, unique up to translations in rotational and spatial positions.

\subsection{Jump map}\label{s:jumps}
Rather than on the existence of periodic waves, we now turn our focus on their parametrization, assuming the existence of a given reference wave $\ucV$. As announced, parameters associated with $\ucV$ are underlined, and, more generally, any functional $\cF$ evaluated at the reference wave is denoted $\ucF$. 

In the unscaled framework, instead of wavenumbers $(\kx,\kp)$, we rather manipulate the spatial period $\Xx:=1/\kx$, and the rotational shift\footnote{We refrain from using the word Floquet exponent for $\xip$ to avoid confusion with Floquet exponents involved in integral transforms.} $\xip:=\kp/\kx$ that satisfy
$$
\cV(\cdot +\Xx)=\eD^{\xip\,\bJ}\cV(\cdot).
$$
Our goal is to show the existence of nearby waves and to determine which parameters are suitable for wave parametrization among
\begin{align*}
\omp\,,&&\text{rotational pulsation}\\
\cx\,,&&\text{spatial speed}\\
\cV(0),\,\frac{\dD \cV}{\dD x}(0),&&\text{initial data for the wave profile ODE}\\ 
\mup,\,\mux,&&\text{constants of integration associated with conservation laws}\\
\Xx,&&\text{spatial period}\\
\xip,&&\text{rotational shift after a period}\\
\vphip,\,\vphix,&&\text{rotational and spatial translations.}
\end{align*}

It follows from the Cauchy-Lipschitz theory that functions $\cV$ satisfying equation \eqref{e:unscaled-profile} are uniquely and smoothly determined by initial data $(\cV(0),\,\frac{\dD \cV}{\dD x}(0))=(\cV_0,\cV_1)$, and parameters of the equation $(\omp,\,\cx)$, on some common neighborhood of $[0,\uXx]$ provided that $(\cV_0,\cV_1,\omp,\cx)$ is sufficiently close to $(\ucV(0),\,\tfrac{\dD \ucV}{\dD x}(0),\uomp,\ucx)$. Note that the point $0$ plays no particular role and we may use a spatial translation to replace it with another nearby point so as to ensure suitable conditions on $(\ucV(0),\,\tfrac{\dD \ucV}{\dD x}(0))$. In particular, there is no loss in generality in assuming that $\ucV(0)\cdot \frac{\dD \ucV}{\dD x}(0)\neq0$.

At this stage, to carry out algebraic manipulations it is convenient to introduce notation
\begin{align*}
\Sp[\bU]&\,:=\,\bJ\bU\cdot\nabla_{\bU_{x}}\Ham_0[\bU]\,,\\
\Sx[\bU]&\,:=\,-\Ham_0[\bU]+\bU_x\cdot\nabla_{\bU_{x}}\Ham_0[\bU]\,.
\end{align*}
so that \eqref{e:aug-mass}-\eqref{e:aug-mom} is written as
\begin{align}\label{e:massu}
\mup&\,=\,\Sp[\cV]+\cx\Mass[\cV]\,,\\
\label{e:momu}
\mux&\,=\,\Sx[\cV]+\omp\Mass[\cV]\,.
\end{align}
Now we observe that $\dD_{\bU_x}(\Sp,\Sx)(\ucV(0),\tfrac{\dd\ucV}{\dd x}(0))$ has determinant $(\kappa(\|\ucV(0)\|^2))^2\,\ucV(0)\cdot\frac{\dD\ucV}{\dD x}(0)\neq0$. In particular, as a consequence of the Implicit Function Theorem, for $(\cV_0,\cV_1,\cx,\omp,\mup,\mux)$ near $(\ucV(0),\tfrac{\dd\ucV}{\dd x}(0),\ucx,\uomp,\umup,\umux)$, 
\begin{align*}
\mup&\,=\,\Sp(\cV_0,\cV_1)+\cx\Mass(\cV_0)\,,\\
\mux&\,=\,\Sx(\cV_0,\cV_1)+\omp\Mass(\cV_0)\,,
\end{align*}
is smoothly (and equivalently) solved as 
\[
\cV_1=\cV^{1}(\cV_0;\cx,\omp,\mup,\mux)\,.
\]
The same is true near $(\ucV(\Xx),\tfrac{\dd\ucV}{\dD x}(\Xx),\ucx,\uomp,\umup,\umux)$. This implies that, on one hand, one may replace $(\cV_0,\cV_1,\omp,\cx)$ with $(\cV_0,\omp,\cx,\mup,\mux)$ in the parametrization of solutions to \eqref{e:unscaled-profile} and, on the other hand, since values of $(\Sp[\cV]+\cx\Mass[\cV],\Sx[\cV]+\omp\Mass[\cV])$ are invariant under the flow of \eqref{e:unscaled-profile}, that, as a consequence of the Cauchy-Lipschitz theory, solutions to \eqref{e:unscaled-profile} defined on a neighborhood of $[0,\uXx]$ extend as solutions on $\R$ such that $\cV(\cdot +\Xx)=e^{\xip\,\bJ}\cV(\cdot)$ if and only if $\cV(\Xx)=e^{\xip\,\bJ}\cV(0)$.

We now show that we may replace $(\cV_0,\omp,\cx,\mup,\mux)$ with $(\vphip,\vphix,\omp,\cx,\mup,\mux)$ by taking the solution corresponding to $\cV_0=\ucV(0)$ and acting with rotational and spatial translations. The action of rotational and spatial translations is $\cV(\cdot) \mapsto \cV_{\vphip,\vphix}:=\eD^{\vphip\bJ}\cV(\cdot+\vphix)$. Obviously it leaves the set of periodic-wave profiles invariant and, among parameters, interacts only with initial data, thus, after the elimination of $\cV_1$, only with $\cV_0$. Let us denote by $\cV^{(\mux,\cx,\omp,\mup)}$ the solution to \eqref{e:unscaled-profile} such that 
\begin{align*}
\cV^{(\mux,\cx,\omp,\mup)}(0)&=\ucV(0)\,,&
\frac{\dD\ }{\dD x}\cV^{(\mux,\cx,\omp,\mup)}(0)
&=\cV^{1}(\ucV(0);\cx,\omp,\mup,\mux)\,.
\end{align*}
At background parameters the map $(\vphip,\vphix,\mux,\cx,\omp,\mup)\mapsto (\cV^{(\mux,\cx,\omp,\mup)})_{\vphip,\vphix}(0)$ has Jacobian determinant with respect to variations in $(\vphip,\vphix)$ equal to $\ucV(0)\cdot\frac{\dD\ucV}{\dD x}(0)\neq0$. Thus, as claimed, as a consequence of the Implicit Function Theorem, one may smoothly and invertibly replace $(\cV_0,\omp,\cx,\mup,\mux)$ with $(\vphip,\vphix,\omp,\cx,\mup,\mux)$ to parametrize solutions to \eqref{e:unscaled-profile} near the background profile.

As a conclusion, when identified up to rotational and spatial translations, periodic-wave profiles are smoothly identified as the zero level set of the map
\[
(\mux,\cx,\omp,\mup,\Xx,\xip)\mapsto \cV^{(\mux,\cx,\omp,\mup)}(\Xx)-\eD^{\xip\bJ}\cV^{(\mux,\cx,\omp,\mup)}(0)\,.
\]
Now, at background parameters, the foregoing map has Jacobian determinant with respect to variations in $(\Xx,\xip)$ equal to $\ucV(0)\cdot\frac{\dD\ucV}{\dD x}(0)\neq0$. Therefore a third application of the Implicit Function Theorem achieves the proof of the following proposition.

\begin{proposition}
Near a periodic-wave profile with non constant mass, periodic wave profiles form a six-dimensional manifold smoothly parametrized as
\[
(\vphip,\vphix,\omp,\cx,\mup,\mux)
\mapsto (\cV^{(\omp,\cx,\mup,\mux)}_{\vphip,\vphix},\Xx(\omp,\cx,\mup,\mux),\xip(\omp,\cx,\mup,\mux))
\]
with for any $(\vphip,\vphix)$,
\[
\cV^{(\omp,\cx,\mup,\mux)}_{\vphip,\vphix}(\cdot)
\,=\,\eD^{\vphip\bJ}\cV^{(\omp,\cx,\mup,\mux)}_{0,0}(\cdot+\vphix)\,.
\]
\end{proposition}

\subsection{Madelung's transformation}\label{s:EK}

To ease comparisons with the analyses in \cite{Benzoni-Noble-Rodrigues-note,Benzoni-Noble-Rodrigues,Benzoni-Mietka-Rodrigues,BMR2-I,BMR2-II} for dispersive systems of Korteweg type, including Euler--Korteweg systems and quasilinear Korteweg--de Vries equations, we now provide hydrodynamic formulations of \eqref{e:nls}/\eqref{e:ab} and correspondences between respective periodic waves. The reader is referred to \cite{Benzoni-survey} for similar discussions concerning other kinds of traveling waves. 

In the present section, we temporarily go back to the general multi-dimensional framework.

On one hand, we consider for $f=a\,+\,\iD\,b$, $\bU=\bp a\\b\ep$, a system in the form
\be\label{e:absHam}
\d_t\bU=\bJ\,\delta \Ham_{\#}[\bU]
\qquad\text{with}\qquad
\Ham_{\#}\left[\bU\right]
=\Ham_{\rm eff}\left(\Mass[\bU],\bImpulse[\bU],\frac12\,\|\nabla_\bfx\bU\|^2\right)\,.
\ee
Then we introduce
\[
\cU(\rho,\theta)\,:=\,\sqrt{2\,\rho}\,\eD^{\theta\,\bJ}(\beD_1)\,\,,\qquad (\rho,\theta)\in\R_+\times\R\,,
\]
and 
\[
H_{\#}[(\rho,\bfv)]
:=\Ham_{\rm eff}\left(\rho,\rho\,\bfv,
\frac{1}{4\rho}\|\nabla_\bfx\rho\|^2+\rho\,\|\bfv\|^2\right)
\]
and observe that
\begin{align*}
\rho
&=\Mass[\cU(\rho(\cdot),\theta(\cdot))]\,,&
\nabla_\bfx \theta
&=\frac{\bImpulse[\cU(\rho(\cdot),\theta(\cdot))]}{\Mass[\cU(\rho(\cdot),\theta(\cdot))]}\,,&
H_{\#}[(\rho,\nabla_\bfx \theta)]
&=\Ham_{\#}\left[\cU(\rho(\cdot),\theta(\cdot))\right]\,.
\end{align*}

We also point out that
\begin{align}\label{e:MTrho}
\cU(\rho,\theta)\cdot
\bJ\,\delta\Ham_{\#}\left[\cU(\rho(\cdot),\theta(\cdot))\right]
&=\Div_\bfx\left(\delta_\bfv H_{\#}[(\rho,\nabla_\bfx \theta)]\right)\\\label{e:MTv}
\frac{1}{2\rho}\bJ\,\cU(\rho,\theta)\cdot
\bJ\,\delta\Ham_{\#}\left[\cU(\rho(\cdot),\theta(\cdot))\right]
&=\delta_\rho H_{\#}[(\rho,\nabla_\bfx \theta)]
\end{align}
so that if $\bU$ solves \eqref{e:absHam} and is bounded away from 
zero then
\be\label{def:Madelung}
(\rho,\bfv):=\left(\Mass[\bU],\frac{\bImpulse[\bU]}{\Mass[\bU]}\right)
\ee
solves 
\be\label{e:absEK}
\d_t\begin{pmatrix}\rho\\\bfv\end{pmatrix}=\cJ\,\delta H_{\#}[(\rho,\bfv)]
\ee
with the constraint that $\bfv$ is curl-free, where $\cJ$ denotes the skew-symmetric operator
\[
\cJ:=\begin{pmatrix}0&\Div_\bfx\\\nabla_\bfx&0\end{pmatrix}\,.
\]
Note that the curl-free constraint is preserved by the time-evolution 
so that it is sufficient to prescribe it on the initial data.

Reciprocally if $(\rho,\bfv)$ solves \eqref{e:absEK} and $\rho$ is 
bounded below away from zero, then for any $\theta$ such that
\be\label{e:abstheta}
\d_t\theta=\delta_\rho H_{\#}[(\rho,\bfv)]
\ee
we have $\nabla_\bfv\theta=\bfv$ and $\bU:=\cU(\rho(\cdot),\theta(\cdot))$ 
solves \eqref{e:absHam}. Note moreover that under such conditions, for 
any $(t_0,\bfx_0,\theta_0)$, \eqref{e:abstheta} possesses a unique 
solution such that $\theta(t_0,\bfx_0)=\theta_0$ and that, for any 
$\bfx_0$, \eqref{e:abstheta} could alternatively be replaced by : 
for any $t$, $\d_t\theta(t,\bfx_0)=\delta_\rho H_{\#}[(\rho,\bfv)]
(t,\bfx_0)$ and $\nabla_\bfv\theta(t,\cdot)=\bfv(t,\cdot)$.

We point out that whereas the Madelung transformation 
$\bU\mapsto (\rho,\bfv)$ quotients the rotational invariance, it 
preserves the time and space translation invariances. With respect 
to the latter, we consider
\[
Q_j[\rho,\bfv]\,:=\,\rho\,\bfv\cdot\beD_j\,,\qquad j=1,\cdots,d,
\]
and observe that on one hand $Q_j$ generates spatial translations 
along the direction $\beD_j$ in the sense that if $\bfv$ is 
curl-free then
\[
\beD_j\cdot\nabla\begin{pmatrix}\rho\\\bfv\end{pmatrix}
\,=\,\cJ\,\delta Q_j[(\rho,\bfv)]
\]
and that on the other hand 
\begin{align*}
Q_j[(\rho,\nabla \theta)]
&=\Impulse_j\left[\cU(\rho(\cdot),\theta(\cdot))\right]\,.
\end{align*}
We also note that \eqref{e:absEK} implies
\[
\d_t\left(Q_j(\rho,\bfv)\right)\,=\,
\d_j\left(\rho\,\d_\rho H_{\#}[(\rho,\bfv)]-H_{\#}[(\rho,\bfv)]\right)
+\Div_\bfx\left(v_j\,\nabla_\bfv H_{\#}[(\rho,\bfv)]
+\rho_{x_j}\,\nabla_{\nabla_\bfx\rho} H_{\#}[(\rho,\bfv)]\right)
\]
and, for comparison with \eqref{e:momentum}, that when 
$\bU=\cU(\rho(\cdot),\theta(\cdot))$, $\bfv=\nabla_\bfx \theta$,
\begin{align*}
\nabla_{\bU_{x_j}}\Impulse_j[\bU]\cdot\bJ\delta\Ham_{\#}[\bU]
&=\rho\,\d_\rho H_{\#}[(\rho,\bfv)]\,,\\
\bJ\delta\Impulse_j[\bU]\cdot\nabla_{\bU_{x_\ell}}\Ham_{\#}[\bU]
&=v_j\,\d_{v_\ell} H_{\#}[(\rho,\bfv)]
+\rho_{x_j}\,\d_{\rho_{x_\ell}} H_{\#}[(\rho,\bfv)]\,.
\end{align*}

Concerning the time translation invariance, we note that 
\eqref{e:absEK} implies
\[
\d_t\left(H_{\#}(\rho,\bfv)\right)\,=\,
\Div_\bfx\Big(\delta_\rho H_{\#}[(\rho,\bfv)]\,\nabla_\bfv H_{\#}[(\rho,\bfv)]
+\Div_\bfx(\nabla_\bfv H_{\#}[(\rho,\bfv)])\,\nabla_{\nabla_\bfx\rho} H_{\#}
[(\rho,\bfv)]\Big)
\]
and, for comparison with \eqref{e:Ham}, that when 
$\bU=\cU(\rho(\cdot),\theta(\cdot))$, $\bfv=\nabla_\bfx \theta$,
\[
\nabla_{\bU_{x_j}}\Ham_{\#}[\bU]\cdot\bJ\delta\Ham_{\#}[\bU]
\,=\,
\delta_\rho H_{\#}[(\rho,\bfv)]\,\d_{v_j} H_{\#}[(\rho,\bfv)]
+\Div_\bfx(\nabla_\bfv H_{\#}[(\rho,\bfv)])\,\d_{\rho_{x_j}} H_{\#}
[(\rho,\bfv)]
\]

In the hydrodynamic formulation, what replace to some extent the rotational invariance and its accompanying conservation law for $\Mass[\bU]$ are the fact that the time evolution in \eqref{e:absEK} obeys a system of $d+1$ conservation laws and that one may add to $H_{\#}$ any affine function of $(\rho,\bfv)$ without changing \eqref{e:absEK}. With this respect, to compare \eqref{e:mass} with the equation on $\d_t\rho$, note that when $\bU=\cU(\rho(\cdot),\theta(\cdot))$, $\bfv=\nabla_\bfx \theta$,
\[
\bJ\delta\Mass[\bU]\cdot\nabla_{\bU_{x_j}}\Ham_{\#}[\bU]
\,=\,\d_{v_j} H_{\#}[(\rho,\bfv)]\,.
\]

To make the discussion slightly more concrete, we compute that when 
$\Ham_{\#}=\Ham_0$ one receives
\begin{equation}
\label{e:concrete-H}
H_0[(\rho,\bfv)]:=
H_{\#}[(\rho,\bfv)]=
\kappa(2\,\rho)\,\rho\,\|\bfv\|^2
+\frac{\kappa(2\,\rho)}{4\,\rho}\|\nabla \rho\|^2+W(2\,\rho)
\end{equation}
and that when $\Ham_{\#}=\Hamp$ one receives
\[
\Hp[(\rho,\bfv)]:=
H_{\#}[(\rho,\bfv)]=
H_0[(\rho,\bfv)]
-\omp\rho+\cx Q_1(\rho,\bfv)\,.
\]

Turning to the identification of periodic traveling waves moving in 
the direction $\beD_1$, we now restrict spatial variable to 
dimension $1$ and consider functions independent of time. We point 
out that $\cV$ is a solution to
\begin{align*}
0&\,=\,\delta\Hamp[\cV]\,,&
\mup&=\bJ\cV\cdot\nabla_{\bU_{x}}\Hamp[\cV]\,,
\end{align*}
bounded away from zero if and only if $\cV=\cU(\rho(\cdot),
\theta(\cdot))$, with $\rho$ bounded below away from zero, 
$v=\tfrac{\dd\theta}{\dd x}$, and $0\,=\,\delta\HEK[(\rho,v)]$ where
\begin{align}\label{eq:EK}
\HEK[(\rho,v)]&:=\Hp[(\rho,v)]
-\mup\,v
\,=\,
H_0[(\rho,v)]
-\omp\rho-\mup\,v+\cx Q(\rho,v)\,.
\end{align}
Moreover then with $\mux$ as in \eqref{e:aug-mom},
\begin{align*}
\mux&=-\Hp[(\rho,v)]+v\,\mup+\rho_{x}\,\d_{\rho_{x}} \Hp[(\rho,v)]\,\\
&=\rho_{x}\,\d_{\rho_{x}} \HEK[(\rho,v)]-\HEK[(\rho,v)]\,,
\end{align*}
and 
\begin{equation}\label{defnu}
v\,=\,\nu(\rho;\cx,\mup):=
\frac{\mup-\cx\,\rho}{2\,\rho\,\kappa(2\,\rho)}\,.
\end{equation}
Furthermore we stress that under these circumstances there exists 
$\kp$ such that $x\mapsto\eD^{-\kp\,x\,\bJ}\cV(x)$ is periodic of 
period $\Xx$ if and only $(\rho,v)$ is periodic of period $\Xx$, 
and when this happens $\kp$ is the average of $v$ over one period.

With notation\footnote{Except that $(\rho,v)$ plays the role of 
$(v,u)$ in \cite{BMR2-II}.} from \cite{BMR2-II}, we have untangled 
the correspondences in parameters 
\begin{align*}
\cx&=c\,,&
\mux&=\mu\,,&
\Xx&=\Xi\,,&\\
\omp&=-\lambda_\rho\,,&
\mup&=-\lambda_v\,,&
\kp&=\langle v\rangle\,,&
\end{align*}
besides the pointwise correspondences of mass, momentum and 
Hamiltonian.

We also point out that we have recovered the reduction of profile 
equations to a two-dimensional Hamiltonian system associated with 
\be\label{e:profile-rho}
\frac{\kappa(2\,\rho)}{4\,\rho}\left(\frac{\dd \rho}{\dd x}\right)^2
+\cW_\rho(\rho;\cx,\omp,\mup)\,=\,\mux\,,
\ee
where
\begin{equation}\label{def:wrho}
\cW_\rho(\rho;\cx,\omp,\mup)
:=-W(2\,\rho)
-\kappa(2\,\rho)\,\rho\,(\nu(\rho))^2
+\omp\,\rho+\mup\,\nu(\rho)
-\cx Q(\rho,\nu(\rho))\,,
\end{equation}
with $\nu(\rho)=\nu(\rho;\cx,\mup)$.

\subsection{Action integral}\label{s:action}

Motivated by the foregoing subsections we introduce 
\be\label{def:action}
\Theta(\mux,\cx,\omp,\mup):= \int_{0}^{\Xx} 
\left(\Ham_0[\cV]+\cx\Impulse[\cV]
-\omp\,\Mass[\cV]
-\mup\,\frac{\Impulse[\cV]}{\Mass[\cV]}
+\mux\right)\,\dd x\,,
\ee
with $(\Xx,\cV)$ associated with $(\mux,\cx,\omp,\mup)$ as in 
Section~\ref{s:jumps}. Note that $\Theta$ is indeed independent of 
$(\vphip,\vphix)$, and since
\[
\kp\,\Xx\,=\,\xip\,=\,
\int_{0}^{\Xx} \frac{\Impulse[\cV]}{\Mass[\cV]}\,\dd x
\]
we also have
\[
\Theta(\mux,\cx,\omp,\mup)= \int_{0}^{\Xx} 
\left(\Ham_0[\cV]+\cx\Impulse[\cV]
-\omp\,\Mass[\cV]
-\mup\,\kp
+\mux\right)\,\dd x\,,
\]
with $(\Xx,\kp,\cV)$ associated with $(\mux,\cx,\omp,\mup)$ as 
in Section~\ref{s:jumps}.

Based on \eqref{e:profile-rho} we stress the following basic alternative formula
\[
\Theta(\mux,\cx,\omp,\mup)
=
2\,\int_{\rhomin}^{\rhomax} 
\sqrt{\mux-\cW_\rho(\rho;\cx,\omp,\mup))}\sqrt{\frac{
\kappa(2\,\rho)}{\rho}}\,\dd \rho
\]
where $\rhomin=\rhomin(\mux,\cx,\mup,\omp)$ and 
$\rhomax=\rhomax(\mux,\cx,\mup,\omp)$ are respectively the minimum 
and the maximum values of $\Mass[\cV]$. Note that 
$\rhomin$ and $\rhomax$ are (locally) characterized by
\begin{align*}
\mux&=\cW_\rho(\rhomin;\cx,\omp,\mup)\,,&
\mux&=\cW_\rho(\rhomax;\cx,
\omp,\mup)\,.
\end{align*}

A fundamental observation, intensively used in 
\cite{Benzoni-Noble-Rodrigues-note,Benzoni-Noble-Rodrigues,
Benzoni-Mietka-Rodrigues,BMR2-I,BMR2-II}, is that
\begin{equation}\label{dtheta}
\left\{
\begin{array}{ll}
\di \d_{\mux}\Theta(\mux,\cx,\omp,\mup)&\,=\,\Xx\,,\\
\di \d_{\cx}\Theta(\mux,\cx,\omp,\mup)&\,=\,
\di \int_{0}^{\Xx} \Impulse[\cV]\,\dd x\,,\\
\di \d_{\omp}\Theta(\mux,\cx,\omp,\mup)&\,=\,
\di -\int_{0}^{\Xx} \Mass[\cV]\,\dd x\,,\\
\di \d_{\mup}\Theta(\mux,\cx,\omp,\mup)&\,=\,
\di -\int_{0}^{\Xx} \frac{\Impulse[\cV]}{\Mass[\cV]}\,\dd x\,.
\di 
\end{array}
\right.
\end{equation}
See for instance \cite[Proposition~1]{Benzoni-Noble-Rodrigues-note} for a proof\footnote{Let us recall 
that in this reference, the role of $(\rho,v)$ is played by $(v,u)$. Note that the proof given there uses $\ucV\cdot\ucV_x(0)=0$, but this may be assumed up to an harmless spatial translation.} of this elementary fact. For each of those we also have 
\[
\d_{\#}\Theta(\mux,\cx,\omp,\mup)
=
\int_{\rhomin}^{\rhomax} 
\frac{\d_{\#}(\mux-\cW_\rho)(\rho;\cx,\omp,\mup)}{\sqrt{\mux-\cW_\rho(\rho;\cx,\omp,\mup))}}\sqrt{\frac{2\,\kappa(2\,\rho)}{2\,\rho}}\,\dd \rho\,.
\]

\subsection{Asymptotic regimes}
\label{s:asymp}
As in \cite{BMR2-I,BMR2-II}, we shall specialize most of the general results to two asymptotic regimes, small amplitude and large period asymptotics.

We make explicit here the descriptions of both regimes in terms of 
parameters. Let $(\urho^{(0)},\ukp^{(0)})\in(0,\infty)\times\R$. 
Then for any $\phi^{(0)}\in\R$, 
\[
\cV^{(0)}(x)=\sqrt{2\urho^{(0)}}\eD^{(\phi^{(0)}+\ukp^{(0)}\,x)\bJ}(\beD_1)
\]
defines an unscaled profile with parameters 
$(\umux^{(0)},\ucx^{(0)},\uomp^{(0)},\umup^{(0)})$ 
determined by (see \eqref{e:aug-mass},\eqref{e:aug-mom})
\begin{align*}
\umup^{(0)}&=\,\ucx^{(0)}\urho^{(0)}+\kappa(2\urho^{(0)})\,2\,\urho^{(0)}\,\ukp^{(0)}\,,\\
\uomp^{(0)}&=\,\ucx^{(0)}\ukp^{(0)}
+\left(\kappa'(2\urho^{(0)})\,2\,\urho^{(0)}+\kappa(2\urho^{(0)})\right)\,(\ukp^{(0)})^2
+2W'(2\urho^{(0)})\\
\umux^{(0)}&=\,
-\frac12\,\kappa(2\urho^{(0)})\,2\urho^{(0)}\,(\ukp^{(0)})^2-W(2\urho^{(0)})
-\ucx^{(0)}\urho^{(0)}\ukp^{(0)}
+\uomp^{(0)}\urho^{(0)}
+\umup^{(0)}\ukp^{(0)}
\end{align*}
except for $\ucx^{(0)}\in\R$, which may be chosen arbitrarily. Using $\nu$, and $\cW_\rho$ introduced in \eqref{defnu}-\eqref{def:wrho}, the determination of parameters is equivalently 
written as 
\begin{align*}
\ukp^{(0)}
&=\nu(\urho^{(0)};\ucx^{(0)},\umup^{(0)})\,,\\
0&=\d_\rho\cW_\rho(\urho^{(0)};\ucx^{(0)},\uomp^{(0)},\umup^{(0)})\,,\\
\umux^{(0)}&=\cW_\rho(\urho^{(0)};\ucx^{(0)},\uomp^{(0)},\umup^{(0)})\,.
\end{align*}
On this alternate formulation, it is clear that we could instead 
fix $(\urho^{(0)},\umup^{(0)},\ucx^{(0)})\in(0,\infty)\times\R^2$ 
and determine $(\ukp^{(0)},\uomp^{(0)},\umux^{(0)})$ 
correspondingly. 

We are only interested in non-degenerate constant solutions, and 
thus assume
\[
\d_\rho^2\cW_\rho(\urho^{(0)};\ucx^{(0)},\uomp^{(0)},\umup^{(0)})
\neq0.
\]
Under this condition, for any $(\cx^{(0)},\,\omp^{(0)},\,\mup^{(0)})$ 
in some neighborhood of $(\ucx^{(0)},\,\uomp^{(0)},\,\umup^{(0)})$ 
there is a unique corresponding 
\[
(\rho^{(0)},\kp^{(0)},\mux^{(0)}):=
(\rho^{(0)},\kp^{(0)},\mux^{(0)})(\cx^{(0)},\omp^{(0)},\mup^{(0)})
\] 
in some neighborhood of $(\urho^{(0)},\ukp^{(0)},\umux^{(0)})$.

When 
\[
\d_\rho^2\cW_\rho(\urho^{(0)};\ucx^{(0)},\uomp^{(0)},\umup^{(0)})>0\,,
\]
to any  $(\cx,\omp,\mup,\mux)$ sufficiently close to 
$(\ucx^{(0)},\uomp^{(0)},\umup^{(0)},\umux^{(0)})$ and satisfying
\[
\mux>\mux^{(0)}(\cx,\omp,\mup)
\]
corresponds a unique --- up to rotational and spatial translations 
invariances --- periodic traveling wave with mass 
close\footnote{Recall that there could be various branches 
corresponding to the same parameters. We give this precision to 
exclude other branches; see Figure~\ref{phasesmall}.} 
to $\rho^{(0)}(\cx,\omp,\mup)$. The small 
amplitude limit denotes the asymptotics $\mux-\mux^{(0)}(\cx,\omp,
\mup)\to 0$ and the \emph{small amplitude regime} is the zone 
where $\mux-\mux^{(0)}(\cx,\omp,\mup)$ is small but positive. 
Incidentally we point that the limiting small amplitude period is 
given by
\be\label{e:harm_period}
\Xx^{(0)}(\cx,\omp,\mup)
:=2\pi\,\sqrt{\frac{\kappa(2\,\rho^{(0)})}{2\rho^{(0)}
\d_\rho^2\cW_\rho(\rho^{(0)};\cx,\omp,\mup)}}\,.
\ee
with $\rho^{(0)}=\rho^{(0)}(\cx,\omp,\mup)$.
\begin{figure}[ht]
\begin{center}
 \includegraphics[scale=0.33]{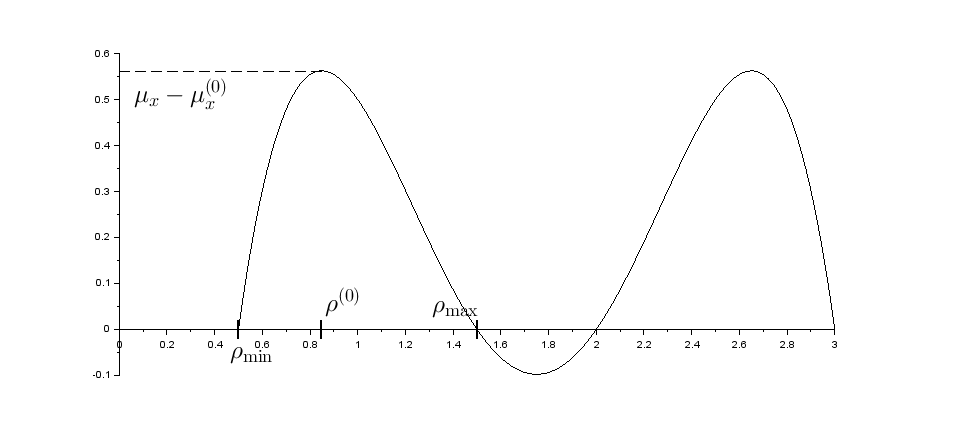}
 \includegraphics[scale=0.33]{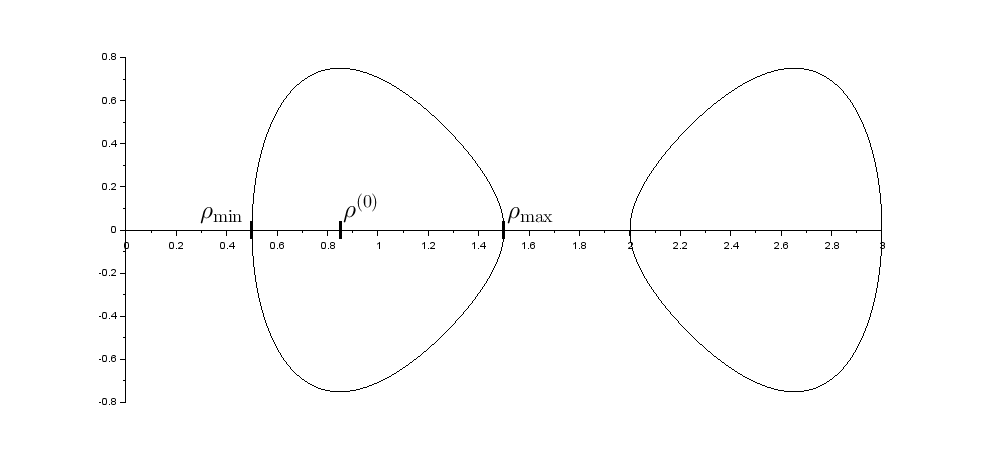}
 \end{center}
 \caption{{\bf Small-amplitude limit} with two branches with same parameters. The upper graph is the graph of $\mu_x-\cW_\rho$ as a function of $\rho$. The lower graph is, in the $(\rho,\tfrac{\dD \rho}{\dD x})$ phase plane, the level set defined by \eqref{e:profile-rho}. The two closed curves 
 correspond to two periodic waves, the curve of interest being the 
 one circling $\rho^{(0)}$.} 
\label{phasesmall}
 \end{figure}

When
\[
\d_\rho^2\cW_\rho(\urho^{(0)};\ucx^{(0)},\uomp^{(0)},\umup^{(0)})<0\,,
\]
there are at most two solitary wave profiles with parameters 
$(\ucx^{(0)},\uomp^{(0)},\umup^{(0)})$, namely at most one with 
$\urho^{(0)}$ as both an infimum and an endstate for its mass and 
at most one with $\urho^{(0)}$ as both a supremum and an endstate 
for its mass; see Figure~\ref{phasesol}. Concerning the large-period regime we restrict to the case when the periodic-wave profile asymptotes a single solitary-wave profile 
and leave aside the case\footnote{There are yet more possibilities 
(involving fronts/kinks besides solitary-waves) but they may be 
thought as degenerate in the sense that they form a manifold of a 
smaller dimension. The two-bumps case is non-degenerate but was 
left aside in \cite{BMR2-I} as \emph{a priori} significantly different 
from the single-bump case dealt with here and there.} when the 
periodic wave profile is asymptotically obtained by gluing two 
pieces of distinct solitary wave profiles sharing the same endstate. From now on we focus on the case where $\urho^{(0)}$ is an infimum. Note that when there are two solitary waves with the same endstate/parameters they generate distinct branches of (single-bump) periodic waves thus they may be analyzed independently. Moreover we point out that the related analysis of the supremum case is completely analogous. The existence of a solitary wave of such a type is equivalent to the existence of $\urho^{(s)}>\urho^{(0)}$ such that 
\begin{align*}
\cW_\rho(\urho^{(s)};\ucx^{(0)},\uomp^{(0)},\umup^{(0)})&\,
=\,\umux^{(0)}\,,&
\d_\rho\cW_\rho(\urho^{(s)};\ucx^{(0)},\uomp^{(0)},\umup^{(0)})&
\,>\,0\,,
\end{align*} 
and
\[
\forall \rho\in (\urho^{(0)},\urho^{(s)}),\qquad
\cW_\rho(\rho;\ucx^{(0)},\uomp^{(0)},\umup^{(0)})\,\neq\,\umux^{(0)}\,,
\]
where $(\urho^{(0)},\umux^{(0)})=(\rho^{(0)},\mux^{(0)})(\ucx^{(0)},
\uomp^{(0)},\umup^{(0)})$. The situation is stable by perturbation of 
parameters $(\ucx^{(0)},\uomp^{(0)},\umup^{(0)})$. 
Assuming the latter, one deduces that to any $(\mux,\cx,\omp,\mup)$ 
sufficiently close to $(\umux^{(0)},\ucx^{(0)},\uomp^{(0)},
\umup^{(0)})$ and satisfying
\[
\mux<\mux^{(0)}(\cx,\omp,\mup),
\]
corresponds a unique --- up to rotational and spatial translations 
invariances --- periodic traveling wave with mass average and mass 
minimum close to $\rho^{(0)}(\cx,\omp,\mup)$. The large period limit 
denotes the asymptotics $\mux-\mux^{(0)}(\cx,\omp,\mup)\to 0$ and the 
\emph{large period regime} is the zone where 
$\mux-\mux^{(0)}(\cx,\omp,\mup)$ is sufficiently small but negative.\\
\begin{figure}[ht]
\begin{center}
 \includegraphics[scale=0.4]{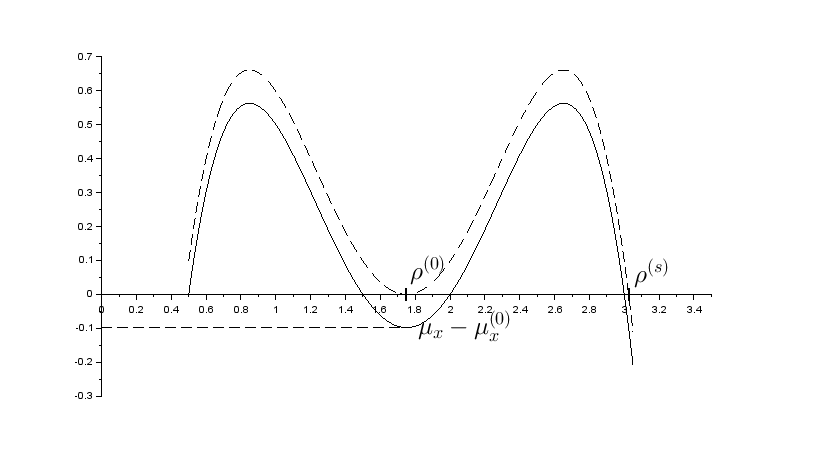}
 \includegraphics[scale=0.35]{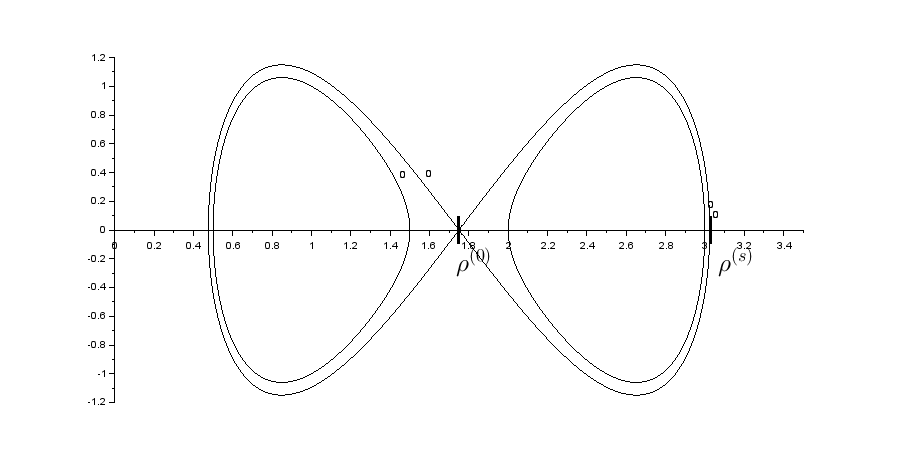}
 \end{center}
 \caption{{\bf Solitary-wave limit} with two branches with same parameters. The upper graph is the graph of $\mu_x-\cW_\rho$ as a function of $\rho$. The lower graph is, in the $(\rho,\tfrac{\dD \rho}{\dD x})$ phase plane, the level set defined by \eqref{e:profile-rho}. In both graphs we superimpose images corresponding to parameters of the solitary-wave limit and nearby parameters corresponding to periodic waves of a large period. The curves of interest are the right-hand ones.} 
\label{phasesol}
 \end{figure}
From the point of view of solitary waves themselves, it is actually 
both more natural and more convenient to keep a parametrization by 
$(\cx,\rho_{(0)},\kp)$ rather than by $(\cx,\omp,\mup)$, with 
$\rho_{(0)}$ the endstate. This is 
consistent with the fact that variations in the endstate (thus in 
$(\rho,\kp)$) play no role in the classical stability analysis of 
solitary waves (under localized perturbations). Assuming as above 
that there is a $\urho^{(s)}$ associated with 
$(\ucx^{(0)},\urho_{(0)},\ukp^{(0)})$, one deduces that for 
any $(\cx,\rho_{(0)},\kp)$ sufficently close to 
$(\ucx^{(0)},\urho^{(0)},\ukp^{(0)})$ there exists 
$\rho^{(s)}=\rho^{(s)}(\cx,\rho_{(0)},\kp)$ close to $\urho^{(s)}$ 
such that
\begin{align*}
\cW_\rho(\rho^{(s)};\cx,\omp^{(0)},\mup^{(0)})&\,=\,(\mux)_{(0)}\,,&
\d_\rho\cW_\rho(\rho^{(s)};\cx,\omp^{(0)},\mup^{(0)})&\,>\,0\,,
\end{align*} 
and
\[
\forall \rho\in (\rho_{(0)},\rho^{(s)}),\qquad
\cW_\rho(\rho;\cx,\omp^{(0)},\mup^{(0)})\,\neq\,(\mux)_{(0)}\,,
\]
where $(\omp^{(0)},\mup^{(0)})=(\omp^{(0)},\mup^{(0)})(\cx,\rho^{(0)},
\kp)$ is defined implicitly by
\[
(\rho_{(0)},\kp)
=(\rho^{(0)},\kp^{(0)})(\cx,\omp^{(0)},\mup^{(0)})
\]
and $(\mux)_{(0)}=(\mux)_{(0)}(\cx,\rho_{(0)},\kp):=\mux^{(0)}
(\cx,\omp^{(0)}(\cx,\rho_{(0)},\kp),\mup^{(0)}(\cx,\rho_{(0)},\kp))$. 
The mass of the corresponding solitary-wave profile 
$\rho_{(s)}=\rho_{(s)}(\,\cdot\,;\cx,\rho_{(0)},\kp)$ is then obtained by 
solving
\[
\frac{\kappa(2\,\rho_{(s)})}{2\,\rho_{(s)}}
\frac{\dd^2 \rho_{(s)}}{\dd x\,{}^2}
\,=\,
-\left(\frac{\kappa'(2\,\rho_{(s)})}{2\,\rho_{(s)}}
-\frac{\kappa(2\,\rho_{(s)})}{4\,\rho_{(s)}^2}\right)
\left(\frac{\dd \rho_{(s)}}{\dd x}\right)^2
-\d_\rho\cW_\rho(\rho_{(s)};\cx,\omp^{(0)},\mup^{(0)})\,,
\]
with $\rho_{(s)}(\,0\,;\cx,\rho_{(0)},\kp)=\rho^{(s)}
(\cx,\rho_{(0)},\kp)$. Then the unscaled profile $\cV_{(s)}=\cV_{(s)}
(\,\cdot\,;\cx,\rho_{(0)},\kp)$ is obtained through\footnote{The choice 
of the point where the value $\rho^{(s)}(\cx,\rho,\kp)$ is achieved (resp. of $\theta_{(s)}(0;\cx,\rho,\kp)$), quotients the invariance by spatial (resp. rotational) translation.}
\[
\cV_{(s)}\,=\,\sqrt{2\,\rho_{(s)}}\,\eD^{\theta_{(s)}\,\bJ}(\beD_1)\,,
\qquad\qquad
\theta_{(s)}(x)\,=\,\int_0^x \nu(\rho_{(s)}(\,\cdot\,;
\cx,\rho_{(0)},\kp);\cx,\mup^{(0)}(\cx,\rho,\kp))\,.
\]
Stability conditions are expressed in terms of
\be\label{def:action-s}
\Theta_{(s)}(\cx,\rho,\kp):= \int_{-\infty}^{\infty} 
\left(\Ham_0[\cV_{(s)}]+\cx\Impulse[\cV_{(s)}]
-\omp^{(0)}\,\Mass[\cV_{(s)}]
-\mup^{(0)}\,\frac{\Impulse[\cV_{(s)}]}{\Mass[\cV_{(s)}]}
+(\mux)_{(0)}\right)\,.
\ee

Concerning the small amplitude limit, though this is less crucial, 
at some point it will also be convenient to adopt a parametrization 
of limiting harmonic wavetrains by $(\kx,\rho_{(0)},\kp)$ (rather than 
by $(\cx,\omp,\mup)$). Our starting point was a parametrization by 
$(\cx,\rho_{(0)},\kp)$ so that we only need to examine the invertibility of 
the relation $\cx\mapsto 1/\Xx^{(0)}$ at fixed $(\rho,\kp)$. The 
equation to invert is
\[
\d_\rho^2\cW_\rho(\rho;\cx,\omp,\mup)
\,=\,\frac12\frac{\kappa(2\,\rho)}{2\rho}\,(2\pi\,\kx)^2\,.
\]
with $(\omp,\mup)$ associated with $(\cx,\rho,\kp)$ through
\begin{align}\label{eq:harmonic-omp-mup}
\kp
&=\nu(\rho;\cx,\umup)\,,&
0&=\d_\rho\cW_\rho(\rho;\cx,\omp,\mup)\,.
\end{align}
Straightforward computations detailed in \cite[Appendix~A]{BMR2-II} show that
\begin{align*}
\d_\rho^2\cW_\rho(\rho;\cx,\omp,\mup)
&\,=\,\frac{1}{\kappa(2\,\rho)\,2\rho}
\,\det(\bB\Hess H^{(0)}(\rho,\kp)+\cx\,\I_2)\\
-2\,\kappa(2\,\rho)\,2\rho\,\d_\rho\nu(\rho;\cx,\mup)&\,=\,
2\cx+\Tr(\bB\Hess H^{(0)}(\rho,\kp))
\end{align*}
where 
\begin{align}\label{e:zerodisp}
\bB&:=\bp0&1\\1&0\ep\,,&
H^{(0)}(\rho,v):=\kappa(2\,\rho)\rho\,v^2+W(2\,\rho)\,.
\end{align}
Let us stress incidentally that $H^{(0)}$ is the zero dispersion limit of the Hamiltonian $H_0$ of the hydrodynamic formulation of the Schr\"odinger equation and $\bB$ is the self-adjoint matrix involved in this formulation. As a result, if $\d_\rho\nu(\urho^{(0)};\ucx^{(0)},\umup^{(0)})\neq0$ then locally one may indeed parametrize waves by $(\kx,\rho,\kp)$ and we shall denote
\begin{align*}
\cx&\,=\,\cx^{(0)}(\kx,\rho,\kp)\,,&
\omx^{(0)}(\kx,\rho,\kp)&:=\,-\kx\,\cx^{(0)}(\kx,\rho,\kp)\,,
\end{align*}
the corresponding harmonic phase speed and associated spatial time frequency.

\subsection{General plane waves}\label{s:profile-general}

We now explain how to extend the foregoing analysis to more general plane waves in the form \eqref{def:wave-general}. So far, we have discussed explicitly the case when $\bkp$ and $\bkx$ point in the direction of $\beD_1$. The main task is to show how to reduce to the case when $\bkp$ and $\bkx$ are colinear, that is, when $\tkp=0$.

Let us first observe that for any vector $\tkp$, the frame change
\[
\bU(t,\bfx)
\,=\,
\eD^{\tkp\cdot \bfx\,\bJ}\,\tbU(t,\bfx)\,,
\]
changes \eqref{e:ab} into
\begin{align}
\d_t\tbU&\,=\,\bJ\delta\Hamkp[\tbU]\,,\nonumber\\
\Hamkp[\tbU]&\,:=\,\Ham_0(\tbU,(\nabla_\bfx+\tkp\bJ)\tbU)\label{e:moving-kp}\\
&\,\,=\,
\Ham_0(\tbU,\nabla_\bfx\tbU)+\frac12\,\|\tbU\|^2\,\kappa(\|\tbU\|^2)\,\|\tkp\|^2
\,+\,\kappa(\|\tbU\|^2)\,\tkp\cdot\,\bImpulse[\tbU]\,.\nonumber
\end{align}

As a consequence, if one is simply interested in analyzing the structure of waves or the behavior of solutions arising from longitudinal perturbations or more generally from perturbations depending only on directions orthogonal to $\tkp$, it is sufficient to fix $\ex$ and $\tkp$ (orthogonal to each other) and replace $\Ham_0$ with $\Ham_{0,\tkp}$ defined by
\[
\Ham_{0,\tkp}[\tbU]:=\Ham_0(\tbU,\nabla_\bfx\tbU)+\frac12\,\|\tbU\|^2\,\kappa(\|\tbU\|^2)\,\|\tkp\|^2
\]
or equivalently to replace $W$ with $W_{\tkp}$ defined through
\[
W_{\tkp}(\alpha):=W(\alpha)\,+\,\frac12\,\alpha\,\kappa(\alpha)\,\|\tkp\|^2\,.
\]

With this point of view, all quantities manipulated in previous subsections of the present section should be thought as implicitly depending on $\ex$ and $\tkp$. Note however that actually they do not depend on $\ex$ and depend on $\tkp$ only through $\|\tkp\|^2$. In particular their first-order derivatives with respect to $\tkp$ vanish at $\tkp=0$.

To prepare the analysis of stability under general perturbations, let us make explicit the relations defining constants of integration and averaged quantities for general plane waves taken in the form
\[
\bU(t,\bfx)\,=\,
\eD^{\left(\tkp\cdot(\bfx-\ucx\,\uex\,t)+\uomp\,t\right)\bJ}\,
\cV(\ex\,\cdot(\bfx-\ucx\,\ex\,t))\,,
\]
generalizing \eqref{def:wave-unscaled}. 
We still have
\[
\mup=\Sp[\cV]+\cx\Mass[\cV]
\]
but $\Sp$ should be taken as
\[
\Sp[\cV]\,=\,\bJ\cV\cdot \left(\ex\cdot\nabla_{\bU_{\bfx}}\right)\Ham_{0,\tkp}(\cV,\ex\,\d_\zeta\cV)\,.
\]
Likewise 
\[
\mux\,=\,\Sx[\cV]+\omp\Mass[\cV]
\]
with
\[
\Sx[\bV]\,=\,-\Ham_{0,\tkp}(\cV,\ex\,\d_\zeta\cV)
+\d_\zeta\cV\,\cdot\left(\ex\cdot\nabla_{\bU_{\bfx}}\right)\Ham_{0,\tkp}(\cV,\ex\,\d_\zeta\cV)\,.
\] 
This implies that
\begin{align*}
\bImpulse[\bU]&=\ex\cdot\,\bImpulse(\cV,\ex\,\d_\zeta \cV)\,\ex
+\Mass[\cV]\,\tkp\,,\\
\bJ\bU\cdot\nabla_{\bU_{\bfx}}\Ham_0[\bU]
&=\ex \,\left(\mup-\cx\Mass[\cV]\right)
+\tkp\,\kappa(\|\cV\|^2)\,2\,\Mass[\cV]\,,\\
\frac12\bJ\bU\cdot\bJ\delta\Ham_0[\bU]-\Ham_0[\bU]
&=\mux\,-\,\cx \ex\cdot\,\bImpulse(\cV,\ex\,\d_\zeta \cV)
-\kappa(\|\cV\|^2)\,\|\d_\zeta\cV\|^2
+\kappa'(\|\cV\|^2)\,\Mass[\cV]\,\|\tkp\|^2\,,\\
\bJ\delta\,\Impulse_j[\bU]\cdot\nabla_{\bU_{x_\ell}}\Ham_0[\bU]
&=\kappa(\|\cV\|^2)\,\left((\ex)_j\,\d_\zeta\cV+(\tkp)_j\,\bJ\cV\right)
\cdot\left((\ex)_\ell\,\d_\zeta\cV+(\tkp)_\ell\,\bJ\cV\right)\,,
\end{align*}
with right-hand terms evaluated at $(\ex\,\cdot(\bfx-\ucx\,\ex\,t))$.


\section{Structure of the spectrum}\label{s:structure}

Now we turn to gathering key facts about the spectrum of operators arising from linearization, in suitable frames, about periodic plane waves.

\subsection{The Bloch transform}\label{s:Bloch}

Our first observation is that, thanks to a suitable 
integral transform, the spectrum of the linearized operator $\cL$ defined in \eqref{def:L}  may be studied through normal-mode analysis.

To begin with, we introduce a suitable Fourier-Bloch 
transform, as a mix of a Floquet/Bloch transform in the $x$ variable and the
Fourier transform in the $\bfy$ variable:
\be\label{Bloch}
\cB(g)(\xi,x,\bfeta)\ =\ \check{g}(\xi,x,\bfeta)\ :=\ 
\sum_{j\in\Z} \eD^{\ii\,2j\pi x}\ \widehat{g}(\xi+2j\pi,\bfeta),
\ee
where $\widehat{g}$ is the usual Fourier transform normalized so that
for $\bfx=(x,\bfy)$
\begin{align*}
\cF(g)(\xi,\bfeta)=\hat g(\xi,\bfeta)&\ :=\ \frac{1}{(2\pi)^d} 
\int_\R \eD^{-\iD\xi x-\iD\bfeta\cdot\bfy} g(\bfx)\ \dD \bfx,\\
g(\bfx)&\ =\ \int_\R \eD^{\iD\xi x+\iD\bfeta\cdot\bfy}\ \hat{g}(\xi,\bfeta)\ 
\dD\xi\,\dD\bfeta.
\end{align*}
Obviously, $\check{g}(\xi,\cdot,\bfeta)$ is periodic of period one
for any $(\xi,\,\bfeta)$, 
that is,
$$
\forall x\in\R,\qquad\check{g}(\xi,x+1,\bfeta)\ =\ \check{g}(\xi,x,\bfeta).
$$
As follows readily from \eqref{Bloch} and basic Fourier theory, 
$(2\pi)^{d/2}\,\cB$ is a total 
isometry from $L^2(\R^d)$ to $L^2((-\pi,\pi)\times(0,1)\times\R^{d-1})$,
and it satisfies the inversion formula
\be\label{inverse-Bloch}
g(\bfx)\ =\ \int_{-\pi}^\pi\int_{\R^{d-1}} \eD^{\iD\xi x+\iD\bfeta\cdot\bfy}\ \check{g}(\xi,x,\bfeta)\ \dD\xi\,\dD\bfeta\,.
\ee
The Poisson summation formula provides an alternative equivalent 
formula for \eqref{Bloch}
$$
\check{g}(\xi,x,\bfeta)\ =\ \frac{1}{2\pi}\sum_{\ell\in\Z}
\eD^{-\ii\xi (x+\ell)}\cF_\bfy(g)(x+\ell,\bfeta)
$$
where $\cF_\bfy$ denotes the Fourier transform in the $\bfy$-variable 
only. 

The key feature of the transform $\cB$ is that in some sense it diagonalizes differential operators whose coefficients do not depend on $\bfy$ and are $1$-periodic in $x$. For large classes of such operators $\cP=P(x,\partial_x,\nabla_\bfy)$, stands
\[
\cB(\cP u)(\xi,x,\bfeta)=P(x,\partial_x+\ii\xi,\ii\bfeta)\,\cB(u)(\xi,x,\bfeta)
\]
so that the action of such operators on functions defined on $\R^d$ is reduced to the action of $\cP_{\xi,\bfeta}=P(x,\partial_x+\ii\xi,\ii\bfeta)$ on $1$-periodic functions, parametrized by $(\xi,\bfeta)$.

In particular, for $\cL$ as in \eqref{def:L} we do have
\[
(\cL g)(\bfx)\ =\ \int_{-\pi}^\pi\int_{\R^{d-1}}\eD^{\iD\xi \bfx
+\iD\bfeta\cdot\bfy}\ (\cL_{\xi,\bfeta}\check{g}(\xi,\cdot,\bfeta))(x)
\ \dD\xi\,\dD\bfeta,
\]
where $\cL_{\xi,\bfeta}$ acts on $1$-periodic functions and inherits from $\Ham=\Ham^x+\Ham^\bfy$ the splitting 
\begin{align*}
\cL_{\xi,\bfeta}&:=\,\cL^x_\xi+\cL^\bfy_\bfeta\,,&
\text{ with }&&
\cL^\bfy_\bfeta&:=\|\bfeta\|^2\,\kappa(\|\bU\|^2)\,\bJ\,,
\end{align*}
and $\cL^x_\xi$ given by
\begin{align*}
\cL^x_\xi\bV
\,=\,\bJ\,\Big(&\dD_{(\bU,\bU_x)}(\nabla_\bU\Ham)(\ubU,\nabla_{\bfx}\ubU)(\bV,(\d_x+\ii\xi)\bV)\\
&\,-(\d_x+\ii\xi)\,\left(\dD_{(\bU,\bU_x)}(\nabla_{\bU_x}\Ham)(\ubU,\nabla_\bfx\ubU)(\bV,(\d_x+\ii\xi)\bV)\right)\Big)\,.
\end{align*}
On\footnote{We insist on the substrict ${}_{\per}$ to emphasize that 
corresponding domains involve $H^s_{\per}((0,1))$ spaces, thus 
effectively encoding periodic boundary conditions when $s>1/2$. 
Notation "$\cL_{\xi,\bfeta}^{\per}$ acts on $L^2((0,1))$" would be 
mathematically more accurate but more cumbersome.} 
$L_{\per}^2((0,1))$ each $\cL_{\xi,\bfeta}$ has compact resolvent 
and  depends analytically on $(\xi,\bfeta)$ in the strong resolvent 
sense.

It is both classical and relatively straightforward to derive from the 
latter and the isometry of $(2\pi)^{d/2}\,\cB$ that the spectrum of 
$\cL$ on $L^2(\RR)$ coincides with the union over 
$(\xi,\bfeta)\in [-\pi,\pi]\times\R^{d-1}$ of the spectrum of each 
$\cL_{\xi,\bfeta}$ on  $L_{\per}^2((0,1))$. For some more details see 
for instance \cite[p.30-31]{R}.

\subsection{Linearizing Madelung's transformation}\label{s:linearized-EK}

We would like to point out here how the analysis of Section~\ref{s:EK}
may be extended to the spectral level. We stress that working with 
Bloch-Fourier symbols $\cL_{\xi,\bfeta}$ provides crucial 
simplifications in the arguments.

Firstly we observe that linearizing \eqref{e:MTrho}-\eqref{e:MTv} 
provides all the necessary algebraic identities. Secondly we note 
that applying a Bloch-Fourier transform to both sides of the foregoing 
identities yields the required algebraic conjugations between 
respective Bloch-Fourier symbols. 

To go beyond algebraic relations, we start with a few notational or 
elementary considerations. 
\begin{enumerate} 
\item From elementary elliptic regularity arguments it follows that 
the $L^2_{\per}$-spectrum of each $\cL_{\xi,\bfeta}$ coincides with 
its $H^1_{\per}$-spectrum. 
\item With $L^2_{\rm \Curl_{\xi,\bfeta}}((0,1))$ denoting the space 
of $L^2((0,1);\C)^d$-functions $\bfv$ such that\footnote{The condition means: $(\d_x+\iD\xi)v_j\,=\,\iD\,\eta_{j-1}v_1$ for $2\leq j\leq d$ and $\eta_{j}v_\ell=\eta_{\ell}v_j$ for $1\leq j,\,\ell\leq d-1$. When $d=1$, $L^2_{\rm \Curl_{\xi,\bfeta}}((0,1))=L^2((0,1);\C)$.} 
\[
\begin{pmatrix}(\d_x+\iD\xi)\\\iD\bfeta\end{pmatrix}\wedge\bfv\,=\,0\,,
\]
we observe that when $(\xi,\bfeta)\in[-\pi,\pi]\times\R^{d-1}
\setminus\{(0,0)\}$, 
\[
\cI_{\xi,\bfeta}\,:\quad H^1_{\per}((0,1))\longrightarrow 
L^2_{\rm \Curl_{\xi,\bfeta}}((0,1))\,,\qquad
\theta\mapsto \begin{pmatrix}(\d_x+\iD\xi)\\\iD\bfeta\end{pmatrix}
\theta,
\]
is a bounded invertible operator. 
\item The linearization of the relation
\[
\bU=\eD^{-\frac{\kp}{\kx}(\cdot)\bJ}\cU(\rho(\cdot),\theta(\cdot)),
\] 
at $\ucU$, $(\urho,\utheta)$, is given by
\[
\mathfrak{m}\,:\quad H^1_{\per}((0,1);\C^2)\longrightarrow 
H^1_{\per}((0,1);\C)^2\,,\quad
\bV\mapsto \left(\ucU\cdot\bV,\ \frac{\bJ\ucU}{2\urho}\,\cdot\bV\right),
\]
and is bounded and invertible with inverse
\[
\mathfrak{m}^{-1}\,:\quad
H^1_{\per}((0,1);\C)^2\longrightarrow H^1_{\per}((0,1);\C^2)\,,\quad
(\rho,\theta)\mapsto \rho\,\frac{\ucU}{2\,\urho}\,+\theta\,\bJ\ucU\,.
\]
\end{enumerate}
Considering $\cL_{\xi,\bfeta}$ as an operator on $H^1_{\per}$, and 
denoting $L_{\xi,\bfeta}$ the corresponding Bloch-Fourier symbol for
the associated Euler--Korteweg system \eqref{e:absEK}, we deduce 
when $(\xi,\bfeta)\in[-\pi,\pi]\times\R^{d-1}\setminus\{(0,0)\}$, 
the conjugation
\[
\cL_{\xi,\bfeta}
\,=\,
\mathfrak{m}^{-1}
\begin{pmatrix}\I&0\\0&\cI_{\xi,\bfeta}^{-1}\end{pmatrix}
L_{\xi,\bfeta}
\begin{pmatrix}\I&0\\0&\cI_{\xi,\bfeta}\end{pmatrix}\mathfrak{m}
\,.
\]

Of course, the conjugation yields identity of spectra including 
algebraic multiplicities but also identity of detailed algebraic 
structure of each eigenvalue. When $d=1$, by continuity of the 
eigenvalues with respect to $\xi$, one also concludes that $\cL_{0}$ 
and $L_{0}$ share the same spectrum including algebraic multiplicities,
but algebraic structures may differ (and as stressed below in general 
they do). Note that to go from spectral to linear stability it is 
actually crucial to examine semi-simplicity of eigenvalues. 

For this reason, we focus now a bit more on the case 
$(\xi,\bfeta)=(0,0)$. To begin with, denoting $L^2_{0}((0,1))$  
the space of $L^2((0,1);\C)^d$-functions $\bfv$ of the form 
\[
\begin{pmatrix}v\\0\end{pmatrix}\,,\quad\textrm{with }\int_0^1v\,=\,0\,,
\]
we observe that 
\[
\cI^{(0)}\,:\quad H^1_{\per}((0,1))\longrightarrow L^2_{0}((0,1))\,,\qquad
\theta\mapsto \begin{pmatrix}\d_x\\0\end{pmatrix}\theta,
\]
is a bounded invertible operator. Moreover we point that $L_{0,0}$ 
leaves $H^1_{\per}((0,1))\times L^2_{0}((0,1))$ invariant\footnote{For 
an unbounded operator $A$ defined on $X$ with domain $D$, we say that 
$Y$, a subspace of $X$, is left invariant by $A$ if 
$A(D\cap Y)\subset Y$ and in this case $A_{|Y}$ is defined on $Y$ with 
domain $D\cap Y$.} and its restriction is conjugated to $\cL_{0,0}$ 
through
\[
\cL_{0,0}
\,=\,
\mathfrak{m}^{-1}
\begin{pmatrix}\I&0\\0&(\cI^{(0)})^{-1}\end{pmatrix}
(L_{0,0})_{|\,H^1_{\per}((0,1))\times L^2_{0}((0,1))}
\begin{pmatrix}\I&0\\0&\cI^{(0)}\end{pmatrix}\mathfrak{m}\,.
\]
Denoting $\pi_0$ the orthogonal projector of 
$H^1_{\per}((0,1))\times L^2_{\rm \Curl_{0,0}}((0,1))$ on 
$H^1_{\per}((0,1))\times L^2_{0}((0,1))$, we also note that 
$(\I-\pi_0)\,L_{0,0}\,(\I-\pi_0)$ is identically zero and 
$\pi_0\,L_{0,0}\,(\I-\pi_0)$ is bounded. As a conclusion, one 
derives when $\lambda$ is non zero and does not belong to the 
spectrum of $\cL_{(0,0)}$ or equivalently, when $\lambda$ is 
non zero and does not belong to the spectrum of $L_{(0,0)}$
\begin{align*}
(\lambda \I-L_{(0,0)})^{-1}
&\,=\,
\frac{1}{\lambda}\,(\I-\pi_0)\\
&\quad+\begin{pmatrix}\I&0\\0&\cI^{(0)}\end{pmatrix}\mathfrak{m}^{-1}
(\lambda \I-\cL_{(0,0)})^{-1}\,
\mathfrak{m}
\begin{pmatrix}\I&0\\0&(\cI^{(0)})^{-1}\end{pmatrix}
\left(
\pi_0+\frac{1}{\lambda}\pi_0\,L_{0,0}\,(\I-\pi_0)
\right),
\end{align*}
so that for nonzero eigenvalues the algebraic 
structures\footnote{Recall that the algebraic structure of an 
eigenvalue $\lambda_0$ of an operator $A$ is read on the singular 
part of $\lambda \mapsto (\lambda\I-A)^{-1}$ at $\lambda=\lambda_0$.} 
of $\cL_{(0,0)}$ and $L_{0,0}$ are the same.

As we comment further below, in general $0$ is an eigenvalue of 
$\cL_{(0,0)}$ of algebraic multiplicity $4$ with two Jordan blocks 
of height $2$, whereas, when $d=1$, $0$ is an eigenvalue of $L_{0}$ 
of algebraic multiplicity $4$ with geometric multiplicity $3$ and 
one Jordan block of height $2$.

\subsection{The Evans function}\label{s:Evans-def}

Since each $\cL_{\xi,\bfeta}$ acts on functions of a scalar variable, 
it is convenient to analyze their spectra by focusing on spatial 
dynamics, rewriting spectral problems in terms of ODEs of the spatial 
variable. Adapting the construction of Gardner 
\cite{Gardner-structure-periodic} to the situation at hand, 
this leads to the introduction of a suitable Evans function.

To keep spectral ODEs as simple as possible, it is expedient to 
work with unscaled equations as in Section~\ref{s:profile}. 
Explicitly, with notation from Section~\ref{s:profile}, for 
$\lambda\in\C$ and $\bfeta\in\R^{d-1}$, we consider 
$R(\cdot,x_0;\lambda,\bfeta)$ the solution operator of the first-order 
$4$-dimensional differential operator canonically associated with 
the second-order $2$-dimensional operator
$\bJ \Hess \Hamp[\ucV]+\|\bfeta\|^2\,\kappa(\|\ucV\|^2)\bJ
-\lambda $. 
Note that $R(x_0,x_0;\lambda,\bfeta)=\I_4$. Accordingly we 
introduce the Evans function
\begin{equation}\label{e:evans}
D^{x_0}_\xi(\lambda,\bfeta)=\det\left(R(x_0+\uXx,x_0;\lambda,\bfeta)
-\eD^{\iD\xi}\diag(\eD^{\uxip\bJ},\eD^{\uxip\bJ})\right)\,.
\end{equation}
The choice of $x_0$ is immaterial, we shall set $x_0=0$ and drop 
the corresponding superscript in the following.

The backbone of the Evans function theory is that $\lambda_0$ belongs 
to the spectrum of $\cL_{\xi,\bfeta}$ if and only if $\lambda_0$ is a
root of $D_\xi(\cdot,\bfeta)$ and that its (algebraic) multiplicity 
as an eigenvalue of $\cL_{\xi,\bfeta}$ agrees with its multiplicity 
as a root of $D_\xi(\cdot,\bfeta)$. The first part of the claim is a 
simple reformulation of the fact that the spectrum of 
$\cL_{\xi,\bfeta}$ contains only eigenvalues, whereas the second 
part may be derived from the expression of resolvents of 
$\cL_{\xi,\bfeta}$ at $\lambda$ in terms of solution operators 
$R(\cdot,\cdot;\lambda,\bfeta)$ and the characterization/definition 
of algebraic multiplicity at $\lambda_0$ as the rank of the residue 
at $\lambda_0$ of the resolvent map.

To a large extent, the benefits from using an Evans function instead 
of directly studying spectra are the same as those arising from the consideration of characteristic polynomials to study finite-dimensional spectra.

\subsection{High-frequency analysis}\label{s:high}

It is quite straightforward to check that when $|\Re(\lambda)|$ is
sufficiently large, $\lambda$ does not belong to the spectrum of 
any $\cL_{\xi,\bfeta}$. When $\lambda$ is real and $\xi\in\{0,\pi\}$, $D_\xi(\lambda,\bfeta)$ is real-valued and we would like to go further and determine its sign when $(\lambda,\bfeta)$ is sufficiently large with $\lambda$ real. This is useful in order to derive instability 
criteria based on the Intermediate Value Theorem.

Since the principle part of $\cL_{\xi,\bfeta}$ has non-constant coefficients, this is not completely trivial but one may reduce the computation to the constant-coefficient case by a homotopy argument similar to the one in \cite{Benzoni-Mietka-Rodrigues}.

\begin{proposition}\label{p:hig-freq}
Let $\ucV$ be an unscaled wave profile (in the sense of 
\eqref{e:unscaled-profile}). There exists $R_0>0$ such that 
for any $(\lambda,\bfeta)\in\R\times\R^{d-1}$ satisfying 
$|\lambda|+\|\bfeta\|^2\geq R_0$ we have $D_0(\lambda,\bfeta)>0$ 
and $D_\pi(\lambda,\bfeta)>0$.
\end{proposition}

\begin{proof}
An elementary Lax-Milgram type argument shows that when 
$(\lambda,\bfeta)$ is sufficiently large (with $\lambda$ real) 
independently of $\theta\in[0,1]$, $\lambda$ does not belong to 
the spectrum of 
\[
\cL^{(\theta)}_{0,\bfeta}:=(1-\theta)\cL_{0,\bfeta}
+\theta\,\bJ(-(\ukx\,\d_x+\ukp\bJ)^2+\|\bfeta\|^2)
\]
on $L_{\per}^2((0,1))$, for any $\theta\in[0,1]$. The needed 
estimates stem from the form 
\[
\cL^{(\theta)}_{0,\bfeta}
\,=\,
\bJ((\theta\,+(1-\theta)\kappa(\|\ucU\|^2))\,(-\ukx^2\d_x^2
+\|\bfeta\|^2))+\textrm{lower order terms independent of }
\lambda\textrm{ and }\bfeta
\]
and the fact that $\min(\kappa(\|\ucU\|^2))>0$. Indeed, for 
some positive constants $c$, $C$ independent of 
$(\lambda,\bfeta,\theta)\in \R\times\R^{d-1}\times[0,1]$
\begin{align*}
\langle (\bJ \bV-\sign(\lambda)\bV,(\cL^{(\theta)}_{0,\bfeta}
-\lambda)\bV)\rangle_{L^2}
&\geq c\,(\|\bV\|_{H^1}^2+(\|\bfeta\|^2+|\lambda|)\|\bV\|^2)
-C\,\|\bV\|_{H^1}\,\|\bV\|_{L^2}\\
&\geq \frac{c}{2}\,(\|\bV\|_{H^1}^2+(\|\bfeta\|^2+|\lambda|)\|\bV\|^2)
\end{align*}
provided that $(\lambda,\bfeta)$ is sufficiently large and 
$\bV\in H^2_{\per}((0,1))$. A similar bound holds for the adjoint 
problem.

For corresponding Evans functions, this implies that 
$(\lambda,\bfeta,\theta)\mapsto D^{(\theta)}_0(\lambda,\bfeta)$ 
has a constant sign on
\[
\left\{\,(\lambda,\bfeta,\theta)\in\R\times\R^{d-1}\times[0,1]\,;\,
|\lambda|+\|\bfeta\|\geq R_0
\,\right\}
\]
for some $R_0>0$. This sign is easily evaluated by considering 
either $D^{(1)}_0(\lambda,0)$ when $\lambda$ is large 
or\footnote{When $d=1$, this requires first to embed artificially 
the spectral problem at hand in a corresponding higher-dimensional 
problem.} $D^{(1)}_0(0,\bfeta)$ when $\|\bfeta\|$ is large. The 
foregoing computations can be made even more explicit running first 
another homotopy argument moving $\bJ(-(\ukx\,\d_x+\ukp\bJ)^2
+\|\bfeta\|^2)$ to $\bJ(-\ukx^2\,\d_x^2+\|\bfeta\|^2)$ thus 
reducing to $\uxip=0$, in this case we have
$D^{(1)}_0(0,\bfeta)=(\eD^{\|\bfeta\|\uXx}-1)^2\,(\eD^{-\|\bfeta\|\uXx}
-1)^2$.

The study of $D_\pi(\lambda,\bfeta)$ is nearly identical and thus 
omitted.
\end{proof}

\subsection{Low-frequency analysis}\label{s:low}

Now we turn to the derivation of an expansion of 
$D_\xi(\lambda,\bfeta)$ when $(\lambda,\xi,\bfeta)$ is small. We 
begin with a few preliminary remarks to prepare such an expansion. 

For the sake of brevity in algebraic manipulations, 
we introduce notation 
\[
[A]_0:=A(\uXx)-\eD^{\uxip\bJ}A(0)\,.
\]
Let us observe that if $\cA[\bU]=\cA(\bU,\,\bU_x)$ is rotationally invariant, 
\[
L\cA[\bU]\bV=L\cA[\eD^{\vphip\bJ}\bU]\eD^{\vphip\bJ}\bV
\]
for any $\vphip\in \R$, hence if $\cV(\cdot+\uXx)=\eD^{\uxip\bJ}\cV(\cdot)$ 
\begin{align*}
(L\cA[\cV]\bfpsi)(0)&\,=\,
\dD_{(\bU,\bU_x)}\cA(\cV(\uXx),\cV_x(\uXx))(\eD^{\uxip\bJ}\,
\bfpsi(0),\eD^{\uxip\bJ}\,\bfpsi_x(0))\,,\\
(L\cA[\cV]\bfpsi)(\uXx)-(L\cA[\cV]\bfpsi)(0)&\,
=\,\dD_{(\bU,\bU_x)}\cA(\cV(\uXx),\cV_x(\uXx))([\bfpsi]_0,
[\bfpsi_x]_0)\,.
\end{align*}

All relevant quantities depend on $\bfeta$ 
only through $\|\bfeta\|^2$ and a wealth of information on the 
regime $(\lambda,\xi,\bfeta)$ small --- used repeatedly below 
without mention --- is obtained by differentiating 
\eqref{e:unscaled-profile}, \eqref{e:massu} and \eqref{e:momu}; 
for $a=\omp,\cx,\mup,\mux$,
\begin{align}
\label{e:dPer}
[\d_a \ucV]_0&=\d_a\uxip\,\eD^{\uxip\bJ}\bJ\ucV(0)
-\d_a\uXx\,\eD^{\uxip\bJ}\ucV_x(0)\,,\\
(L(\Sp+\ucx \Mass)[\ucV]\d_a \ucV)(0)&
=\d_a\umup-\d_a\ucx\,\Mass[\ucV](0)\,,\\
(L(\Sx+\uomp \Mass)[\ucV]\d_a \ucV)(0)
&=\d_a\umux-\d_a\uomp\,\Mass[\ucV](0)\,.
\end{align}

At last, we derive from Appendix~\ref{s:Noether} that if 
\[
\lambda \bfpsi\,=\,\bJ \Hess \Hamp[\ucV]\bfpsi+\|\bfeta\|^2\,\kappa(\|\ucV\|^2)\bJ\bfpsi
\]
then 
\begin{align}
\label{e:Lmass}
\lambda L\Mass[\ucV]\bfpsi
&=\d_x\left(L(\Sp+\ucx \Mass)[\ucV]\bfpsi\right)
-\|\bfeta\|^2\,\kappa(\|\ucV\|^2)\,\bJ\ucV\cdot\bfpsi\,,\\
\label{e:Lmom}
\lambda L\Impulse_1[\ucV]\bfpsi
&=\d_x\left(L(\Sx+\uomp \Mass)[\ucV]\bfpsi\right)
-\|\bfeta\|^2\,\kappa(\|\ucV\|^2)\,\ucV_x\cdot\bfpsi\\
\nonumber
&\quad 
+\d_x\left(\frac12\bJ\ucV\cdot\left(\lambda\bfpsi-\|\bfeta\|^2\,\kappa(\|\ucV\|^2)\bJ\bfpsi\right)\right)\,,
\end{align}
Moreover, as already mentioned in section \ref{s:jumps}
\[
\det(\dD_{\bU_x}(\Sp,\Sx)(\ucV(\uXx),\ucV_x(\uXx))
\,=\,(\kappa(\|\ucV(0)\|^2))^2\ \ucV(0)\cdot\ucV_x(0)\,.
\]

\bt\label{th:low-freq}
With notation from Section~\ref{s:profile}, consider an unscaled wave profile $\ucV$ such that $\ucV\cdot \ucV_x\not\equiv0$. Then the corresponding Evans function expands uniformly in $\xi\in[-\pi,\pi]$ as
\begin{align}
D_\xi(\lambda,\bfeta)
&=\det\begin{pmatrix}\lambda\,\Sigma_t-(\eD^{\iD\xi}-1)\I_4+\frac{\|\bfeta\|^2}{\lambda}\Sigma_{\bfy}\end{pmatrix}\\
&+\cO\left((|\lambda|+|\xi|+\|\bfeta\|^2)\,
(|\lambda|^2+|\xi|^2+\|\bfeta\|^2)\,(|\lambda|(|\lambda|+|\xi|)+\|\bfeta\|^2)\right)
\nonumber
\end{align} 
when $(\lambda,\bfeta)\to(0,0)$, with 
\begin{align*}
\Sigma_t
&=\begin{pmatrix}\d_a\uxip\\[0.25em]
\d_a\uXx\\[0.25em]
\int_0^{\uXx}L\Mass[\ucV]\d_a\ucV+\Mass[\ucV](0)\,\d_a\uXx\\[0.25em]
\int_0^{\uXx}L\Impulse_1[\ucV]\d_a\ucV
+\Impulse_1[\ucV](0)\,\d_a\uXx
\end{pmatrix}_{a=\omp,\cx,\mup,\mux}\,,\\
\Sigma_{\bfy}
&=\begin{pmatrix}0&0&0&0\\
0&0&0&0\\
\int_0^{\uXx}\kappa(\|\ucV\|^2)\,\|\ucV\|^2
&-\int_0^{\uXx}\kappa(\|\ucV\|^2)\,\bJ\ucV\cdot\ucV_x
&0&0\\[0.25em]
\int_0^{\uXx}\kappa(\|\ucV\|^2)\,\ucV_x\cdot\bJ\ucV
&-\int_0^{\uXx}\kappa(\|\ucV\|^2)\,\|\ucV_x\|^2
&0&0
\end{pmatrix}\,.
\end{align*}
\et

Note that the structure of $\Sigma_{\bfy}$ is consistent with the fact that there is actually no singularity in the low-frequency expansion of the Evans function, every power of $\|\bfeta\|^2/\lambda$ is balanced by a corresponding power of $\lambda$.

\begin{proof}
By a density argument on the point where the Evans function is considered, we may reduce the analysis to the case when $\ucV(0)\cdot \ucV_x(0)\neq0$.

Guided by rotation and translation invariance, we introduce
\begin{align*}
\Psi_{1}(\cdot;\lambda,\bfeta)&\,=\,R(\cdot,0;\lambda,\bfeta)\begin{pmatrix}
\bJ\ucV(0)\\
\bJ\ucV_x(0)
\end{pmatrix}&
\Psi_{2}(\cdot;\lambda,\bfeta)&\,=\,R(\cdot,0;\lambda,\bfeta)\begin{pmatrix}
\ucV_x(0)\\
\ucV_{xx}(0)
\end{pmatrix}\,,\\
\Psi_{3}(\cdot;\lambda,\bfeta)&\,=\,R(\cdot,0;\lambda,\bfeta)\begin{pmatrix}
\d_{\mup}\ucV(0)\\
\d_{\mup}\ucV_x(0)
\end{pmatrix}&
\Psi_{4}(\cdot;\lambda,\bfeta)&\,=\,R(\cdot,0;\lambda,\bfeta)\begin{pmatrix}
\d_{\mux}\ucV(0)\\
\d_{\mux}\ucV_{x}(0)
\end{pmatrix}\,,
\end{align*}
so that in particular
\begin{align*}
\Psi_{1}(\cdot\,;0,0)&\,=\,\begin{pmatrix}
\bJ\ucV\\
\bJ\ucV_x
\end{pmatrix}\,,&
\Psi_{2}(\cdot\,;0,0)&\,=\,\begin{pmatrix}
\ucV_x\\
\ucV_{xx}
\end{pmatrix}\\
\Psi_{3}(\cdot\,;0,0)&\,=\,\begin{pmatrix}
\d_{\mup}\ucV\\
\d_{\mup}\ucV_x
\end{pmatrix}\,,&
\Psi_{4}(\cdot\,;0,0)&\,=\,\begin{pmatrix}
\d_{\mux}\ucV\\
\d_{\mux}\ucV_{x}
\end{pmatrix}\,.
\end{align*}
Then we set $\Psi=\begin{pmatrix}\Psi_1&\Psi_2&\Psi_3&\Psi_4\end{pmatrix}$ and observe that from the computations in Section~\ref{s:jumps} stems
\[
D_\xi(\lambda,\bfeta)\,=\,(\kappa(\|\ucV(0)\|^2))^2\,\det\left([\Psi(\cdot;\lambda,\bfeta)]_0
-(\eD^{\iD\xi}-1)\diag(\eD^{\uxip\bJ},\eD^{\uxip\bJ})\Psi(0;\lambda,\bfeta)\right)\,.
\]
Note that each $\Psi_\ell(\cdot;\lambda,\bfeta)$ splits as $(\bfpsi_\ell(\cdot;\lambda,\bfeta),(\bfpsi_\ell)_x(\cdot;\lambda,\bfeta))$ for some $\bfpsi_\ell(\cdot;\lambda,\bfeta)$ and that 
\[
\lambda \bfpsi_\ell\,=\,\bJ \Hess \Hamp[\ucV]\bfpsi_\ell+\|\bfeta\|^2\,\kappa(\|\ucV\|^2)\bJ\bfpsi_\ell\,.
\]

We may now use the identities \eqref{e:Lmass}
\eqref{e:Lmom}) and perform line combinations so as to obtain that 
$(\ucV(0)\cdot \ucV_x(0))\times D_\xi(\lambda,\bfeta)$ coincides 
with the determinant of a matrix of the form
\begin{align*}\hspace{-4em}
\begin{pmatrix}[\bfpsi_\ell]_0-(\eD^{\iD\xi}-1)\,\eD^{\uxip\bJ}\bfpsi_\ell(0)\\[0.25em]
\lambda\,\int_0^{\uXx}L\Mass[\ucV](\bfpsi_\ell)
+\|\bfeta\|^2\,\int_0^{\uXx}\kappa(\|\ucV\|^2)\,\bJ\ucV\cdot\bfpsi_\ell
-(\eD^{\iD\xi}-1)\,(L(\Sp+\ucx \Mass)[\ucV]\bfpsi_\ell)(0)\\[0.25em]
\lambda\,\int_0^{\uXx}L\Impulse_1[\ucV](\bfpsi_\ell)
+\|\bfeta\|^2\,\int_0^{\uXx}\kappa(\|\ucV\|^2)\,\ucV_x\cdot\bfpsi_\ell\hspace{14em}\\
\hspace{4em}-(\eD^{\iD\xi}-1)\,
\left(\left(L(\Sx+\uomp \Mass)[\ucV]
+\left(\frac{\lambda}{2}\bJ\ucV-\|\bfeta\|^2\,\kappa(\|\ucV\|^2)\ucV\right)\cdot\right)\bfpsi_\ell\right)(0)
\end{pmatrix}_{\ell}
\end{align*}
in the limit $(\lambda,\bfeta)\to(0,0)$ (where we have left implicit the dependence of $\bfpsi_\ell$ on $(\lambda,\bfeta)$ for concision's sake). Then we observe that it follows from invariances by rotational and spatial translations that the first two columns of the foregoing matrix are of the form
\begin{align*}
\begin{pmatrix}\\
\cO(|\lambda|+|\xi|+\|\bfeta\|^2)\\
\cO(|\lambda|+|\xi|+\|\bfeta\|^2)\\
\cO(|\lambda|\,(|\lambda|+|\xi|)+\|\bfeta\|^2)\\
\cO(|\lambda|\,(|\lambda|+|\xi|)+\|\bfeta\|^2)
\end{pmatrix}
\end{align*}
when $(\lambda,\bfeta)\to(0,0)$ and that, as follows by comparing respective equations, both $\d_\lambda\bfpsi_1(\cdot;0,0)$ and $\d_{\omp}\ucV$ on one hand and $\d_\lambda\bfpsi_2(\cdot;0,0)$ and $-\d_{\cx}\ucV$ on another hand differ only by a linear combination of $\bfpsi_1(\cdot;0,0)$, $\bfpsi_2(\cdot;0,0)$, $\bfpsi_3(\cdot;0,0)$ and $\bfpsi_4(\cdot;0,0)$. 

Therefore from a direct expansion and a column manipulation one derives that 
\[
(-\ucV(0)\cdot \ucV_x(0))\times D_\xi(\lambda,\bfeta)
=\det\begin{pmatrix}C_1&C_2&C_3&C_4\end{pmatrix}
\]
with
\begin{align*}
C_1&=
\begin{pmatrix}\lambda\,[\d_{\omp}\ucV]_0-(\eD^{\iD\xi}-1)\,\eD^{\uxip\bJ}\bJ\ucV(0)\\
\lambda^2\,\int_0^{\uXx}L\Mass[\ucV]\d_{\omp}\ucV
+\|\bfeta\|^2\,\int_0^{\uXx}\kappa(\|\ucV\|^2)\,\|\ucV\|^2\\
\lambda^2\,\int_0^{\uXx}L\Impulse_1[\ucV]\d_{\omp}\ucV
+\|\bfeta\|^2\,\int_0^{\uXx}\kappa(\|\ucV\|^2)\,\ucV_x\cdot\bJ\ucV
\end{pmatrix}
+\begin{pmatrix}\cO(|\lambda|(|\lambda|+|\xi|)+\|\bfeta\|^2)\\
\cO((|\lambda|^2+\|\bfeta\|^2)(|\lambda|+|\xi|)+\|\bfeta\|^4)\\
\cO((|\lambda|^2+\|\bfeta\|^2)(|\lambda|+|\xi|)+\|\bfeta\|^4)
\end{pmatrix}\,,
\end{align*}
\begin{align*}
C_2&=
\begin{pmatrix}\lambda\,[\d_{\cx}\ucV]_0+(\eD^{\iD\xi}-1)\,\eD^{\uxip\bJ}\ucV_x(0)\\
\lambda^2\,\int_0^{\uXx}L\Mass[\ucV]\d_{\cx}\ucV
+\lambda\,(\eD^{\iD\xi}-1)\,\Mass[\ucV](0)
-\|\bfeta\|^2\,\int_0^{\uXx}\kappa(\|\ucV\|^2)\,\bJ\ucV\cdot\ucV_x\\
\lambda^2\,\int_0^{\uXx}L\Impulse_1[\ucV]\d_{\cx}\ucV
+\lambda\,(\eD^{\iD\xi}-1)\,\Impulse_1[\ucV](0)
-\|\bfeta\|^2\,\int_0^{\uXx}\kappa(\|\ucV\|^2)\,\|\ucV_x\|^2
\end{pmatrix}\\
&\qquad
+\begin{pmatrix}\cO(|\lambda|(|\lambda|+|\xi|)+\|\bfeta\|^2)\\
\cO((|\lambda|^2+\|\bfeta\|^2)(|\lambda|+|\xi|)+\|\bfeta\|^4)\\
\cO((|\lambda|^2+\|\bfeta\|^2)(|\lambda|+|\xi|)+\|\bfeta\|^4)
\end{pmatrix}\,,
\end{align*}
\begin{align*}
C_3&=
\begin{pmatrix}[\d_{\mup}\ucV]_0-(\eD^{\iD\xi}-1)\,\eD^{\uxip\bJ}\d_{\mup}\ucV(0)\\
\lambda\,\int_0^{\uXx}L\Mass[\ucV]\d_{\mup}\ucV
+\|\bfeta\|^2\,\int_0^{\uXx}\kappa(\|\ucV\|^2)\,\bJ\ucV\cdot\d_{\mup}\ucV
-(\eD^{\iD\xi}-1)\\
\lambda\,\int_0^{\uXx}L\Impulse_1[\ucV]\d_{\mup}\ucV
+\|\bfeta\|^2\,\int_0^{\uXx}\kappa(\|\ucV\|^2)\,\ucV_x\cdot\d_{\mup}\ucV
\end{pmatrix}
+\begin{pmatrix}\cO(|\lambda|+\|\bfeta\|^2)\\
\cO((|\lambda|+\|\bfeta\|^2)(|\lambda|+|\xi|+\|\bfeta\|^2))\\
\cO((|\lambda|+\|\bfeta\|^2)(|\lambda|+|\xi|+\|\bfeta\|^2))
\end{pmatrix}\,,
\end{align*}
and
\begin{align*}
C_4&=
\begin{pmatrix}
[\d_{\mux}\ucV]_0-(\eD^{\iD\xi}-1)\,\eD^{\uxip\bJ}\d_{\mux}\ucV(0)\\
\lambda\,\int_0^{\uXx}L\Mass[\ucV]\d_{\mux}\ucV
+\|\bfeta\|^2\,\int_0^{\uXx}\kappa(\|\ucV\|^2)\,\bJ\ucV\cdot\d_{\mux}\ucV\\
\lambda\,\int_0^{\uXx}L\Impulse_1[\ucV]\d_{\mux}\ucV
+\|\bfeta\|^2\,\int_0^{\uXx}\kappa(\|\ucV\|^2)\,\ucV_x\cdot\d_{\mux}\ucV
-(\eD^{\iD\xi}-1)
\end{pmatrix}
+\begin{pmatrix}\cO(|\lambda|+\|\bfeta\|^2)\\
\cO((|\lambda|+\|\bfeta\|^2)(|\lambda|+|\xi|+\|\bfeta\|^2))\\
\cO((|\lambda|+\|\bfeta\|^2)(|\lambda|+|\xi|+\|\bfeta\|^2))
\end{pmatrix}\,.
\end{align*}

Then the result follows steadily from an expansion of the determinant and a few manipulations on the first two lines based on Formula~\eqref{e:dPer} for $[\d_a \ucV]_0$.
\end{proof}


\section{Longitudinal perturbations}\label{s:1-d}

We begin by completing and discussing consequences of the latter sections on the stability analysis for longitudinal perturbations. For results derived --- \emph{via} Madelung's transformation --- from corresponding known results for larger classes of Euler--Korteweg systems, we also provide some hints about direct proofs.

\subsection{Co-periodic perturbations}\label{s:co-periodic}

As in \cite{Benzoni-Noble-Rodrigues-note,Benzoni-Mietka-Rodrigues}, we connect stability with respect to co-periodic longitudinal perturbations with properties of the Hessian of the action integral $\Theta$. We remind that $\Theta$ is considered as a function of 
$(\mu_x,c_x,\omp,\mup)$, in that order.

At the spectral level, restricting to co-periodic longitudinal perturbations corresponds to focusing on $\cL_{0,0}$, the Bloch-Fourier symbol at $(\xi,\bfeta)=(0,0)$. It is thus worth pointing out that it follows from identities in \eqref{dtheta} that 
the matrix $\Sigma_t$ in Theorem~\ref{th:low-freq} is such that
\begin{equation}\label{e:sigmatoTheta}
\Sigma_t
\,=\,\begin{pmatrix}
0&0&0&1\\
1&0&0&0\\
0&0&1&0\\
0&1&0&0
\end{pmatrix}\,
\begin{pmatrix}
\I_2&0\\
0&-\I_2
\end{pmatrix}
\,\Hess
\,\Theta
\,\begin{pmatrix}
0&0&0&1\\
0&1&0&0\\
1&0&0&0\\
0&0&1&0
\end{pmatrix}
\end{equation}
so that 
\begin{equation}\label{D0bas}
D_0(\lambda,0)
=\,\lambda^4\,\det\left(\Hess\,\Theta\right)
+\cO\left(|\lambda|^5\right)
\end{equation}
as $|\lambda|\to0$. Combining it with Proposition~\ref{p:hig-freq} provides the first half of the following theorem. 

\bt\label{th:co-periodic}
Let $\ucU$ be a wave profile of parameter $(\umux,\ucx,\uomp,\umup)$ such that $\Hess\,\Theta(\umux,\ucx,\uomp,\umup)$ is non-singular.
\begin{enumerate}
\item The number of eigenvalues of $\cL_{0,0}$ in $(0,+\infty)$, counted with algebraic multiplicity, is 
\begin{itemize}
\item even if $\det\left(\Hess\,\Theta\right)(\umux,\ucx,\uomp,\umup)>0$ ;
\item odd if $\det\left(\Hess\,\Theta\right)(\umux,\ucx,\uomp,\umup)<0$.
\end{itemize}
In particular, in the latter case the wave is spectrally exponentially unstable to co-periodic longitudinal perturbations.
\item Assume that $\d_{\mux}^2\Theta(\umux,\ucx,\uomp,\umup)\neq0$ and that the negative signature of $\Hess\,\Theta(\umux,\ucx,\uomp,\umup)$ equals two. Then the wave is conditionally orbitally stable in $H^1_{\per}((0,\uXx))$.
\end{enumerate}
\et

By conditional orbital stability in $H^1_{\per}((0,\uXx))$, we mean 
that for any $\delta_0>0$ there exists $\eps_0>0$ such that for any 
$\bU_0$ satisfying
\[
\inf_{(\vphip,\vphix)\in\RR^2}\left\|\bU_0\,-\,\eD^{\vphip\bJ}\ucU(\,\cdot+\vphix\,)\right\|_{H^1_{\per}((0,1))}\,\leq\,\eps_0
\]
and any solution\footnote{Knowing in which precise sense does not matter since only conservations are used in the stability argument.} $\bU$ to \eqref{e:moving-nls} defined on an interval $I$ containing $0$, starting from $\bU(0,\cdot)=\bU_0$ and sufficiently smooth to guarantee
\begin{itemize}
\item $\bU\in\cC^0(I;H^1_{\per}((0,1)))$;
\item $t\mapsto\int_0^{\uXx} \Mass[\bU(t,\cdot)]$, $t\mapsto\int_0^{\uXx} \Impulse_1[\bU(t,\cdot)]$ and $t\mapsto\int_0^{\uXx} \Ham[\bU(t,\cdot)]$ are constant on $I$;
\end{itemize}
then for any $t\in I$,
\[
\inf_{(\vphip,\vphix)\in\RR^2}\left\|\bU(t,\cdot)\,-\,
\eD^{\vphip\bJ}\ucU(\,\cdot+\vphix\,)\right\|_{H^1_{\per}((0,1))}\,\leq\,\delta_0\,.
\]

To go from conditional orbital stability to orbital stability, one needs to know that for the notion of solution at hand controlling the $H^1$-norm is sufficient to prevent finite-time blow-up. This is in particular the case when $\kappa$ is constant; see e.g. \cite[Section~3.5]{Cazenave}).

\begin{proof}
The first point  is a direct consequence of identity 
\eqref{D0bas} and $D_0(\lambda,0)>0$ for $\lambda$ real and large 
(proposition \ref{p:hig-freq}).

The second part is deduced from a corresponding result for the 
Euler-Korteweg system \eqref{e:absEK}: 
$(\urho,\uu)$ is conditionnally orbitally stable in $H^1_{per}\times
L^2_{per}((0,1))$ if the negative signature of 
$\Hess\,\Theta(\umux,\ucx,\uomp,\umup)$ equals two. See Theorem~3 and its accompanying remarks in \cite[Section~4.2]{Benzoni-Mietka-Rodrigues} (conveniently summarized as \cite[Theorem~1]{BMR2-I}). The conversion to our setting stems from the following lemma and the fact that System~\eqref{e:absEK} preserves the integral of $v_1$.
\end{proof}

\begin{lemma} \label{diffeosobolev}
\begin{enumerate}
\item For any $c_0>0$, there exist $\eps>0$ and $C$ such that if 
$\ucU\in H^1_{\per}((0,1))$ satisfies $\|\ucU\|\geq c_0$ then 
for any $(\vphip,\vphix)\in\R^2$ and any $\bU\in H^1_{\per}((0,1))$
satisfying 
\[
\left\|\bU\,-\,\eD^{\vphip\bJ}\ucU(\,\cdot+\vphix\,)\right\|_{H^1_{\per}((0,1))}
\,\leq\,\eps\,,
\]
with 
\begin{align*}
(\rho,\tv)&=
\left(\Mass[\bU],
\frac{\Impulse_1[\bU]}{\Mass[\bU]}\right)\,,&
(\urho,\utv)&=
\left(\Mass[\ucU],
\frac{\Impulse_1[\ucU]}{\Mass[\bU]}\right)\,,
\end{align*}
there holds $(\rho,\tv)$, $(\urho,\utv)\in H^1_{\per}((0,1))\times L^2((0,1))$, $\int_0^1\tv=\int_0^1\utv=0$, and
\begin{align*}
\left\|(\rho,\tv)
-(\urho,\utv)(\,\cdot+\vphix\,)\right\|_{H^1_{\per}\times L^2_{\per}}
&\leq
C\,\left(1+\left\|\ucU\right\|_{H^1_{\per}}^3\right)\,
\left\|\bU\,-\,\eD^{\vphip\bJ}\ucU(\,\cdot+\vphix\,)\right\|_{H^1_{\per}}\,.
\end{align*}
\item There exists $C$ such that if 
\begin{align*}
\ucU&=\sqrt{2\urho}\,\eD^{\utheta\bJ}(\beD_1)\,,&
\bU&=\sqrt{2\rho}\,\eD^{\theta\bJ}(\beD_1)\,,
\end{align*}
with $(\rho,\d_x\theta)$, $(\urho,\d_x\utheta)\in H^1_{\per}((0,1))\times L^2((0,1))$, $\int_0^1\d_x\theta=\int_0^1\d_x\utheta=0$,\\
then $\bU\in H^1_{\per}((0,1))$, $\ucU\in H^1_{\per}((0,1))$ and, for any $\vphix\in\R$,
\begin{align*}
&\left\|\bU\,-\,\eD^{\vphip\bJ}\ucU(\,\cdot+\vphix\,)\right\|_{H^1_{\per}}\\
&\ \leq\,
C\,\left(1+\left\|(\rho,\d_x\theta)\right\|_{H^1_{\per}\times L^2_{\per}}^2
+\left\|(\urho,\d_x\utheta)\right\|_{H^1_{\per}\times L^2_{\per}}^2\right)
\left\|(\rho,\d_x\theta)
-(\urho,\d_x\utheta)(\,\cdot+\vphix\,)\right\|_{H^1_{\per}\times L^2_{\per}}
\end{align*}
where
\[
\vphip\,=\,\int_0^1(\theta(\zeta)-\utheta(\,\zeta+\vphix\,))d\zeta\,.
\]
\end{enumerate}
\end{lemma}

\begin{proof}
The proof of the lemma is quite straightforward, using the continuous embedding 
\[
H^1_{\per}((0,1))\hookrightarrow L^\infty((0,1))
\]
and the Poincar\'e inequality. We use the latter in the following form. There exists $C$ such that for any $\theta$ such that $\d_x\theta\in L^2((0,1))$, $\int_0^1\d_x\theta=0$, we have $\theta \in H^1_{\per}((0,1))$ and if $\int_0^1 \theta=0$,
\[
\|\theta\|_{L^2((0,1))}\leq C\,\|\d_x\theta\|_{L^2((0,1))}\,.
\]
\end{proof}

The first part of the foregoing theorem could also be deduced from 
\cite{Benzoni-Noble-Rodrigues-note,Benzoni-Mietka-Rodrigues} 
through Section~\ref{s:linearized-EK}. In the reverse direction, 
we expect that the second part could be deduced from abstract 
results directly concerning equations of the same type as the 
nonlinear Schr\"odinger equation --- 
see~\cite{GSS-II,DeBievre-RotaNodari} --- 
essentially as the conclusions in \cite{Benzoni-Mietka-Rodrigues} 
used here were deduced there by combining an abstract result --- 
\cite[Theorem~3]{Benzoni-Mietka-Rodrigues} --- with a result proving 
connections with the action integral --- 
\cite[Theorem~7]{Benzoni-Mietka-Rodrigues}.

\medskip

As in \cite{BMR2-I} for systems of Korteweg type, we now specialize Theorem~\ref{th:co-periodic} to two asymptotic regimes, small amplitude and large period. To state our result, in the small amplitude regime, we need one more non-degeneracy index
\begin{align}\label{eq:falpha}
\falpha_0(\cx,\rho,\kp)
&:=\frac{1}{8(\d_\rho^2\cW_\rho)^3}
\ \Big[\ \frac53(\d_\rho^3\cW_\rho)^2
-\d_\rho^2\cW_\rho\,\d_\rho^4\cW_\rho\\
&\qquad\nn
-4\,\d_\rho^2\cW_\rho\,\d_\rho^3\cW_\rho\,
\left(\frac{\kappa'(2\rho)}{\kappa(\rho)}-\frac{1}{2\,\rho}\right)\\
&\qquad\nn
+16\,(\d_\rho^2\cW_\rho)^2\,
\left(\frac{\kappa''(2\rho)}{\kappa(\rho)}
-\frac{1}{2\,\rho}\frac{\kappa'(2\rho)}{\kappa(\rho)}
+\frac{1}{2\,(\rho)^2}\right)\ \Big]
\end{align}
with derivatives of $\cW_\rho$ evaluated at $(\rho;\cx,\omp,\mup)$, $(\omp,\mup)$ being associated with $(\cx,\rho,\kp)$ through \eqref{eq:harmonic-omp-mup}. The following theorem is then merely a translation of Corollaries~1 and~2 in \cite{BMR2-I}.
\bt\label{th:coperiodic_asymptotics}
\begin{enumerate}
\item In the small amplitude regime near a $(\ucx^{(0)},\urho^{(0)},\ukp^{(0)})$ such that\footnote{With $\umup^{(0)}$ associated with $(\ucx^{(0)},\urho^{(0)},\ukp^{(0)})$ through \eqref{eq:harmonic-omp-mup}.}
\begin{align*}
\d_\rho\nu(\urho^{(0)};\ucx^{(0)},\umup^{(0)})&\neq0\,,&
\falpha_0(\ucx^{(0)},\urho^{(0)},\ukp^{(0)})&\neq0\,,
\end{align*}
we have that $\d_{\mux}^2\Theta\neq0$ and that the negative signature of $\Hess\,\Theta$ equals two so that waves are conditionally orbitally stable in $H^1_{\per}((0,\uXx))$.
\item In the large period regime, $\d_{\mux}^2\Theta\neq0$ and 
\begin{itemize}
\item if $\d_{\cx}^2\Theta_{(s)}(\ucx^{(0)},\urho^{(0)},\ukp^{(0)})>0$ then in the large period regime near $(\ucx^{(0)},\urho^{(0)},\ukp^{(0)})$, the negative signature of $\Hess\,\Theta$ equals two so that waves are conditionally orbitally stable in $H^1_{\per}((0,\uXx))$;
\item if $\d_{\cx}^2\Theta_{(s)}(\ucx^{(0)},\urho^{(0)},\ukp^{(0)})<0$ then in the large period regime near $(\ucx^{(0)},\urho^{(0)},\ukp^{(0)})$, the negative signature of $\Hess\,\Theta$ equals three so that waves are spectrally exponentially unstable to co-periodic longitudinal perturbations.
\end{itemize}
\end{enumerate}
\et

A few comments are in order.
\begin{enumerate}
\item The condition $\falpha_0\neq0$ is directly connected to the condition $\d_{\mux}^2\Theta\neq0$ since $\falpha_0\,\Xx^{(0)}$ is the limiting value of $\d_{\mux}^2\Theta\neq0$ in the small-amplitude regime; see \cite[Theorem~4]{BMR2-I}.
\item The small-amplitude regime considered here is disjoint from the one analyzed for the semilinear cubic Schr\"odinger equations in \cite{Gallay-Haragus-spectral-small} since here the constant asymptotic mass is nonzero, namely $\urho^{(0)}>0$. 
\item The condition on $\partial_{\cx}^2\Theta_{(s)}$ agrees with 
the usual criterion for stability of solitary waves, known as the Vakhitov-Kolokolov slope condition; see e.g. \cite{GSS-II}.
\end{enumerate}

\subsection{Side-band perturbations}\label{s:side-band}
Side-band perturbations are perturbations corresponding to Floquet 
exponents $\xi$ arbitrarily small but non zero.
As in \cite{Benzoni-Noble-Rodrigues-note,Benzoni-Noble-Rodrigues,BMR2-II}, we analyse the spectrum of $\cL_{\xi,0}$ near 
$0$ when $\xi$ is small. Some instability criteria associated with
this part of the spectrum could be deduced readily from
Theorem~\ref{th:low-freq}. Yet we postpone slightly these
conclusions since we are more interested in proving that such
rigorous conclusions agree with those guessed from formal
geometrical optics considerations. 

Thus let us consider the two-phases slow/fastly-oscillatory \emph{ansatz}  
\be\label{e:ansatz}
\bU^{(\eps)}(t,x)
\,=\,\eD^{\frac{1}{\eps}\vphip^{(\eps)}(\eps\,t,\eps\,x)\,\bJ}\cU^{(\eps)}\left(\eps\,t,\eps\,x;
\frac{\vphix^{(\eps)}(\eps\,t,\eps\,x)}{\eps}\right)
\ee 
with, for any $(T,X)$, $\zeta\mapsto\cU^{(\eps)}(T,X;\zeta)$ periodic of period $1$ and, as $\eps\to0$,
\begin{align*}
\cU^{(\eps)}(T,X;\zeta)&=
\cU_{0}(T,X;\zeta)+\eps\,\cU_{1}(T,X;\zeta)+o(\eps)\,,\\
\vphip^{(\eps)}(T,X)&=
(\vphip)_{0}(T,X)+\eps\,(\vphip)_{1}(T,X)+o(\eps)\,,\\
\vphix^{(\eps)}(T,X)&=
(\vphix)_{0}(T,X)+\eps\,(\vphix)_{1}(T,X)+o(\eps)\,.
\end{align*}
Requiring \eqref{e:ansatz} to solve \eqref{e:ab} up to a remainder of size $o(1)$ is equivalent to $\zeta\mapsto\cU_{0}(T,X;\zeta)$ being a scaled profile of a periodic traveling wave of \eqref{e:unscaled-profile}. Explicitly, 
\begin{equation}\label{e:scaled2}
\bJ\delta \cH_0(\cU_0,\eD_1\,(\kp\bJ+\kx\d_\zeta)\cU_0)
=\omp\bJ\cU_0
-\cx\,(\kp\bJ+\kx\d_\zeta)\cU_0.
\end{equation}
with local parameters (depending on slow variables $(T,X)$) related to phases by
\begin{align*}
\d_T(\vphip)_{0}&=\omp-\kp\,\cx\,,&
\d_X(\vphip)_{0}&=\kp\,,&
\d_T(\vphix)_{0}&=\omx\,,&
\d_X(\vphix)_{0}&=\kx\,.
\end{align*}
Symmetry of derivatives already constrains the slow evolution of wave parameters with
\begin{align*}
\d_T\kp&=\,\d_X\left(\omp-\kp\,\cx\right)\,,&
\d_T\kx&=\,\d_X\omx\,.
\end{align*}
Since periodic profiles form a $4$-dimensional manifold (after 
discarding translation and rotation parameters), in order to determine the leading-order dynamics of \eqref{e:ansatz}, we need two more equations. The fastest way to obtain such equations is to also require \eqref{e:ansatz} to solve \eqref{e:mass} and \eqref{e:momentum} up to remainders of size $o(\eps)$. Observing 
that all quantities in \eqref{e:mass} and \eqref{e:momentum} are 
independent of phases, 
\begin{align*}
\d_T(\Mass(\cU_{0}))
&\,=\,\d_X(\Sp(\cU_{0},\eD_1\,(\kx\d_\zeta+\kp\bJ)\cU_{0}))
+\d_\zeta\left(*\right)\,,\\
\d_T(\Impulse_1(\cU_{0},\eD_1\,(\kx\d_\zeta+\kp\bJ)\cU_{0}))
&\,=\,\d_X\big(\big(\n_{\bU_x}\Impulse_1 \cdot\bJ\delta \cH_0+\Sx\big)
(\cU_{0},\eD_1\,(\kx\d_\zeta+\kp\bJ)\cU_{0})\big)
+\d_\zeta\left(**\right)\,,
\end{align*}
with omitted terms $*$ and $**$ $1$-periodic in $\zeta$. Averaging in $\zeta$ (using \eqref{e:scaled2}) provides two more equations, completing the modulation system
\be\label{e:W-formal} 
\left\{
\begin{array}{rcl}
\d_T\kx&=&\d_X\omx\\\hspace{-0.75em}
\d_T(\langle\Impulse(_1\cU_{0},\eD_1\,(\kx\d_\zeta+\kp\bJ)\cU_{0})\rangle)
&=&\hspace{-0.75em}
\d_X(\langle(\omp\cM-\cx\Impulse_1
+\Sx)(\cU_{0},\eD_1\,(\kx\d_\zeta+\kp\bJ)\cU_{0})\rangle)\\
\d_T(\langle\Mass(\cU_{0})\rangle)
&=&\d_X(\langle\Sp(\cU_{0},\eD_1\,(\kx\d_\zeta+\kp\bJ)\cU_{0})\rangle)\\
\d_T\kp&=&\d_X\left(\omp-\kp\,\cx\right)
\end{array}
\right.
\ee
where $\di \langle\,\cdot\,\rangle=\int_0^1 \cdot\, d\zeta$ is 
the average over a periodic cell.

The reader may wonder why in the foregoing formal derivation we have asked for \eqref{e:ab} to be satisfied at order $1$ and for \eqref{e:mass} and \eqref{e:momentum} to be satisfied at order $\eps$. Alternatively, one may ask for \eqref{e:ab} to be satisfied at order $\eps$ and check that requirements on \eqref{e:mass} and \eqref{e:momentum} come as necessary conditions. One may also check that when \eqref{e:ab} is satisfied at order $1$ so are \eqref{e:mass} and \eqref{e:momentum}. 

System \eqref{e:W-formal} should be thought as a system for functions defined on the manifold of periodic traveling waves (identified when coinciding up to rotational and spatial translations). To make this more concrete, we now rewrite it in terms of parameters $(\mux,\cx,\omp,\mup)$. To do so, with notation from Section~\ref{s:profile}, we introduce 
\begin{equation}\label{e:averageMQ}
\left\{
\begin{array}{ll}
\di \mass(\mux,\cx,\omp,\mup)&:=\langle\Mass[\cV]\rangle
\,=\,\di \frac{1}{\Xx}\int_0^{\Xx}\Mass[\cV]dx\,,\\
\di \impulse(\mux,\cx,\omp,\mup)&:=\langle\Impulse[\cV]\rangle
\,=\,\di \frac{1}{\Xx}\int_0^{\Xx}\Impulse_1[\cV]dx\,,
\end{array}\right.
\end{equation}
where $\cV$ is the unscaled profile associated with $(\mux,\cx,\omp,\mup)$ and $\Xx$ is the corresponding period. By making use of \eqref{e:massu} and \eqref{e:momu}, one obtains
\be\label{e:W}
\left\{
\begin{array}{rcl}
\d_T\kx&=&\d_X\omx\\
\d_T\impulse
&=&\d_X\left(\mux-\cx\impulse\right)\\
\d_T\mass
&=&\d_X\left(\mup-\cx\mass\right)\\
\d_T\kp&=&\d_X\left(\omp-\cx\,\kp\right)
\end{array}
\right.
\ee
as an alternative form of \eqref{e:W-formal}. To connect with the analysis of other sections in terms of the action integral $\Theta$, we recall \eqref{dtheta}
\begin{align*}
\kx&=\frac{1}{\d_{\mux}\Theta}\,,&
\begin{pmatrix}
1\\
\impulse\\
\mass\\
\kp
\end{pmatrix}
&\,=\,\frac{\bA_0\,\nabla\Theta}{\d_{\mux}\Theta}\,,&
\textrm{with}\quad
\bA_0&:=\,
\begin{pmatrix}
1&0&0&0\\
0&1&0&0\\
0&0&-1&0\\
0&0&0&-1
\end{pmatrix}\,.
\end{align*}
Thus (for smooth solutions) System~\eqref{e:W} takes the alternative form
\be\label{e:W-action}
\kx\bA_0\Hess\Theta\,\left(\d_T+\cx\d_X\right)
\begin{pmatrix}
\mux\\\cx\\\omp\\\mup
\end{pmatrix}
\,=\,\bB_0\,\d_X
\begin{pmatrix}
\mux\\\cx\\\omp\\\mup
\end{pmatrix}
\ee
with
\begin{align*}
\bB_0&=\,
\begin{pmatrix}
0&1&0&0\\
1&0&0&0\\
0&0&0&1\\
0&0&1&0
\end{pmatrix}\,.
\end{align*}

\br
One may check that the modulated system \eqref{e:W}, also often called Whitham's system, agrees with the one derived for the associated Euler--Korteweg system \eqref{e:absEK} by injecting a one-phase slow/fastly-oscillatory \emph{ansatz}. See \cite{R,Benzoni-Noble-Rodrigues,BMR2-II} for a discussion of the latter. This may be achieved by direct comparisons of either formal \emph{ansatz}, averaged forms or more concrete parameterized forms.
\er

We now specialize the use of System~\eqref{e:W} to the discussion of the dynamics near a particular periodic traveling wave. Note that traveling-wave solutions fit the \emph{ansatz} \eqref{e:ansatz} and correspond to the case when phases $\vphip$ and $\vphix$ are affine functions of the slow variables and wave parameters are constant. Thus, when $\ubU$ is a wave profile of parameters $(\umux,\ucx,\uomp,\umup)$, one may expect that the stability\footnote{Incidentally we point out that from the homogeneity of first-order systems it follows that ill-posedness and stability are essentially the same for systems such as \eqref{e:W}.} of $(\umux,\ucx,\uomp,\umup)$ as a solution to \eqref{e:W} is necessary to the stability of $\ubU$ as a solution to \eqref{e:moving-nls}. The literature proving such a claim at the spectral level is now quite extensive and we refer the reader to \cite{Serre,Noble-Rodrigues}, \cite{Benzoni-Noble-Rodrigues,Bronski-Hur-Johnson}, \cite{KR}, \cite{JNRYZ} for results respectively on parabolic systems, Hamiltonian systems of Korteweg type, lattice dynamical systems and some hyperbolic systems with discontinuous waves. Yet this is the first time\footnote{Except for the almost simultaneous \cite{Clarke-Marangell}. See detailed comparison in Section~\ref{s:large-time}.} that a result for a class of systems with symmetry group of dimension higher than one is established. 

In the present case, the spectral validation of \eqref{e:W} is a simple corollary of Theorem~\ref{th:low-freq} based on a counting root argument for analytic functions, since
\[
\lambda\,\Sigma_t-(\eD^{\iD\xi}-1)\I_4
\,=\,\begin{pmatrix}
0&0&0&1\\
1&0&0&0\\
0&0&1&0\\
0&1&0&0
\end{pmatrix}\,
\left(
\lambda\,\bA_0
\,\Hess
\,\Theta
-(\eD^{\iD\xi}-1)\,
\bB_0
\right)
\,\begin{pmatrix}
0&0&0&1\\
0&1&0&0\\
1&0&0&0\\
0&0&1&0
\end{pmatrix}\,.
\]

\bc\label{c:W-cor}
Consider an unscaled wave profile $\ucV$ such that $\ucV\cdot \ucV_x\not\equiv0$, with associated parameters $(\umux,\ucx,\uomp,\umup)$.
\begin{enumerate}
\item The following three statements are equivalent.
\begin{itemize}
\item $0$ is an eigenvalue of algebraic multiplicity $4$ of $\cL_{0,0}$.
\item The map $(\mux,\cx,\omp,\mup)\mapsto (\kx,\impulse,\mass,\kp)$ is a local diffeomorphism near $(\umux,\ucx,\uomp,\umup)$.
\item $\Hess\Theta (\umux,\ucx,\uomp,\umup)$ is non-singular.
\end{itemize}
\item Assume that $\Hess\Theta (\umux,\ucx,\uomp,\umup)$ is non-singular. Then there exist $\lambda_0>0$, $\xi_0>0$ and $C_0$ such that 
\begin{itemize}
\item for any $\xi\in [-\xi_0,\xi_0]$, $\cL_{\xi,0}$ possesses $4$ eigenvalues (counted with algebraic multiplicity) in the disk $B(0,\lambda_0)$;
\item if $a-\ucx$ is a characteristic speed of \eqref{e:W} at $(\umux,\ucx,\uomp,\umup)$ of algebraic multiplicity~$r$, that is, if $a$ is an eigenvalue of $(\kx\bA_0\Hess\Theta(\umux,\ucx,\uomp,\umup))^{-1}\bB_0$ of algebraic multiplicity~$r$, then for any $\xi\in [-\xi_0,\xi_0]$, $\cL_{\xi,0}$ possesses $r$ eigenvalues (counted with algebraic multiplicity) in the disk $B(\iD\ukx\xi\,a\,,\,C_0|\xi|^{1+\frac1r})$.
\end{itemize}
In particular if System~\eqref{e:W} is not weakly hyperbolic at $(\umux,\ucx,\uomp,\umup)$, that is, if $(\kx\bA_0\Hess\Theta(\umux,\ucx,\uomp,\umup))^{-1}\bB_0$ possesses a non-real eigenvalue, then the wave is spectrally unstable to longitudinal side-band perturbations.
\end{enumerate}
\ec

A few comments are in order.
\begin{enumerate}
\item Note that the subtraction of $\ucx$ in the second part of the corollary accounts for the fact that System~\ref{e:W} is not expressed in a co-moving frame. 
\item The second part of the foregoing corollary could also be deduced from results in \cite{Benzoni-Noble-Rodrigues} through Madelung's transformation.
\end{enumerate}
 
We now turn to the small-amplitude and large-period regimes. 
To describe the small-amplitude regime, we need to introduce 
two instability indices
\begin{align}\label{def:sideindex}
\delta_{hyp}(\cx,\omp,\mup)
&:=\,W''(2\,\rho^{(0)})
+\left(\kappa''(2\,\rho^{(0)})\,\rho^{(0)}+\kappa'(2\,\rho^{(0)})\right)\,(\kp^{(0)})^2
\end{align}
and
\begin{align}\label{def:sideindex_II}
&\delta_{BF}(\cx,\omp,\mup)\\
&:=\,
\left(\frac12\frac{\kappa(2\,\rho^{(0)})}{2\,\rho^{(0)}}
\left(\frac{2\pi}{\Xx^{(0)}}\right)^2\right)^3
\,\left(
-3\,\left(\frac{\kappa'(2\,\rho^{(0)})}{\kappa(2\,\rho^{(0)})}\right)^2
-2\,\frac{\kappa'(2\,\rho^{(0)})}{\kappa(2\,\rho^{(0)})}\,\frac{1}{2\,\rho^{(0)}}
+\frac{\kappa''(2\,\rho^{(0)})}{\kappa(2\,\rho^{(0)})}
\right)\nn\\
&
+\left(\frac12\frac{\kappa(2\,\rho^{(0)})}{2\,\rho^{(0)}}
\left(\frac{2\pi}{\Xx^{(0)}}\right)^2\right)^2\,\nn\\
&\qquad\times
\,\Bigg(
W''(2\,\rho^{(0)})\,
\left(
-12\,\left(\frac{\kappa'(2\,\rho^{(0)})}{\kappa(2\,\rho^{(0)})}\right)^2
-6\,\frac{\kappa'(2\,\rho^{(0)})}{\kappa(2\,\rho^{(0)})}\,\frac{1}{2\,\rho^{(0)}}\,
+4\,\left(\frac{1}{2\,\rho^{(0)}}\right)^2
+3\,\frac{\kappa''(2\,\rho^{(0)})}{\kappa(2\,\rho^{(0)})}
\right)\nn\\
&\qquad\qquad
+4\,W'''(2\,\rho^{(0)})\,\left(\frac{\kappa'(2\,\rho^{(0)})}{\kappa(2\,\rho^{(0)})}
+2\,\frac{1}{2\,\rho^{(0)}}\right)
+2\,W''''(2\,\rho^{(0)})
\Bigg)\nn\\
&+\left(\frac12\frac{\kappa(2\,\rho^{(0)})}{2\,\rho^{(0)}}
\left(\frac{2\pi}{\Xx^{(0)}}\right)^2\right)\nn\\
&\qquad\times\,\Bigg(
12\,(W''(2\,\rho^{(0)}))^2\,
\left(
\left(\frac{\kappa'(2\,\rho^{(0)})}{\kappa(2\,\rho^{(0)})}\right)^2
+4\,\frac{\kappa'(2\,\rho^{(0)})}{\kappa(2\,\rho^{(0)})}\,\frac{1}{2\,\rho^{(0)}}\,
+3\,\left(\frac{1}{2\,\rho^{(0)}}\right)^2
\right)\nn\\
&\qquad\qquad
+8\,W''(2\,\rho^{(0)})\,W'''(2\,\rho^{(0)})
\,\left(4\,\frac{\kappa'(2\,\rho^{(0)})}{\kappa(2\,\rho^{(0)})}
+5\,\frac{1}{2\,\rho^{(0)}}\right)\nn\\
&\qquad\qquad
+\frac43\left(W'''(2\,\rho^{(0)})\right)^2
+6\,W''(2\,\rho^{(0)})\,W''''(2\,\rho^{(0)})
\Bigg)\nn\\
&+8\,W''(2\,\rho^{(0)})\,
\left(W'''(2\,\rho^{(0)})\,
+3\,W''(2\,\rho^{(0)})\,
\left(\frac{\kappa'(2\,\rho^{(0)})}{\kappa(2\,\rho^{(0)})}
+\frac{1}{2\,\rho^{(0)}}\right)\right)^2\,.\nn
\end{align}
where $(\rho^{(0)},\kp^{(0)})$ are the associated limiting mass and rotational shift and $\Xx^{(0)}$ is the associated period.

The following theorem is a consequence of 
Corollary~\ref{c:W-cor} and results in \cite{BMR2-II} for the 
Euler--Korteweg systems, namely Theorems~7 and~8 
respectively for the first and second points\footnote{In notation of \cite{BMR2-II}, $\delta_{BF}$ is $\Delta_{MI}$.}. 

\bt\label{th:side-band_asymptotics}
\begin{enumerate}
\item In the small amplitude regime near a $(\ucx^{(0)},\urho^{(0)},\ukp^{(0)})$ such that 
\[
\d_\rho\nu(\urho^{(0)};\ucx^{(0)},\umup^{(0)})\neq0\,,
\]
$\Hess\Theta$ is non singular and if 
\[
\delta_{hyp}(\ucx^{(0)},\uomp^{(0)},\umup^{(0)})<0
\qquad\textrm{or}\qquad
\delta_{BF}(\ucx^{(0)},\uomp^{(0)},\umup^{(0)})<0,
\]
then waves are spectrally exponentially unstable to longitudinal side-band perturbations.
\item If $\d_{\cx}^2\Theta_{(s)}(\ucx^{(0)},\urho^{(0)},\ukp^{(0)})\neq0$ then, in the large period regime near $(\ucx^{(0)},\urho^{(0)},\ukp^{(0)})$, $\Hess\Theta$ is non singular and if 
\[
\d_{\cx}^2\Theta_{(s)}(\ucx^{(0)},\urho^{(0)},\ukp^{(0)})<0,
\]
then in the large period regime near $(\ucx^{(0)},\urho^{(0)},\ukp^{(0)})$, waves are spectrally exponentially unstable to longitudinal side-band perturbations.
\end{enumerate}
\et

A few comments are worth stating. In particular, we borrow here some of the upshots of the much more comprehensive analysis in \cite{BMR2-II}.
\begin{enumerate}
\item Again we point out that the small-amplitude regime considered here is disjoint from the one analyzed for the semilinear cubic Schr\"odinger equations in \cite{Gallay-Haragus-spectral-small}. Let us however stress that for the semilinear cubic Schr\"odinger equations our instability criterion provides instability if and only if the potential is focusing, independently of the particular limit value under consideration. This is consistent with the conclusions for the case $\rho^{(0)}=0$ derived in \cite{Gallay-Haragus-spectral-small}.
\item In the small amplitude limit, the characteristic velocities split in two groups of two. One of these groups converges to the linear group velocity at the limiting constant value and the sign of $\delta_{BF}$ precisely determines how this double root splits. The corresponding instability is often referred to as the Benjamin--Feir instability. The other group converges to the characteristic velocities of a dispersionless hydrodynamic system at the limiting constant value; see \cite[Theorem~7]{BMR2-II}. The sign of $\delta_{hyp}$ decides the weak hyperbolicity of the latter system. When $\kappa$ is constant, it is directly related to the focusing/defocusing nature of the potential $W$ (namely $W''$ negative/positive).
\item A similar scenario takes place in the large period limit, with the phase velocity of the solitary wave replacing the linear group velocity. The sign of $\d_{\cx}^2\Theta_{(s)}$ determines how the double root splits. However, due to the nature of endstates of solitary waves, the dispersionless system is always hyperbolic, hence the reduction to a single instability index. See Appendix~\ref{s:constant} for some related details.
\end{enumerate}

\subsection{Large-time dynamics}\label{s:large-time}

Our interest in modulated systems also hinges on the belief that they play a deep role in the description of the large-time dynamics. In other words, one expects that near stable waves the large-time dynamics is well-approximated by simply varying wave parameters in a space-time dependent way and that the dynamics of these parameters is itself well-captured by some (higher-order version of a) modulated system.

The latter scenario has been proved to occur at the nonlinear level for a large class of parabolic systems \cite{JNRZ-RD2,JNRZ-conservation} and at the linearized level for the Korteweg--de Vries equation \cite{R_linKdV}. The reader is also referred to \cite{R,R_Roscoff} for some more intuitive arguments supporting the general claim.

We would like to extend here a small part of the analysis in \cite{R_linKdV} to the class of equations under consideration. We begin by revisiting the second part of Corollary~\ref{c:W-cor} from the point of view of Floquet symbols rather than Evans' functions. The goal is to provide a description of how eigenfunctions and spectral projectors behave near the quadruple eigenvalue at the origin. Once this is done, the arguments of \cite{R_linKdV} may be directly imported and provide different results (adapted to the presence of a two-dimensional group of symmetries) but with nearly identical --- thus omitted --- proofs.

In a certain way, we leave the point of view convenient for spatial dynamics to focus on time dynamics. To do so, it is expedient to use scaled variables so as to normalize period and to parameter waves not by phase-portrait parameters $(\mux,\cx,\omp,\mup)$ but by modulation parameters $(\kx,\kp,\impulse,\mass)$. The first part of Corollary~\ref{c:W-cor} proves that the latter is possible when the eigenvalue at the origin is indeed of multiplicity $4$. Therefore in the present subsection, we consider scaled profiles $\cU$ as in \eqref{e:moving-nls},and parameters $(\mux,\cx,\omp,\mup)$ as functions of $(\kx,\kp,\impulse,\mass)$. In scaled variables, the averaged mass and impulse from \eqref{e:averageMQ} take the form
\begin{align*}
\di\mass&=\langle \Mass(\cU)\rangle=\int_0^1\frac12\|\cU\|^2\,,\\
\di \impulse&=\langle\Impulse_1(\cU,\eD_1 (\kx\partial_x+\kp\bJ)\cU)\rangle
=\int_0^1\frac12\bJ\cU\cdot(\kx\partial_x+\kp\bJ)\cU\,.
\end{align*}

Our focus is on the operator $\cL_{\xi,0}=\cL_\xi^x$. Correspondingly we consider the Whitham matrix-valued map
\begin{equation}\label{WhiMat}
\bW(\mux,\cx,\omp,\mup)\,:=\,\Jac
\bp
\omx\\
\omp-\cx\,\kp\\
\mux-\cx\impulse\\
\mup-\cx\mass
\ep\,.
\end{equation}
To connect both objects, we shall use various algebraic relations obtained from profile equations and conservation laws that we first derive.

Differentiating profile equation $\delta\cH_u[\cU]=0$
with respect to rotational and spatial translation parameters (left implicit here) and to $(\impulse,\mass)$ yields
\begin{align}\label{dprofile1}
\cL_{0,0}\,\ucU_x&\,=\,0\,,&
\cL_{0,0}\,\d_{\impulse}\ucU&\,=\,
\left(\d_{\impulse}\uomp-\ukp\,\d_{\impulse}\ucx\right)\bJ\,\ucU
-\ukx\d_{\impulse}\ucx\,\ucU_x\,,\\
\cL_{0,0}\,\bJ\,\ucU&\,=\,0\,,&
\cL_{0,0}\,\d_{\mass}\ucU&\,=\,
\left(\d_{\mass}\uomp-\ukp\,\d_{\mass}\ucx\right)\bJ\,\ucU
-\ukx\,\d_{\mass}\ucx\,\ucU_x\,.
\end{align}
To highlight the role of $\d_{\kx}\ucU$ and $\d_{\kp}\ucU$, we expand
\[
\cL_{\xi,0}
\,=\,\cL_{0,0}
+\iD\ukx\xi\,\cL_{(1)}
+(\iD\ukx\xi)^2\,\cL_{(2)}\,.
\]
Differentiating profile equations with respect to $(\kx,\kp)$ leaves
\begin{equation}\label{dprofile2}
 \left\{
 \begin{array}{ll}
\cL_{0,0}\,\d_{\kx}\ucU&\,=\,
\left(\d_{\kx}\uomp-\ukp\,\d_{\kx}\ucx\right)\bJ\,\ucU
-\ukx\,\d_{\kx}\ucx\ \ucU_x
-\cL_{(1)}\,\ucU_x\,,\\
\cL_{0,0}\,\d_{\kp}\ucU&\,=\,
\left(\d_{\kp}\uomp-\ukp\,\d_{\kp}\ucx\right)\,\bJ\,\ucU
-\ukx\,\d_{\kp}\ucx\ \ucU_x
-\cL_{(1)}\,\bJ\,\ucU\,.
\end{array}\right.
\end{equation}
By differentiating the definitions of mass and impulse averages, 
we also obtain that 
\begin{align}\label{dimpulse}
\int_0^1\delta\Impulse_1(\ucU,\eD_1(\ukx\d_x+\ukp\bJ)\ucU)\,\ucU_x&\,=\,0\,,&
\int_0^1\delta\Impulse_1(\ucU,\eD_1(\ukx\d_x+\ukp\bJ)\ucU)\,\bJ\,\ucU&\,=\,0\,,\\
\int_0^1\delta\Impulse_1(\ucU,\eD_1(\ukx\d_x+\ukp\bJ)\ucU)\,\d_\impulse\ucU&\,=\,1\,,&
\int_0^1\delta\Impulse_1(\ucU,\eD_1(\ukx\d_x+\ukp\bJ)\ucU)\,\d_\mass\ucU&\,=\,0\,,
\end{align}
\begin{align}
\int_0^1\delta\Impulse_1(\ucU,\eD_1(\ukx\d_x+\ukp\bJ)\ucU)\,\d_{\kx}\ucU&\,=\,
-\int_0^1\dD_{\bU_x}\Impulse_1(\ucU,\eD_1(\ukx\d_x+\ukp\bJ)\ucU)\,\ucU_x\\
&\,=\,-\frac{1}{\ukx}(\uimpulse-\ukp\umass)\,,\nn\\
\int_0^1\delta\Impulse_1(\ucU,\eD_1(\ukx\d_x+\ukp\bJ)\ucU)\,\d_{\kp}\ucU
&\,=\,
-\int_0^1\dD_{\bU_x}\Impulse(\ucU,\eD_1(\ukx\d_x+\ukp\bJ)\ucU)\,\bJ\,\ucU\\
&\,=\,-\umass\,,\nn
\end{align}
and
\begin{align}\label{dmasse}
\int_0^1\delta\Mass[\ucU]\,\ucU_x&\,=\,0\,,&
\int_0^1\delta\Mass[\ucU]\,\d_\impulse\ucU&\,=\,0\,,&
\int_0^1\delta\Mass[\ucU]\,\d_{\kx}\,\ucU&\,=\,0\,,&\\
\int_0^1\delta\Mass[\ucU]\,\bJ\,\ucU&\,=\,0\,,&
\int_0^1\delta\Mass[\ucU]\,\d_\mass\ucU&\,=\,1\,,&
\int_0^1\delta\Mass[\ucU]\,\d_{\kp}\,\ucU&\,=\,0\,,\,.
\end{align}

At last, linearizing conservation laws for mass and impulse provides for any smooth $\bV$
\begin{align*}
\delta\Mass[\ucU]\cdot\cL_{\xi,0} \bV&\\
=\ukx(\d_x&+\iD\xi)
\Bigg(
\n_\bU (\Sp+\ucx\Mass)(\ucU,\eD_1(\ukx\d_x+\ukp\bJ)\ucU)\cdot\bV\\&\qquad\qquad
+\n_{\bU_x} (\Sp+\ucx\Mass)(\ucU,\eD_1(\ukx\d_x+\ukp\bJ)\ucU)\cdot
(\ukx(\d_x+\iD\xi)+\ukp\bJ)\bV
\Bigg)
\end{align*}
and
\begin{align*}
\n_\bU\Impulse_1(\ucU,&\eD_1(\ukx\d_x+\ukp\bJ)\ucU)\cdot\cL_{\xi,0} \bV
+\n_{\bU_x}\Impulse_1(\ucU,\eD_1(\ukx\d_x+\ukp\bJ)\ucU)\cdot
(\ukx(\d_x+\iD\xi)+\ukp\bJ)\,\cL_{\xi,0}\bV\\
=\ukx(\d_x&+\iD\xi)
\Bigg(\n_{\bU_x} \Impulse_1(\ucU,\eD_1(\ukx\d_x+\ukp\bJ)\ucU)\cdot\cL_{\xi,0}\bV\\&\qquad\qquad
+\n_\bU (\Sx+\uomp\Mass)(\ucU,\eD_1(\ukx\d_x+\ukp\bJ)\ucU)\cdot\bV\\&\qquad\qquad
+\n_{\bU_x} (\Sx+\uomp\Mass)(\ucU,\eD_1(\ukx\d_x+\ukp\bJ)\ucU)\cdot
(\ukx(\d_x+\iD\xi)+\ukp\bJ)\bV
\Bigg)\,.
\end{align*}
Evaluating at $\xi=0$ and integrating show  
\begin{align}\label{kerL*}
\cL_{0,0}^*\,\delta\Impulse_1(\ucU,(\ukx\d_x+\ukp\bJ)\ucU)&=0,& 
\cL_{0,0}^*\,\delta\Mass[\ucU]&=0,
\end{align}
where $\cL_{0,0}^*$ denotes the adjoint of $\cL_{0,0}$. Alternatively the latter may be checked by using explicit expressions of $\delta\Impulse_1(\ucU,\eD_1(\ukx\d_x+\ukp\bJ)\ucU)$ and $\delta\Mass[\ucU]$ in terms of $\bJ\ucU$ and $\ucU_x$ and Hamiltonian duality $\cL_{0,0}^*=-\bJ^{-1}\,\cL_{0,0}\,\bJ$. At next orders, for $\bV$ smooth and periodic we also deduce
\begin{align}\label{highermasse}
\langle\delta\Mass[\ucU];\cL_{(1)}\bV\rangle_{L^2}
&=\langle\delta(\Sp+\ucx\Mass)(\ucU,\eD_1(\ukx\d_x+\ukp\bJ)\ucU);\bV\rangle_{L^2}\,,\\
\langle\delta\Mass[\ucU];\cL_{(2)}\bV\rangle_{L^2}
&=\langle\nabla_{\bU_x}(\Sp+\ucx\Mass)(\ucU,\eD_1(\ukx\d_x+\ukp\bJ)\ucU);\bV\rangle_{L^2}\,,
\end{align}
and
\begin{align}\label{higherimpulse}
\langle\delta\Impulse_1(\ucU,\eD_1(\ukx\d_x+\ukp\bJ)\ucU);\cL_{(1)}\bV\rangle_{L^2}
&=\langle\delta(\Sx+\uomp\Mass)(\ucU,\eD_1(\ukx\d_x+\ukp\bJ)\ucU);\bV\rangle_{L^2}\,,\\
\langle\delta\Impulse_1(\ucU,\eD_1(\ukx\d_x+\ukp\bJ)\ucU);\cL_{(2)}\bV\rangle_{L^2}
&=\langle\nabla_{\bU_x}(\Sx+\uomp\Mass)(\ucU,\eD_1(\ukx\d_x+\ukp\bJ)\ucU);\bV\rangle_{L^2}\,,
\end{align}
In the foregoing relations,  $\langle\cdot;\cdot\rangle_{L^2}$ denotes the canonical Hermitian scalar product on $L^2((0,1);\C^2)$, $\C$-linear on the right\footnote{That is, $\langle f;g \rangle_{L^2}=\int_0^1 \bar f \cdot g$.}.

These are the key algebraic relations to prove the following proposition. 

\bt\label{th:W-Bloch}
Let $\ucU$ be a wave profile such that $\ucU\cdot \ucU_x\not\equiv0$ and that $0$ has algebraic multiplicity exactly $4$ as an eigenvalue of $\cL_{0,0}$. Assume that eigenvalues of $\ubW$ are distinct.\\
There exist $\lambda_0>0$, $\xi_0\in(0,\pi)$, analytic curves $\lambda_j:[-\xi_0,\xi_0]\to B(0,\lambda_0)$, $j=1,2,3,4$, such that for $\xi\in[-\xi_0,\xi_0]$
\[
\sigma(\cL_{\xi,0})\cap B(0,\lambda_0)\ =\ \left\{\ \lambda_j(\xi)\ \middle|\ j\in\{1,2,3,4\}\ \right\}
\]
and associated left and right eigenfunctions $\tpsi_j(\xi,\cdot)$ and $\psi_j(\xi,\cdot)$, $j=1,2,3,4$, satisfying pairing relations\footnote{With $\delta^j_\ell=1$ if $j=\ell$, and $\delta^j_\ell=0$ otherwise.}
\[
\langle\tpsi_j(\xi,\cdot),\psi_\ell(\xi,\cdot)\rangle_{L^2}=\iD\ukx\xi\ \delta^j_\ell,\qquad 1\le j,\ell\le 4,
\]
obtained as
\[
\begin{array}{rcccl}
\displaystyle
\psi_j(\xi,\cdot)&=&\displaystyle
\quad
\sum_{\ell=1}^2\beta_{\ell}^{(j)}(\xi)\ q_{\ell}(\xi,\cdot)
&+&\displaystyle
(\iD\ukx\xi)\sum_{\ell=3}^4\beta_\ell^{(j)}(\xi)\ q_\ell(\xi,\cdot)\\
\displaystyle
\tpsi_j(\xi,\cdot)&=&\displaystyle
-(\iD\ukx\xi)\,\sum_{\ell=1}^2\tbeta_{\ell}^{(j)}(\xi)\ \tq_{\ell}(\xi,\cdot)
&+&\displaystyle
\quad\sum_{\ell=3}^4\tbeta_\ell^{(j)}(\xi,\cdot)\ \tq_\ell(\xi,\cdot)\\
\end{array}
\]
where
\begin{itemize}
\item $(q_j(\xi,\cdot))_{1\leq j\leq 4}$ and $(\tq_j(\xi,\cdot))_{1\leq j\leq 4}$ are dual bases of spaces associated with the spectrum in $B(0,\lambda_0)$ of respectively $\cL_{\xi,0}$ and its adjoint $\cL_{\xi,0}^*$, that are analytic in $\xi$ and such that
\begin{align*}
(q_1(0,\cdot),q_2(0,\cdot),q_3(0,\cdot),q_4(0,\cdot))
&\,=\,(\ucU_x,\bJ\,\ucU,\d_\impulse\ucU,\d_\mass\ucU)\,,\\
(\tq_3(0,\cdot),\tq_4(0,\cdot))
&\,=\,(\delta\Impulse_1(\ucU,\eD_1(\ukx\d_x+\ukp\bJ)\ucU),\ \delta\Mass[\ucU])\\
&\,=\,(-\ukx\,\bJ\,\ucU_x+\ukp\,\ucU,\ \ucU)\,,\\
(\d_\xi q_1(0,\cdot),\d_\xi q_2(0,\cdot))
&\,=\,\iD\ukx\,(\d_{\kx}\ucU,\ \d_{\kp}\ucU)\,;
\end{align*}
\item $(\beta^{(j)}(\xi))_{1\leq j\leq 4}$ and $(\tbeta^{(j)}(\xi))_{1\leq j\leq 4}$ are dual bases of $\bC^4$ that are analytic in $\xi$ and such that $(\beta^{(j)}(0))_{1\leq j\leq 4}$ and $(\tbeta^{(j)}(0))_{1\leq j\leq 4}$ are dual right and left eigenbases of
$\ucx\,\I_4+\ubW$ associated with eigenvalues $(a_0^{(j)})_{1\leq j\leq 4}$ labeled so that
$$
\lambda_j(\xi)\,\stackrel{\xi\to0}{=}\,\iD\ukx\xi a_0^{(j)}+\cO(|\xi|^3)\,,\quad
1\leq j\leq 4\,.
$$
\end{itemize}
\et

The way in which the eigenvalue $0$ of multiplicity $4$ breaks is highly non-generic from the point of view of abstract spectral theory. Indeed we already know from Corollary~\ref{c:W-cor} that the four arising eigenvalues are differentiable at $\xi=0$ and we obtain that when eigenvalues of $\ubW$ are distinct, the four eigenvalues of $\cL_{\xi,0}$ are analytic in $\xi$. This should be contrasted with the fact that eigenvalues arising from generic Jordan blocks of height $2$ are no better than $\frac12$-H\"older (and in particular are not Lipschitz).

\begin{proof}
We make extensive use of standard spectral perturbation theory as 
expounded at length in \cite{Kato}. To begin with, we introduce $
\lambda_0>0$ and $\xi_0>0$ such that for $|\xi|\leq \xi_0$ the
spectrum of $\cL_{\xi,0}$ in $B(0,\lambda_0)$ has multiplicity $4$ 
and denote by $\Pi_\xi$ the corresponding Riesz spectral projector. 
From \eqref{dprofile1} the range of $\Pi_0$ is spanned by 
$(\ucU_x,\bJ\,\ucU,\d_\impulse\ucU,\d_\mass\ucU)$ and from 
\eqref{dimpulse},\eqref{dmasse}, we may choose a 
dual basis of the range of $\Pi_0^*$ in the form $(*,**,
\delta\Impulse_1(\ucU,\eD_1(\ukx\d_x+\ukp\bJ)\ucU),\ \delta\Mass[\ucU])$. 
By Kato's perturbation method, we may extend these dual bases as 
dual bases $(q_j(\xi,\cdot))_{1\leq j\leq 4}$ and $(\tq_j(\xi,
\cdot))_{1\leq j\leq 4}$ of respectively the ranges of $\Pi_\xi$ 
and $\Pi_\xi^*$.

One may use the corresponding coordinates to reduce the study of the spectrum of $\cL_{\xi,0}$ to the consideration of the matrix
\[
\Lambda_\xi:=\bp
\langle \tq_{j}(\xi,\cdot);\cL_{\xi,0}\,q_{\ell}(\xi,\cdot)\rangle_{L^2}
\ep_{(j,\ell)\in\{1,2,3,4\}^2}\,.
\] 
From relations expounded above stems
\[
\Lambda_0
\,=\,\bp
0&0&-\ukx\,\d_{\impulse}\ucx&-\ukx\,\d_{\mass}\ucx\\
0&0&\d_{\impulse}\uomp-\ukp\,\d_{\impulse}\ucx&\d_{\mass}\uomp-\ukp\,\d_{\mass}\ucx\\
0&0&0&0\\
0&0&0&0
\ep\,.
\]
Note in particular that $\Lambda_0^2$ is zero. From 
\eqref{dprofile2} we also derive
\begin{align*}
\langle \tq_{j}(0,\cdot);\cL_{(1)}\,q_{\ell}(0,\cdot)\rangle_{L^2}
&=0\,,
&3\leq j\leq 4\,,&\quad 1\leq \ell\leq 2\,.
\end{align*}
Thus 
\begin{align}\label{e:singlambda}
\tLambda_\xi&:=\frac{1}{\iD\ukx\,\xi}P_\xi^{-1}\,\Lambda_\xi\,P_\xi\,,&
P_\xi&:=\bp \I_2&0\\0&\iD\ukx\,\xi\,\I_2\ep\,,&&
\xi\neq 0,
\end{align}
defines a matrix $\tLambda_\xi$ extending analytically to $\xi=0$.

Our main intermediate goal is to compute $\tLambda_0$. We first show that we may enforce 
\begin{align}\label{higherdual}
\d_\xi q_1(0,\cdot)&=\iD\ukx\,\d_{\kx}\ucU,& 
\d_\xi q_2(0,\cdot)&=\d_{\kp}\ucU\,.
\end{align}
To do so, for $\ell=1,2$, by expanding $\Pi_\xi(\cL_{\xi,0}q_\ell(\xi,\cdot))=\cL_{\xi,0}q_\ell(\xi,\cdot)$, we derive that $\cL_{0,0}\d_\xi q_\ell(0,\cdot)+\iD\ukx\cL_{(1)}q_\ell(0,\cdot)$ belongs to the range of $\Pi_0$. Comparing with equations for $\d_{\kx}\ucU$ and $\d_{\kp}\ucU$, we deduce that $\cL_{0,0}(\d_\xi q_1(0,\cdot)-\iD\ukx\d_{\kx}\ucU)$ and $\cL_{0,0}(\d_\xi q_2(0,\cdot)-\iD\ukx\d_{\kp}\ucU)$ thus also $\d_\xi q_1(0,\cdot)-\iD\ukx\d_{\kx}\ucU$ and $\d_\xi q_2(0,\cdot)-\iD\ukx\d_{\kp}\ucU$ belong to the range of $\Pi_0$. Let $(\alpha_j^{(\ell)})_{1\leq j\leq4,\,1\leq\ell\leq 2}$ be such that
\begin{align*}
\d_\xi q_1(0,\cdot)-\iD\ukx\d_{\kx}\ucU
&=\sum_{j=1}^4\alpha_j^{(1)} q_j(0,\cdot)\,,&
\d_\xi q_2(0,\cdot)-\iD\ukx\d_{\kp}\ucU
&=\sum_{j=1}^4\alpha_j^{(2)} q_j(0,\cdot)\,.
\end{align*}
Lessening $\xi_0$ if necessary, one may then replace $(q_j(\xi,\cdot))_{1\leq j\leq 4}$ with
\begin{align*}
q_1(\xi,\cdot)&-\xi\,\sum_{j=1}^4\alpha_j^{(1)} q_j(\xi,\cdot)\,,&
q_2(\xi,\cdot)&-\xi\,\sum_{j=1}^4\alpha_j^{(2)} q_j(\xi,\cdot)\,,&
q_3(\xi,\cdot)&\,,&
q_4(\xi,\cdot)&\,,
\end{align*}
and $(\tq_j(\xi,\cdot))_{1\leq j\leq 4}$ with
\begin{align*}
\tq_j(\xi,\cdot)&+\xi\,\sum_{\ell=1}^2\talpha_j^{(\ell)}(\xi)\tq_\ell(\xi,\cdot)\,,&
j=1,2,3,4\,,
\end{align*}
with $(\talpha_j^{(\ell)}(\xi))_{1\leq j\leq4,\,1\leq\ell\leq 2}$ tuned to preserve duality relations and have \eqref{higherdual},
that we assume from now on. To make the most of associated relations, we observe that from duality  stems
\begin{align*}
\langle \d_\xi\tq_{j}(0,\cdot);\,q_{\ell}(0,\cdot)\rangle_{L^2}&=
-\langle \tq_{j}(0,\cdot);\,\d_\xi q_{\ell}(0,\cdot)\rangle_{L^2}\,,
&\quad 1\leq j\,,\ell\leq 4\,.
\end{align*}

Since
\begin{align*}
(\tLambda_0)_{j,\ell}
&=
\langle \tq_{j}(0,\cdot);
\frac{1}{\iD\ukx}\cL_{0,0}\,\d_\xi\,q_{\ell}(0,\cdot)
+\cL_{(1)}\,q_{\ell}(0,\cdot)\rangle_{L^2}\,,
&1\leq j\leq 2\,,&\quad 1\leq \ell\leq 2\,,\\
(\tLambda_0)_{j,\ell}
&=(\Lambda_0)_{j,\ell}\,,
&1\leq j\leq 2\,,&\quad 3\leq \ell\leq 4\,,\\
\end{align*}
this gives readily from \eqref{dprofile2}
\begin{align*}
\bp(\tLambda_0)_{j,\ell}\ep_{1\leq j\leq 2,\ 1\leq\ell\leq4}
=\bp
-\ukx\,\d_{\kx}\ucx&-\ukx\,\d_{\kp}\ucx&-\ukx\,\d_{\impulse}\ucx&-\ukx\,\d_{\mass}\ucx\\
\d_{\kx}\uomp-\ukp\,\d_{\kx}\ucx&\d_{\kp}\uomp-\ukp\,\d_{\kp}\ucx
&\d_{\impulse}\uomp-\ukp\,\d_{\impulse}\ucx&\d_{\mass}\uomp-\ukp\,\d_{\mass}\ucx\\
\ep\,.
\end{align*}
The extra relations also carry 
\begin{align*}
(\tLambda_0)_{j,\ell}
&=
\langle \tq_{j}(0,\cdot);
\frac{1}{\iD\ukx}\cL_{(1)}\,\d_\xi\,q_{\ell}(0,\cdot)
+\cL_{(2)}\,q_{\ell}(0,\cdot)\rangle_{L^2}\\
&\qquad+(\umass\,\d_{k^{(\ell)}}\uomp-\uimpulse\d_{k^{(\ell)}}\ucx)\,\delta_j^3\,,
&3\leq j\leq 4\,,&\quad 1\leq \ell\leq 2\,,\\
(\tLambda_0)_{j,\ell}
&=\langle \tq_{j}(0,\cdot);
\cL_{(1)}\,q_{\ell}(0,\cdot)\rangle_{L^2}\\
&\qquad+(\umass\,\d_{m^{(\ell)}}\uomp-\uimpulse\d_{m^{(\ell)}}\ucx)\,\delta_j^3\,,
&3\leq j\leq 4\,,&\quad 3\leq \ell\leq 4\,,
\end{align*}
with $k^{(1)}=\kx$, $k^{(2)}=\kp$, $m^{(3)}=\impulse$, $m^{(4)}=\mass$. Using \eqref{highermasse}, \eqref{higherimpulse} to
evaluate the foregoing expressions leads to the final identification
\[
\tLambda_0\,=\,\ucx\,\I_4+\ubW\,.
\]

The proof is then completed by diagonalizing matrices $\tLambda_\xi$, that have simple eigenvalues provided that $\xi_0$ is taken sufficiently small, and undoing the various transformations.
\end{proof}

We would like to make a few comments on the foregoing proof.
\begin{enumerate}
\item Though this is useless for our purposes, one may also compute explicitly $\tq_1(0,\cdot)$ and $\tq_2(0,\cdot)$ as combinations of $\bJ\ucU_x$, $\ucU$, $\bJ\d_\impulse\ucU$ and $\bJ\d_\mass\ucU$. Indeed it follows from Hamiltonian duality that the four vectors form a basis of the range of $\Pi_0^*$ and their scalar products with $\ucU_x$, $\bJ\ucU$, $\d_\impulse\ucU$ and $\d_\mass\ucU$ are explicitly known.
\item The assumption that the eigenvalues of $\ubW$ are distinct is only used at the very end of the proof. Removing it, the arguments still give an alternative proof of the second part of Corollary~\ref{c:W-cor}. For semilinear equations, to some extent this has already been carried out in the recent \cite{LBJM} with a few variations that we point now.
\begin{enumerate}
\item The authors of \cite{LBJM} further assume that $\cL_{0,0}$ exhibits two Jordan blocks of height $2$ at $0$, in other words they assume that the above matrix $\Lambda_0$ has rank $2$.
\item In \cite{LBJM} no formal interpretation is provided for the underlying instability criterion. In particular no connection with geometrical optics and modulated systems is offered for the matrix $\tLambda_0$. This connection is established in a 
recent preprint \cite{Clarke-Marangell}, building upon \cite{LBJM}. Hence the next remarks also apply to \cite{Clarke-Marangell}.
\item The structure of eigenfunctions is left out of the discussion in \cite{LBJM}, whereas this is our main motivation for reproving in a different way the second part of Corollary~\ref{c:W-cor}. In turn, the main focus of \cite{LBJM} is on spectral stability and the authors supplement their analysis with numerical experiments for cubic and quintic semilinear equations.
\item We have taken advantage of the fact that we have already proven the first part of  Corollary~\ref{c:W-cor} to use modulation coordinates $(\kx,\kp,\impulse,\mass)$ whereas the analysis in \cite{LBJM} is carried out with phase-portrait parameters $(\mux,\cx,\omp,\mup)$. 
\end{enumerate}
\end{enumerate}

\medskip

In the remaining part of this section, since we are only discussing longitudinal perturbations, we assume $d=1$ for the sake of readability. Then, denoting by $(S(t))_{t\in\R}$ the group associated with the operator $\cL$ on $L^2(\R)$ and, for $\xi\in[-\pi,\pi]$, by $(S_\xi(t))_{t\in\R}$ the group associated with the operator $\cL_\xi$ on $L^2_{\per}((0,1))$, we note that
from Bloch inversion \eqref{inverse-Bloch} stems
\[
(S(t)g)(\bfx)\ =\ \int_{-\pi}^\pi\eD^{\iD\xi x}\ 
(S_\xi(t)\check{g}(\xi,\cdot))(x)\ \dD\xi\,.
\]
Our main concern here is to analyze the large-time dynamics for the slow side-band part of the evolution
\[
(\SSp(t)g)(\bfx)\ :=\ \int_{-\pi}^\pi\eD^{\iD\xi x}\ \chi(\xi)\,
(S_\xi(t)\,\Pi_\xi\,\check{g}(\xi,\cdot))(x)\ \dD\xi\,.
\]
where $\chi$ is a smooth cut-off function equal to $1$ on 
$[-\xi_0/2,\xi_0/2]$ and to $0$ outside of $[-\xi_0,\xi_0]$ with 
$\xi_0>0$ as in the statement of Theorem~\ref{th:W-Bloch} 
and $\Pi_\xi$ the associated spectral projector, as in the 
proof of Theorem~\ref{th:W-Bloch}.

Let us explain in which sense this is expected to be the principal part of the linearized evolution for suitably spectrally stable waves. As a first remark we point out that when considering general perturbations on $\R$ (as opposed to co-periodic 
perturbations) we have to abandon not only stability in its 
strongest sense that would require a control of $\|\bU-\ucU\|_X$ 
(in some functional space $X$ of functions 
defined on $\R$) but also orbital stability that here requires a control of 
\[
\inf_{(\vphip,\vphix)\in\RR^2}\left\|\eD^{-\vphip\bJ}\bU(\,\cdot-\vphix\,)\,-\,\ucU\right\|_X\,,
\]
and instead to adopt the notion of \emph{space-modulated stability} that is encoded by bounds on 
\[
\inf_{(\vphip,\vphix)\textrm{ functions on }\R}\left(\left\|\eD^{-\vphip(\cdot)\bJ}\bU(\,\cdot-\vphix(\cdot)\,)\,-\,\ucU\right\|_X
+\|\d_x\vphip\|_X
+\|\d_x\vphix\|_X\right)\,.
\]
Rather than bounding $\|\bV\|_X$ or 
\[\di \inf_{\substack{(\vphip,\vphix)\in\RR^2\\\bV=\vphip\,\bJ\ucU+\vphix\,\ucU_x+\tbV}}\|\tbV\|_X\,,
\]
at the linearized level this consists in trying to bound 
\[
N_X(\bV):=\inf_{\substack{(\vphip,\vphix)\textrm{ functions on }\R\\\bV=\vphip\,\bJ\ucU+\vphix\,\ucU_x+\tbV}}\left(\|\tbV\|_X
+\|\d_x\vphip\|_X
+\|\d_x\vphix\|_X\right)\,,
\]
Note that $N_X$ precisely quotients ``locally'' the unstable 
directions highlighted in (the proof of) Theorem 
\ref{th:W-Bloch}, so that $\vphip,\,\vphix$ should be thought
of as \emph{local parameters}.
We adapt here to the case with a two-dimensional symmetry group the nonlinear notion formalized in \cite{JNRZ-conservation} and its linearized counterpart introduced in \cite{R_linKdV}.  Both notions have been proved to be sharp, respectively for a large class of parabolic systems in \cite{JNRZ-conservation} and for the linearized Korteweg--de Vries equation in \cite{R_linKdV}. The reader is also referred to \cite{R,R_Roscoff} for some more intuitive descriptions of the notions at hand. 

\br
An incorrect choice of stability type would lead to a claim of instability in situations where the global shape is preserved but positions need to be resynchronized either uniformly in space in the orbitally stable case or in a slowly varying way in the space-modulated stable case. In the latter case, the underlying spurious growth is due to the presence of Jordan blocks in the spectrum and it results in departures from the background profile that are algebraic in time (when no space-dependent synchronization is allowed). Thus concluding to a \emph{genuine} instability either at the linear or nonlinear requires extra care in the analysis. See for instance \cite{DR2} for an example of the latter. Unfortunately, though it seems clear that some extra analysis could be carried out to fill this gap, the only general nonlinear instability result available so far \cite{Jin-Liao-Lin} is expressed as an instability for the strongest sense of stability.
\er

Following the lines of \cite{R_linKdV}, one expects that 
for suitably spectrally stable waves the 
following bounds  hold 
\begin{align*}
\|(S(t)-\SSp(t))\bV_0\|_{H^s(\R)}&\leq\,C_s\,N_{H^s(\R)}(\bV_0)\,,&
\quad t\in\R\,,s\in\N\,,\\
\|(S(t)-\SSp(t))\bV_0\|_{L^\infty(\R)}&\leq\,\frac{C}{|t|^{\frac12}}\,N_{L^1(\R)}(\bV_0)\,,&
\quad t\in\R\,,
\end{align*} 
(with constants independent of $(t,\bV_0)$). We shall not try to prove or even formulate more precisely the latter but the reader should keep in mind the claimed $|t|^{-1/2}$ decay so as to compare it with bounds below. In particular the conclusions of the next theorem contains that 
\begin{align*}
N_{L^\infty(\R)}(\SSp(t)\bV_0)&\leq\,\frac{C}{(1+|t|)^{\frac13}}\,N_{L^1(\R)}(\bV_0)\,,&
\quad t\in\R\,.
\end{align*}

\begin{theorem}{\emph{Slow modulation behavior.}}\label{th:behavior}
Under the assumptions of Theorem~\ref{th:W-Bloch} and with its 
set of notation, assume moreover that 
\begin{enumerate}
\item for any $\xi\in[-\xi_0,\xi_0]$, for $j\in\{1,2,3,4\}$, $\lambda_j(\xi)\in\iD\R$;
\item for $j\in\{1,2,3,4\}$, $\d_\xi^3\lambda_j(0)\neq0$.
\end{enumerate}
There exists $C$ such that for any $\bV_0$ such that $N_{L^1(\R)}(\bV_0)<\infty$, there exists 
local parameter functions $\varphi_x,\,\varphi_\phi,\,\impulse$ and 
$\mass$ such that
for any time $t\in\R$ 
\begin{align*}
\ds
\big\|\SSp(t)\,(\bV_0)&-\ds
\vphix(t,\cdot)\,\ucU_x
\,-\,\vphip(t,\cdot)\,\bJ\ucU\\
&-\dd_{\kx,\kp,\impulse,\mass}\ucU\cdot(\ukx\d_x\vphix(t,\cdot),\ukx\d_x\vphip(t,\cdot),\impulse(t,\cdot),\mass(t,\cdot))
\big\|_{L^\infty(\R)}\\[0.5em]
&\leq
\frac{C}{(1+|t|)^{\frac12}}\,N_{L^1(\R)}(\bV_0)\ds
\end{align*}
where $\vphix(t,\cdot)$ and $\vphip(t,\cdot)$ are centered, $\vphix(t,\cdot)$, $\vphip(t,\cdot)$, $\impulse(t,\cdot)$, and 
$\mass(t,\cdot)$ are low-frequency\footnote{In the 
sense that their (distributional) Fourier transform has compact 
support that could be taken arbitrarily close to the origin.}, 
and 
\[
\|(\ukx\d_x\vphix(t,\cdot),\ukx\d_x\vphip(t,\cdot),\impulse(t,\cdot),\mass(t,\cdot))\|_{L^\infty(\R)}
\,\leq\,\frac{C}{(1+|t|)^{\frac13}}\,N_{L^1(\R)}(\bV_0)\,.\ds
\]
\end{theorem}

We omit the proof of Theorem~\ref{th:behavior} since with Theorem~\ref{th:W-Bloch} in hands, the proof is identical to the one of the corresponding result in \cite{R_linKdV}. Theorem~\ref{th:W-Bloch} is the counterpart of \cite[Proposition~2.1]{R_linKdV}, while 
Theorem~\ref{th:behavior} is a low frequency version of \cite[Theorem~1.3]{R_linKdV} (which is why the decay factor is bounded at $t=0$); see in particular \cite[Propositions~3.2 \&~4.2]{R_linKdV}. Yet we would like to add some comments.
\begin{enumerate}
\item A choice of local parameters can be given explicitly :
\begin{align*}
\bp\ukx\d_x\vphix(t,\cdot)\\\ukx\d_x\vphip(t,\cdot)\\\impulse(t,\cdot)\\\mass(t,\cdot)\ep(x)
&\,=\,\spp(t)(\bV_0)(x)\\
&:=\sum_{j=1}^4\int_{-\pi}^\pi\eD^{\iD\xi x+\lambda_j(\xi)\,t}\ \chi(\xi)\,
\beta^{(j)}(\xi)\,\langle \tpsi_j(\xi,\cdot),\widecheck{\bV_0}\,(\xi,\cdot)\rangle_{L^2}\ \dD\xi\,,
\end{align*}
This is motivated by the explicit diagonalization of $\cL_{\xi}\Pi_\xi$ from 
Theorem~\ref{th:W-Bloch} which implies
\begin{align*}
S_\xi(t)\,\Pi_\xi\,g
\,=\,&\frac{1}{\iD\ukx\xi}\left(\sum_{j=1}^4\eD^{\lambda_j(\xi)\,t}
\beta_{1}^{(j)}(\xi)\,\langle \tpsi_j(\xi,\cdot),g\rangle_{L^2}\right)\,q_1(\xi,\cdot)\\
&+\frac{1}{\iD\ukx\xi}\left(\sum_{j=1}^4\eD^{\lambda_j(\xi)\,t}
\beta_{2}^{(j)}(\xi)\,\langle \tpsi_j(\xi,\cdot),g\rangle_{L^2}\right)\,q_2(\xi,\cdot)\\
&+\left(\sum_{j=1}^4\eD^{\lambda_j(\xi)\,t}
\beta_{3}^{(j)}(\xi)\,\langle \tpsi_j(\xi,\cdot),g\rangle_{L^2}\right)\,q_3(\xi,\cdot)\\
&+\left(\sum_{j=1}^4\eD^{\lambda_j(\xi)\,t}
\beta_{4}^{(j)}(\xi)\,\langle \tpsi_j(\xi,\cdot),g\rangle_{L^2}\right)\,q_4(\xi,\cdot)
\end{align*}
the choice of local parameters is then dictated by analyzing the 
various expressions (including remainders) arising from expansions 
with respect to $\xi$ of $q_1(\xi,\cdot)$, $q_2(\xi,\cdot)$ at order $2$ and $q_3(\xi,\cdot)$, $q_4(\xi,\cdot)$, at order $1$.
One may also replace $\chi$ with a cut-off function with support closer to the origin if required.
\item Note that, since $(\d_\xi\lambda_j(0))_{j\in\{1,2,3,4\}}$ are two-by-two distinct, assuming that for any $\xi\in[-\xi_0,\xi_0]$ and any $j$, $\lambda_j(\xi)\in\iD\R$, from the Hamiltonian symmetry of the spectrum one derives that for $|\xi|\leq\xi_0$ ($\xi_0$ sufficiently small) and any $j\in\{1,2,3,4\}$  $\lambda_j(\xi)=\overline{\lambda_j(-\xi)}=-\lambda_j(-\xi)$. In particular, for any $j\in\{1,2,3,4\}$, $\lambda_j(\cdot)$ is an odd function and thus $\d_\xi^2\lambda_j(0)=0$. Therefore the assumption that for $j\in\{1,2,3,4\}$, $\d_\xi^3\lambda_j(0)\neq0$, expresses that the dispersive effects on local parameters are as strong as possible. In contrast the $|t|^{-1/2}$ decay claimed for the leftover part $S(t)-\SSp(t)$ is expected to be derivable from the assumption that outside the origin $(\lambda,\xi)=(0,0)$ second-order derivatives with respect to $\xi$ of spectral curves do not vanish.
\item For the semilinear defocusing cubic Schr\"odinger equation full spectral stability under longitudinal perturbations is known for all the waves and we expect that the remaining assumptions may be checked by reliable elementary numerics by using explicit formula for spectra obtained in \cite{Bottman-Deconinck-Nivala}.
\end{enumerate}

Theorem~\ref{th:behavior} essentially proves that $\SSp(t)(\bV_0)$ fits well with a large-time linearized version of the \emph{ansatz} \eqref{e:ansatz} with $\cU_{0}(T,X;\cdot)$ being a periodic wave profile of parameters such that $\kx=\d_X\vphix$ and $\kp=\d_X\vphip$. We now prove that some version of \eqref{e:W-formal} drives the evolution of local parameters $(\ukx\d_x\vphix,\ukx\d_x\vphip,\impulse,\mass)$ of Theorem~\ref{th:behavior}. We need to modify \eqref{e:W-formal} so as to account for dispersive effects. 

Let $P_0$ diagonalize $\ubW$ so that $P_0=\bp\beta^{(1)}(0)&\beta^{(2)}(0)&\beta^{(3)}(0)&\beta^{(4)}(0)\ep$
\begin{align*}
P_0^{-1}&=\bp\tbeta^{(1)}(0)\\\tbeta^{(2)}(0)\\\tbeta^{(3)}(0)\\\tbeta^{(4)}(0)\ep\,,&
P_0^{-1}\,\ubW\,P_0
&=\diag(a_0^{(1)}-\ucx,a_0^{(2)}-\ucx,a_0^{(3)}-\ucx,a_0^{(4)}-\ucx)\,,
\end{align*}
and define for $q$ an integer 
\[
\uD_q(\xi):=P_0\diag(\lambda_1^{[q]}(\xi)-a_0^{(1)}\iD\ukx\xi,\lambda_2^{[q]}(\xi)-a_0^{(2)}\iD\ukx\xi,\lambda_3^{[q]}(\xi)-a_0^{(3)}\iD\ukx\xi,\lambda_4^{[q]}(\xi)-a_0^{(4)}\iD\ukx\xi)P_0^{-1}
\]
where $\lambda_j^{[q]}(\xi)$ is the $q$th order Taylor expansion of $\lambda_j(\xi)$ near $0$. By convention we also include the pseudo-differential case where $q=\infty$ by choosing $\lambda_j^{(\infty)}$ as a smooth purely imaginary-valued function that coincides with $\lambda_j$ in a neighborhood of zero. Then consider the higher-order linearized modulation system
\be\label{e:Wlin-qth}
\d_t\bp\kx\\\kp\\\impulse\\\mass\ep
\,=\,\ukx\,\left(\ubW+\ucx\I_4\right)\d_x\bp\kx\\\kp\\\impulse\\\mass\ep
+\uD_q(\iD^{-1}\d_x)\bp\kx\\\kp\\\impulse\\\mass\ep\,.
\ee
Note that when $q=3$, $\uD_q(\iD^{-1}\d_x)$ takes the form $\ubD_3(\ukx\d_x)^3$ where $\ubD_3$ is a real-valued matrix.

\begin{theorem}{\emph{Averaged dynamics.}}\label{th:qth}
Let $q$ be an odd integer larger than $1$, or $q=\infty$. Under the assumptions of Theorem~\ref{th:behavior}, there exist $C$ and a cut-off function $\tchi$ such that for any $\bV_0$ such that $N_{L^1(\R)}(\bV_0)<\infty$ there exist $(\vphix^{(0)},\vphip^{(0)})$ centered and low-frequency such that with $\tbV_0:=\bV_0-\vphix^{(0)}\,\ucU_x\,-\,\vphip^{(0)}\,\bJ\ucU$ 
\[
\|\tbV_0\|_{L^1(\R)}\,+\,\|\d_x\vphix^{(0)}\|_{L^1(\R)}
\,+\,\|\d_x\vphip^{(0)}\|_{L^1(\R)}\,\leq\,2\,N_{L^1(\R)}(\bV_0)
\]
and for any such $(\vphix^{(0)},\vphip^{(0)})$ the local parameters $(\ukx\d_x\vphix,\ukx\d_x\vphip,\impulse,\mass)$  of Theorem~\ref{th:behavior} may be chosen in such a way that
with
\[
\bp\kx^{(0)}\\\kp^{(0)}\\\impulse^{(0)}\\\mass^{(0)}\ep
:=\tchi(\iD^{-1}\d_x)\bp\ukx\d_x\vphix^{(0)}\\\ukx\d_x\vphip^{(0)}\\
\delta\Impulse(\ucU,\ukx\d_x\ucU+\ukp\bJ\ucU)\,\tbV_0\\
\delta\Mass[\ucU]\,\tbV_0\ep
\]
for any time $t\in\R$
\begin{align*}
\big\|(\ukx\d_x\vphix(t,\cdot),\ukx\d_x\vphip(t,\cdot),\impulse(t,\cdot),\mass(t,\cdot))
&\ds-
\SigW_q(t)(\kx^{(0)},\kp^{(0)},\impulse^{(0)},\mass^{(0)})
\big\|_{L^\infty(\R)}\\
&\ds\,\leq\,
\frac{C}{(1+|t|)^{\frac{q+1}{2(q+2)}}}
\,N_{L^1(\R)}(\bV_0)
\end{align*}
and 
\begin{align*}
\|(\vphix(t,\cdot),\vphip(t,\cdot))&-\bp\beD_1\\\beD_2\ep\cdot\,(\ukx\d_x)^{-1}\SigW_q(t)(\kx^{(0)},\kp^{(0)},\impulse^{(0)},\mass^{(0)})\|_{L^\infty(\R)}\\
&\,\leq\,C\,N_{L^1(\R)}(\bV_0)
\,\begin{cases}\ds\frac{1}{(1+|t|)^{\frac13}}&\textrm{ if }q\geq 5\\
\ds\frac{1}{(1+|t|)^{\frac15}}&\textrm{ if }q=3
\end{cases}
\end{align*}
where $\SigW_q$ is the solution operator to System~\eqref{e:Wlin-qth}.
\end{theorem}

Again we omit the proof as nearly identical to the one of the corresponding result in \cite{R_linKdV}, namely Theorems $1.4$ ($q=3$) and $1.5$ ($q>3$), but provide a few comments.
\begin{enumerate}
\item In the case $q=3$, one may drop the low-frequency cut-off $\tchi(\iD^{-1}\d_x)$ (provided one restricts to times $|t|\geq 1$) since it is here only to compensate for the fact that when $q\geq5$ one cannot infer good dispersive properties of the Taylor expansions globally in frequency. This is somehow analogous to the fact that slow expansions of well-behaved parabolic systems may produce ill-posed systems. Similar estimates hold for $q=1$ but are somewhat pointless
since the decay rate is then the same as in Theorem~\ref{th:behavior}.
\item If one is willing to use less explicit and pleasant formula for $(\kx^{(0)},\kp^{(0)},\impulse^{(0)},\mass^{(0)})$ then the description of $(\vphix(t,\cdot),\vphip(t,\cdot))$ may actually be achieved up to an error of size $(1+|t|)^{-\frac{q+1}{2(q+2)}}$;
see \cite[Theorem~1.6]{R_linKdV}.
\item If one removes the assumption that $(\vphix^{(0)},\vphip^{(0)})$ is low-frequency then the formula for $(\kx^{(0)},\kp^{(0)},\impulse^{(0)},\mass^{(0)})$ should be modified as
\[
\bp\kx^{(0)}\\\kp^{(0)}\\\impulse^{(0)}\\\mass^{(0)}\ep
:=\tchi(\iD^{-1}\d_x)\bp\ukx\d_x\vphix^{(0)}\\\ukx\d_x\vphip^{(0)}\\
\delta\Impulse(\ucU,\ukx\d_x\ucU+\ukp\bJ\ucU)\,\tbV_0
-(\Mass[\ucU]-\umass)\,\left(\ukp\d_x\vphix^{(0)}-\ukx\d_x\vphip^{(0)}\right)\\
\delta\Mass[\ucU]\,\tbV_0
-(\Mass[\ucU]-\umass)\,\d_x\vphix^{(0)}\ep\,.
\]
To illustrate how the high-frequency corrections arise, let us 
point out that
\[
\langle\delta\Mass[\ucU];\widecheck{\vphix^{(0)}\ucU_x}(\xi,\cdot)\rangle_{L^2}
\,=\,-[\widehat{(\Mass[\ucU]-\umass)\,\d_x\vphix^{(0)}}](\xi)
+\iD\xi\,\langle\Mass[\ucU];\widecheck{\vphix^{(0)}}(\xi,\cdot)-\widehat{\vphix^{(0)}}(\xi)\rangle_{L^2}
\]
and that extra $\xi$-factors bring extra decay.
\item We expect that in the case $q=3$ System~\eqref{e:Wlin-qth} could be derived from higher-order versions of geometrical optics as in \cite{Noble-Rodrigues}. In contrast, the formal derivation of either System~\eqref{e:Wlin-qth} in the general case or of effective data  $(\kx^{(0)},\kp^{(0)},\impulse^{(0)},\mass^{(0)})$  (in particular when $(\vphix^{(0)},\vphip^{(0)})$ is not low-frequency) seems out of reach.
\item The foregoing construction of $D_q$ follows closely the classical construction of artificial viscosity systems as large-time asymptotic equivalents to systems that are only parabolic in the hypocoercive sense of Kawashima. We refer the reader for instance to \cite[Section~6]{Hoff_Zumbrun-NS_compressible_pres_de_zero}, \cite{Rodrigues-compressible}, \cite[Appendix~B]{JNRZ-conservation} or \cite[Appendix~A]{R} for a description of the latter. A notable difference however is that in the diffusive context higher-order expansions of dispersion relations beyond the second-order necessary to capture some dissipation does not provide any sharper description since the second-order expansion already provides the maximal rate compatible with a first-order expansion of eigenvectors. Here one needs to use the full pseudo-differential dispersion relations so as to reach a description up to $\cO(|t|^{-1/2})$ error terms.
\item We infer from Theorems~\ref{th:behavior} and~\ref{th:qth} that at leading order the behavior of $\SSp(t)(\bV_0)$ is captured by a linear modulation of phases $\vphix(t,\cdot)\,\ucU_x+\vphip(t,\cdot)\,\bJ\,\ucU$ and the phase shifts $(\vphix,\vphip)$ are the antiderivative of the two first components of a four-dimensional vector $(\ukx\d_x\vphix,\ukx\d_x\vphip,\impulse,\mass)$ that itself is at leading-order a sum of four linear dispersive waves of Airy type, each one traveling with its own velocity. In particular, three scales coexist: the oscillation of the background wave at scale $1$ in $\ucU_x$ and $\bJ\,\ucU$, spatial separation of the four dispersive waves at linear hyperbolic scale $t$, width of Airy waves of size $t^{1/3}$. We refer the reader to \cite{BJNRZ-KS,R_linKdV} for enlightening illustration by direct simulations of similar multi-scale large-time dynamics.
\end{enumerate}

\section{General perturbations}\label{s:general}

We now come back to the general spectral stability problem. We begin with a corollary to Theorem~\ref{th:low-freq}, from which we recall the following key formula. The Evans function  $D_{\xi}(\lambda,\bfeta)$ expands as
\begin{align*}
D_\xi(\lambda,\bfeta)
\stackrel{(\lambda,\bfeta)\to(0,0)}{=}
&\det\begin{pmatrix}\lambda\,\Sigma_t-(\eD^{\iD\xi}-1)\I_4+\frac{\|\bfeta\|^2}{\lambda}\Sigma_{\bfy}\end{pmatrix}\\
&+\cO\left((|\lambda|+|\xi|+\|\bfeta\|^2)\,
(|\lambda|^2+|\xi|^2+\|\bfeta\|^2)\,(|\lambda|(|\lambda|+|\xi|)+\|\bfeta\|^2)\right)\,.
\end{align*} 
For later use, we introduce the homogeneous fourth-order polynomial with real coefficients\footnote{The formula being extended by continuity to incorporate the cases when $\lambda=0$.}
\be\label{def:Delta}
\Delta_0(\lambda,z,\zeta)
:=\det\begin{pmatrix}
\lambda\,\Sigma_t-z\,\I_4+\frac{\zeta^2}{\lambda}\Sigma_{\bfy}
\end{pmatrix}\,.
\ee
That the coefficients of $\Delta_0$ are real may be seen directly or related to the fact that $D_0(\lambda,\bfeta)$ is real when $\lambda$ and $\bfeta$ are real. Likewise, note that for any $(\eps,\lambda,z,\zeta)\in\C^4$
\begin{align*}
\Delta_0(\lambda,z,-\zeta)&=\Delta_0(\lambda,z,\zeta)\,,&
\Delta_0(-\lambda,-z,\zeta)&=\Delta_0(\lambda,z,\zeta)\,,&\\
\Delta_0(\eps\lambda,\eps z,\eps\zeta)&=\eps^4\Delta_0(\lambda,z,\zeta)\,,&
\Delta_0(\overline{\lambda},\overline{z},\overline{\zeta})&=
\overline{\Delta_0(\lambda,z,\zeta)}\,.
\end{align*}
The second and the fourth relations are inherited from original real and Hamiltonian symmetries.

Since the longitudinal perturbations have already been analyzed at length, the following corollary focuses on perturbations that do have a transverse component. 

\bc\label{c:transverse}
Consider an unscaled wave profile $\ucV$ such that $\ucV\cdot \ucV_x\not\equiv0$.
\begin{enumerate}
\item Assume that there exists $(\lambda_0,\zeta_0)\in \R^2$, such that $\zeta_0\neq0$ and $\Delta_0(\lambda_0,0,\zeta_0)<0$. Then the wave is spectrally exponentially unstable to perturbations that are longitudinally co-periodic and transversally arbitrarily slow, that is, $\cL_{0,\bfeta}$ has eigenvalues of positive real part for 
$\bfeta$ arbitrarily small but nonzero.
\item Assume that there exists $(\lambda_0,\xi_0,\zeta_0)\in \C\times \R^2$, such that $\zeta_0\neq0$ and $\lambda_0$ is a root of $\Delta_0(\cdot,\iD\xi_0,\zeta_0)$ of algebraic multiplicity~$r$. Then there exist $C_0$ and $\eta_0>0$ such that for any $\bfeta$ such that $0<\|\bfeta\|\leq\eta_0$, 
\[
\cL_{\frac{\|\bfeta\|}{|\zeta_0|}\xi_0,\bfeta}
\] 
possesses $r$ eigenvalues (counted with algebraic multiplicity) in the disk 
\[
B\left(\frac{\|\bfeta\|}{|\zeta_0|}\lambda_0\,,\,C_0\|\bfeta\|^{1+\frac1r}\right)\,.
\] 
In particular if $\Re(\lambda_0)\neq0$ then the wave is spectrally unstable.
\end{enumerate}
\ec

\begin{proof}
By using symmetries of $\Delta_0$ we may assume that $\lambda_0\geq0$ and $r_0=1$. Then since Theorem~\ref{th:low-freq} ensures
\[
D_0(\|\bfeta\|\,\lambda_0,\bfeta)\,\stackrel{\|\bfeta\|\to0}{=}\,
\|\bfeta\|^4\,\Delta_0(\lambda_0,1)+\cO(\|\bfeta\|^5)
\]
we deduce that there exists $\bfeta_0>0$ such that for any $0<\|\bfeta\|\leq \bfeta_0$, $D_0(\|\bfeta\|\,\lambda_0,\bfeta)<0$. Comparing with Proposition~\ref{p:hig-freq} this yields that when $0<\|\bfeta\|\leq \bfeta_0$, the spectrum of $\cL_{0,\bfeta}$ intersects $(0,+\infty)$. 

The second part stems from a counting root argument based directly on Theorem~\ref{th:low-freq} and the symmetries of $\Delta_0$ .
\end{proof}

At this stage, two more comments are worth stating.
\begin{enumerate}
\item Since both Theorem~\ref{th:low-freq} and Proposition~\ref{p:hig-freq} include the case $\xi=\pi$ besides the case $\xi=0$, one may obtain a $\xi=\pi$ counterpart to the first parts of Theorem~\ref{th:co-periodic} and Corollary~\ref{c:transverse}. Yet the corresponding  instability criteria are never met since $\Delta_0(0,-2,0)>0$. For a similar reason, though Theorem~\ref{th:low-freq} deals with arbitrary Floquet exponents $\xi$, the second parts of Corollaries~\ref{c:W-cor} and~\ref{c:transverse} involve Floquet exponents converging to $\xi=0$. This is due to the fact that $\Delta_0(0,(\eD^{\iD\xi}-1),0)=0$ if and only if $\xi\in 2\pi\Z$. 
\item We point out that to some extent the restriction to $\xi_0=0$ of the second part of the corollary has already been derived in the recent \cite{LBJM} with a few variations that we point now.
\begin{enumerate}
\item The authors restrict to semilinear equations, a fact that comes with quite a few algebraic simplifications in computations. 
\item They further assume that $\cL_{0,0}$ exhibits two Jordan blocks of height $2$ at $0$.
\item Their proof goes by direct spectral perturbation of $\cL_{0,0}$ rather than Evans functions computations.
\end{enumerate}
\end{enumerate}

\subsection{Geometrical optics}\label{s:WKB}

Prior to studying at length properties of $\Delta_0$, we show that the latter may be derived from a suitable version of geometrical optics \emph{\`a la} Whitham. To start bridging the gap, we first recall \eqref{e:sigmatoTheta} and observe that 
\be \label{evanstowhitham}
\lambda\,\Sigma_t-z\I_4
+\frac{\zeta^2}{\lambda}\Sigma_\bfy
\,=\,\begin{pmatrix}
0&0&0&1\\
1&0&0&0\\
0&0&1&0\\
0&1&0&0
\end{pmatrix}\,
\left(
\lambda\,\bA_0
\,\Hess
\,\Theta
-z\,\bB_0
+\frac{\zeta^2}{\lambda}\,\bC_0
\right)
\,\begin{pmatrix}
0&0&0&1\\
0&1&0&0\\
1&0&0&0\\
0&0&1&0
\end{pmatrix}\,,
\ee
with 
\begin{align*}
\bA_0&=\begin{pmatrix}
       \I_2 & 0\\
       0 & -\I_2 
      \end{pmatrix}\,,&
\bB_0&=
\begin{pmatrix}
 0 & 1 & 0 & 0\\
 1 & 0 & 0 & 0\\
 0 & 0 & 0 & 1\\
 0 & 0 & 1 & 0 
\end{pmatrix}\,,& 
\bC_0
&:=\begin{pmatrix}0&0&0&0\\
0&-\sigma_3
&\sigma_2&0\\
0&-\sigma_2
&\sigma_1&0\\
0&0&0&0
\end{pmatrix}\,,
\end{align*}
where
\begin{align}\label{e:defsigma}
\sigma_1&:=\int_0^{\uXx}\kappa(\|\ucV\|^2)\,\|\ucV\|^2\,,&
\sigma_2&:=\int_0^{\uXx}\kappa(\|\ucV\|^2)\,\bJ\ucV\cdot\ucV_x\,,&
\sigma_3&:=\int_0^{\uXx}\kappa(\|\ucV\|^2)\,\|\ucV_x\|^2\,.
\end{align}
In particular
\be\label{e:ABC}
\Delta_0(\lambda,z,\zeta)
\,=\,\det
\left(
\lambda\,\bA_0
\,\Hess
\,\Theta
-z\,\bB_0
+\frac{\zeta^2}{\lambda}\,\bC_0
\right)\,.
\ee

Now let us start formal asymptotics with a multi-dimensional \emph{ansatz} similar to \eqref{e:ansatz}
\be\label{e:ansatztransverse}
\bU^{(\eps)}(t,\bfx)
\,=\,\eD^{\frac{1}{\eps}\vphip^{(\eps)}(\eps\,t,\eps\,\bfx)\,\bJ}\cU^{(\eps)}\left(\eps\,t,\eps\,\bfx;
\frac{\vphix^{(\eps)}(\eps\,t,\eps\,\bfx)}{\eps}\right)
\ee 
with
\begin{align*}
\cU^{(\eps)}(T,\bX;\zeta)&=
\cU_{0}(T,\bX;\zeta)+\eps\,\cU_{1}(T,\bX;\zeta)+o(\eps)\,,\\
\vphip^{(\eps)}(T,\bX)&=
(\vphip)_{0}(T,\bX)+\eps\,(\vphip)_{1}(T,\bX)+o(\eps)\,,\\
\vphix^{(\eps)}(T,\bX)&=
(\vphix)_{0}(T,\bX)+\eps\,(\vphix)_{1}(T,\bX)+o(\eps)\,,
\end{align*}
with $\cU_{0}(T,\bX;\cdot)$ and $\cU_{1}(T,\bX;\cdot)$ $1$-periodic. Inserting \eqref{e:ansatztransverse} in \eqref{e:ab} yields at leading order
\begin{eqnarray*}
(\d_T(\vphip)_0\bJ+\d_T(\vphix)_0\d_\zeta )\cU_0
=\bJ\delta \cH_0\left(\cU_0,\left(\nabla_\bX(\vphip)_0\bJ+\nabla_\bX(\vphix)_0\d_\zeta
\right)\cU_0\right)\,,
\end{eqnarray*}
so that for each $(T,\bX)$, $\cU_0(T,\bX;\cdot)$ must be a wave profile as in \eqref{def:wave-general} such that
\begin{align*}
\d_T(\vphip)_{0}&=\omp-\cx\kp\,,&
\nabla_\bX(\vphip)_{0}&=\bkp\,,&
\d_T(\vphix)_{0}&=\omx\,,&
\nabla_\bX(\vphix)_{0}&=\bkx\,.
\end{align*}
As a consequence $\bkx$ and $\bkp$ are curl-free and 
\begin{align}\label{e:Whi_schwarz}
\d_T\bkp&=\,\nabla_\bX\left(\omp-\kp\,\cx\right)\,,&
\d_T\bkx&=\,\nabla_\bX\omx\,.
\end{align}

Moreover, inserting \eqref{e:ansatztransverse} in \eqref{e:mass} and \eqref{e:momentum} yields at leading order respectively
\begin{align*}
\d_T(\Mass(\cU_0))=
\Div_\bX\Big(\bJ\cU_0\cdot\nabla_{\bU_{\bfx}}\Ham_0\left(\cU_0,
\left(\bkp\bJ+\bkx\d_\zeta
\right)\cU_0\right)\Big)+\d_\zeta(*)
\end{align*}
and
\begin{align*}
\d_T&\Big(\bImpulse(\cU_0,\left(\bkp\bJ+\bkx\d_\zeta
\right)\cU_0)\Big)\,=\,
\nabla_\bX\left(\frac12\bJ\cU_0\cdot\bJ\delta\Ham_0(\cU_0,\left(\bkp\bJ+\bkx\d_\zeta
\right)\cU_0)-\Ham_0(\cU_0,\left(\bkp\bJ+\bkx\d_\zeta
\right)\cU_0)\right)\\
&\quad+\sum_{\ell}\d_{X_\ell}\Big(\bJ\delta\,\bImpulse(\cU_0,\left(\bkp\bJ+\bkx\d_\zeta
\right)\cU_0)\cdot\nabla_{\bU_{X_\ell}}\Ham_0(\cU_0,\left(\bkp\bJ+\bkx\d_\zeta
\right)\cU_0)\Big)+\d_\zeta(**)
\end{align*}
with $*$ and $**$ $1$-periodic in $\zeta$. Averaging the foregoing equations using formulas in Section~\ref{s:profile-general} provides equations that combined with \eqref{e:Whi_schwarz} yield
\be\label{e:W-more} 
\left\{
\begin{array}{rcl}
\d_T\bkx&=&\nabla_\bX\omx\\
\d_T\bimpulse
&=&\nabla_\bX\left(\mux-\cx\impulse+\frac12\,\tau_0\,\|\tkp\|^2\right)\\
&&+\Div_\bX\left(\tau_1\,\tkp\otimes\tkp
+\tau_2\,(\tkp\otimes\ex+\ex\otimes\tkp)
+\tau_3\,\left(\ex\otimes\ex-\I_d\right)\right)\\
\d_T\mass
&=&\Div_\bX\left((\mup-\cx\mass)\,\ex+\,\tau_1\,\tkp\right)\\
\d_T\bkp&=&\nabla_\bX\left(\omp-\cx\,\kp\right)
\end{array}
\right.
\ee
with curl-free $\bkx$ and $\bkp$, where 
\begin{align*}
\tau_0&:=\langle \kappa'(\|\cU_0\|^2)\,\|\cU_0\|^2\rangle\,,& 
\tau_1&:=\langle\kappa(\|\cU_0\|^2)\,\|\cU_0\|^2\rangle\,,\\
\tau_2&:=\langle\kappa(\|\cU_0\|^2)\,\bJ\cU_0\cdot(\kp\,\bJ+\kx\d_\zeta)\cU_0\rangle\,,&
\tau_3&:=\langle\kappa(\|\cU_0\|^2)\,\|(\kp\,\bJ+\kx\d_\zeta)\cU_0\|^2\rangle\,.
\end{align*}

Before linearizing, in order to prepare System~\ref{e:W-more} for comparison, we recall that $\bkx=\kx\,\ex$, $\bkp=\kp\,\ex+\tkp$ and $\bimpulse=\impulse\,\ex+\mass\,\tkp$, with $\ex$ unitary and $\tkp$ orthogonal to $\ex$ and write \eqref{e:W-more} in terms of $(\kx,\kp,\mass,\impulse,\ex,\tkp)$. By using that, for any derivative $\d_\sharp$, 
\begin{align}\label{e:W-cancel}
\ex\cdot\d_\sharp\ex&=0\,,& 
\ex\cdot\d_\sharp\tkp&=-\tkp\cdot\d_\sharp\ex\,,
\end{align} 
this yields 
\be\label{e:W-more-bis} 
\left\{
\begin{array}{rcl}
\d_T\ex&=&\dfrac{1}{\kx}(\nabla_\bX-\ex\,\ex\cdot\nabla_\bX)\omx\\
\d_T\tkp&=&\hspace{-0.75em}(\nabla_\bX-\ex\,\ex\cdot\nabla_\bX)\left(\omp-\cx\,\kp\right)
-\dfrac{\kp}{\kx}(\nabla_\bX-\ex\,\ex\cdot\nabla_\bX)\omx
-\ex \,\dfrac{\tkp}{\kx}\cdot\nabla_\bX\omx\\
\d_T\kx&=&\ex\cdot\nabla_\bX\omx\\
\d_T\impulse
&=&\ex\cdot\nabla_\bX\left(\mux-\cx\impulse+\frac12\,\tau_0\,\|\tkp\|^2\right)
+\,\mass\,\dfrac{\tkp}{\kx}\cdot\nabla_\bX\omx\\
&&+\ex\cdot\Div_\bX\left(\tau_1\,\tkp\otimes\tkp
+\tau_2\,(\tkp\otimes\ex+\ex\otimes\tkp)
+\tau_3\,\left(\ex\otimes\ex-\I_d\right)\right)\\
\d_T\mass
&=&\Div_\bX\left((\mup-\cx\mass)\,\ex+\,\tau_1\,\tkp\right)\\
\d_T\kp&=&\ex\cdot\nabla_\bX\left(\omp-\cx\,\kp\right)
+\dfrac{\tkp}{\kx}\cdot\nabla_\bX\omx
\end{array}
\right.
\ee
with $\ex$ unitary, $\tkp$ orthogonal to $\ex$ and $\kx\,\ex$ and $\kp\,\ex+\tkp$ curl-free. System~\eqref{e:W-more-bis} may be simplified further by noticing that from curl-free conditions (and \eqref{e:W-cancel}) stem
\begin{align*}
(\nabla_\bX-\ex\,\ex\cdot\nabla_\bX)\kx&=\kx\,(\ex\cdot\nabla_\bX)\ex\,,\\
(\nabla_\bX-\ex\,\ex\cdot\nabla_\bX)\kp&=(\ex\cdot\nabla_\bX)\tkp
+\nabla_\bX\left(\transp{\ex}\right)\tkp
+\kp\,(\ex\cdot\nabla_\bX)\ex\,.
\end{align*}
This leaves as equivalent system
\be\label{e:W-more-ter} 
\left\{
\begin{array}{rcl}
\hspace{-1em}(\d_T+\cx(\ex\cdot\nabla_\bX))\ex\hspace{-1em}&=&-(\nabla_\bX-\ex\,\ex\cdot\nabla_\bX)\cx\\
\hspace{-1em}(\d_T+\cx(\ex\cdot\nabla_\bX))\tkp\hspace{-1em}&=&(\nabla_\bX-\ex\,\ex\cdot\nabla_\bX)\omp
-\ex \,\dfrac{\tkp}{\kx}\cdot\nabla_\bX\omx
-\cx\nabla_\bX\left(\transp{\ex}\right)\tkp\\
\d_T\kx&=&\ex\cdot\nabla_\bX\omx\\
\d_T\impulse
&=&\ex\cdot\nabla_\bX\left(\mux-\cx\impulse+\frac12\,\tau_0\,\|\tkp\|^2\right)
+\,\mass\,\dfrac{\tkp}{\kx}\cdot\nabla_\bX\omx\\
&&\hspace{-3.75em}+\ex\cdot\Div_\bX\left(\tau_1\,\tkp\otimes\tkp
+\tau_2\,(\tkp\otimes\ex+\ex\otimes\tkp)
+\tau_3\,\left(\ex\otimes\ex-\I_d\right)\right)\\
\d_T\mass
&=&\Div_\bX\left((\mup-\cx\mass)\,\ex+\,\tau_1\,\tkp\right)\\
\d_T\kp&=&\ex\cdot\nabla_\bX\left(\omp-\cx\,\kp\right)
+\dfrac{\tkp}{\kx}\cdot\nabla_\bX\omx
\end{array}
\right.
\ee
with $\ex$ unitary, $\tkp$ orthogonal to $\ex$ and $\kx\,\ex$ and $\kp\,\ex+\tkp$ curl-free.

Linearizing System~\eqref{e:W-more-ter} about the constant $(\umux,\ucx,\uomp,\umup,\beD_1,0)$ yields
\be\label{e:W-lin}
\left\{
\begin{array}{rcl}
\left(\d_T+\ucx\d_X\right)\ex
&\,=\,&-\left(\nabla_\bX-\beD_1\,\d_X\right)\cx\\
\left(\d_T+\ucx\d_X\right)\tkp
&\,=\,&\left(\nabla_\bX-\beD_1\,\d_X\right)\omp\\
\ukx\bA_0\Hess\Theta\,\left(\d_T+\ucx\d_X\right)
\begin{pmatrix}
\mux\\\cx\\\omp\\\mup
\end{pmatrix}
&\,=\,&\bB_0\,\d_X
\begin{pmatrix}
\mux\\\cx\\\omp\\\mup
\end{pmatrix}
+\bp
0&0\\
\utau_3&\utau_2\\
\utau_2&\utau_1\\
0&0
\ep\bp \Div_\bX(\ex)\\\Div_\bX(\tkp) \ep
\end{array}\right.
\ee
with extra constraints that $\tkp$ and $\ex$ are orthogonal to $\uex=\beD_1$ and that $\kx\uex+\ukx\ex$ and $\kp\uex+\ukp\ex+\tkp$ are curl-free, where $(\kx,\kp)$ are given explicitly as 
\begin{align}\label{e:kxp-lin}
\kx&=-\ukx^2\dD\,(\d_{\mux}\Theta)(\mux,\cx,\omp,\mup)\,,&
\kp&=\frac{\kx}{\ukx}\,\ukp
-\ukx\dD\,(\d_{\mup}\Theta)(\mux,\cx,\omp,\mup)\,,&
\end{align}
where total derivatives are taken with respect to $(\mux,\cx,\omp,\mup)$ and evaluation is at $(\umux,\ucx,\uomp,\umup,\beD_1,0)$. In System~\ref{e:W-lin}, likewise $\Hess\Theta=\Hess_{(\mux,\cx,\omp,\mup)}\Theta$ is evaluated at $(\umux,\ucx,\uomp,\umup,\beD_1,0)$. For the convenience of the reader, we detail some of the observations and manipulations used to go from \eqref{e:W-more-ter} to \eqref{e:W-lin}:
\begin{enumerate}
\item As in going from \eqref{e:W} to \eqref{e:W-action}, derivatives of $\Theta$ arise in \eqref{e:W-lin} and \eqref{e:kxp-lin} from \eqref{dtheta}.
\item As pointed out in Section~\ref{s:profile-general}, $\dD_{(\ex,\tkp)}\,(\nabla_{(\mux,\cx,\omp,\mup)}\Theta)$ vanish at $(\umux,\ucx,\uomp,\umup,\beD_1,0)$.
\item For any scalar-valued function $a$,
\[
\uex\cdot\Div_\bX\left(a\,\left(\uex\otimes\uex-\I_d\right)\right)\,=\,0\,.
\]
\item From orthogonality constraints  stem that for any derivative $\d_\sharp$, 
\begin{align*}
\uex\cdot\d_\sharp\ex&=0\,,& 
\uex\cdot\d_\sharp\tkp&=0\,,
\end{align*} 
so that
\begin{align*}
\uex\cdot\Div_\bX\left(\uex\otimes\ex+\ex\otimes\uex\right)
&=\Div_\bX(\ex)\,,\\
\uex\cdot\Div_\bX\left(\uex\otimes\tkp+\tkp\otimes\uex\right)
&=\Div_\bX(\tkp)\,.
\end{align*} 
\end{enumerate}

At this stage, noting that for $j\in\{1,2,3\}$, $\sigma_j=\utau_j/\ukx$, a few line manipulations achieve to prove the claimed relation between $\Delta_0$ and modulated systems in the form
\begin{align*}
&\lambda^{2(d-1)}\times \Delta_0(\lambda,\iD\xi,\|\bfeta\|)\\
&=\det\left(\lambda\,\bp \I_{2(d-1)}&0\\
0&\bA_0\,\Hess\,\Theta\ep
-\iD\xi\,
\bp 0&0\\
0&\bB_0\ep
+\left(\begin{array}{cc|cccc}
0&0&0&-\iD\bfeta&0&0\\
0&0&0&0&\iD\bfeta&0\\
\hline\\[-1.25em]
0&0&0&0&0&0\\
\frac{\utau_3}{\ukx}\,\iD\transp{\bfeta}&\frac{\utau_2}{\ukx}\,\iD\transp{\bfeta}
&0&0&0&0\\[0.5em]
\frac{\utau_2}{\ukx}\,\iD\transp{\bfeta}&\frac{\utau_1}{\ukx}\,\iD\transp{\bfeta}
&0&0&0&0\\
0&0&0&0&0&0
\end{array}\right)\right)\,.
\end{align*}

\subsection{Instability criteria}\label{s:criteria}

The rest of the section is devoted to the study of instability criteria contained in Corollary~\ref{c:transverse} and its longitudinal counterparts. 

We begin by rephrasing the main consequence of Corollary~\ref{c:transverse}. A stable wave must satisfy
\begin{enumerate}
\item for any $(\lambda,\zeta)\in \R^2$, $\Delta_0(\lambda,0,\zeta)\geq0$;
\item for any $(\xi,\zeta)\in \R^2$, the roots of the (real) 
polynomial $\omega\mapsto \Delta_0(\iD\omega,\iD\xi,\zeta)$ 
are real. 
\end{enumerate}
Note that the latter condition may be expressed explicitly as
inequality constraints on some polynomial expressions in $(\xi,\zeta)
\in \R^2$ but the involved expressions are rather cumbersome. 

The restriction to $\xi=0$ is much simpler to analyze. To do so, let 
us introduce notation for coefficients of $\Delta_0$
\be\label{eq:poly}
\Delta_0(\lambda,z,\zeta)
\,=\,\sum_{\substack{0\leq m,n,p \leq 4\\m+n+p=4\\p\leq m}}
\delta_{(m,n,p)}\lambda^{m-p}\,z^n\,\zeta^{2p}\,,
\ee
and note that
\[
\Delta_0(\lambda,0,\zeta)
\,=\,\lambda^4\,\delta_{(4,0,0)}
+\zeta^2\lambda^2\,\delta_{(3,0,1)}
+\zeta^4\,\delta_{(2,0,2)}\,.
\]
A straightforward consequence is the following stability 
condition.
\bl\label{l:xi=0}
If the wave is stable, then 
for any $\zeta\in \R^2$, the roots of the polynomial 
$\omega\mapsto \Delta_0(\iD\omega,0,\zeta)$ are real and
\begin{align}\label{ineq:xi=0}
\delta_{4,0,0}&\geq 0\,,&
\delta_{3,0,1}&\geq 2\,\sqrt{|\delta_{4,0,0}\,\delta_{2,0,2}|}\,,&
\delta_{2,0,2}&\geq0\,.&
\end{align}
Moreover if all the signs of the three inequalities in \eqref{ineq:xi=0} are strict then $\Delta_0(\lambda,0,\zeta)>0$
when $(\lambda,\zeta)\in \R^2\setminus\{0\}$ and the roots 
of $\omega\mapsto \Delta_0(\iD\omega,0,\zeta)$ are distinct 
when $\zeta\neq0$.
\el

Note that the mere combination of the cases $\bfeta=0$ --- corresponding to longitudinal perturbations --- (studied in Corollary~\ref{c:W-cor} and in \cite{LBJM}) and 
$\xi=0$ --- corresponding to longitudinally co-periodic perturbations --- (studied in Lemma~\ref{l:xi=0} and in \cite{LBJM}) is \emph{a priori} insufficient to capture the full strength of Corollary~\ref{c:transverse}. Indeed note that the associated instability criteria do not involve coefficients $\delta_{(1,2,1)}$ and $\delta_{(2,1,1)}$. To illustrate why we expect that these coefficients do matter, let us consider an abstract real polynomial $\pi_0$ of the form~\eqref{eq:poly} then
\begin{itemize}
\item fixing all coefficients except $\delta_{(2,1,1)}$ with $\delta_{4,0,0}\neq0$ and choosing some $(\xi_0,\zeta_0)\in (\R^*)^2$, it follows that if $|\delta_{(2,1,1)}|$ is sufficiently large then either $\omega\mapsto \pi_0(\iD\omega,\iD\xi_0,\zeta_0)$ or $\omega\mapsto \pi_0(\iD\omega,-\iD\xi_0,\zeta_0)$ possesses a non-real root; 
\item fixing all coefficients except $\delta_{(1,2,1)}$ with $\delta_{4,0,0}\neq0$ and choosing some $(\xi_0,\zeta_0)\in (\R^*)^2$, it follows that if $|\delta_{(1,2,1)}|$ is sufficiently large with $\delta_{(1,2,1)}<0$ then $\omega\mapsto \pi_0(\iD\omega,\iD\xi_0,\zeta_0)$ possesses a non-real root. 
\end{itemize}
A less pessimistic guess could be that the inspection of the regimes $|\xi|\ll |\zeta|$ and $|\zeta|\ll |\xi|$ (that are perturbations of $\xi=0$ and $\zeta=0$) could involve the missing coefficients and be sufficient to decide the instability criteria encoded by Corollary~\ref{c:transverse}. For an example of a closely related situation where such a scenario does occur, the reader is referred to \cite{Noble-Rodrigues,JNRZ-KdV-KS}. Yet the following lemma suggests that even this weaker claim can be expected only in degenerate situations. Let us also anticipate a bit and stress that the coefficient $\delta_{(1,2,1)}$ plays a deep role in our analysis of the small-amplitude regime.

\begin{lemma}
\begin{enumerate}
\item Assume that $\Sigma_t$ is non singular and that the eigenvalues of $\Sigma_t$ are real and distinct, or equivalently that $\Hess\Theta$ is non singular and that the characteric values of the modulation system are real and distinct. Then there exists $\eps_0>0$ such that when $(\xi,\zeta)\in \R^2\setminus\{(0,0)\}$ is such that $|\zeta|\leq \eps_0\,|\xi|$, the fourth-order real polynomial $\omega\mapsto \Delta_0(\iD\omega,\iD\xi,\zeta)$ possesses four distinct real roots.
\item Assume that $\Sigma_t$ is non singular and that the eigenvalues of $(\Sigma_t)^{-1}\Sigma_\bfy$ are positive and distinct. Then there exists $\eps_0>0$ such that when $(\xi,\zeta)\in \R^2\setminus\{(0,0)\}$ is such that $|\xi|\leq \eps_0\,|\zeta|$, the fourth-order real polynomial $\omega\mapsto \Delta_0(\iD\omega,\iD\xi,\zeta)$ possesses four distinct real roots.
\end{enumerate}
\end{lemma}

Recall that if $\Sigma_t$ is non singular and either $\Sigma_t$ possesses a non real eigenvalue or $(\Sigma_t)^{-1}\Sigma_\bfy$ posesses an eigenvalue in $\C\setminus[0,+\infty)$ then the associated wave is unstable.

\begin{proof}
The distinct character follows from continuity of polynomial roots 
(applied respectively at $\zeta=0$ and $\xi=0$). Then their reality is 
deduced from the stability by complex conjugation of the root set of real polynomials.
\end{proof}

More generally the transition to non real roots of $\omega\mapsto \Delta_0(\iD\omega,\iD\xi,\zeta)$ may only occur near a $(\xi_0,\zeta_0)$ where the polynomial possesses a multiple root. With this in mind, we now elucidate how the breaking of a multiple root occurs near $\xi=0$ and near $\zeta=0$.

\bpr[Breaking of a multiple root near 
$\bfeta=0$]\label{pr:splitting-eta0} Assume that $\Sigma_t$ is non singular and that $\omega_0$ is a real eigenvalue of $(\Sigma_t)^{-1}$ of algebraic multiplicity $r_0$. If 
\begin{align*}
\textrm{either }&&  
(\ r_0&\geq3& \textrm{and}&& 
\delta_{(1,2,1)}+\delta_{(2,1,1)}\omega_0+\delta_{(3,0,1)}\omega_0^2&\neq 0\ ),\\
\textrm{or }&&  
\Big(\ r_0&=2& \textrm{and}&& 
\frac{\delta_{(1,2,1)}+\delta_{(2,1,1)}\omega_0+\delta_{(3,0,1)}\omega_0^2}{\frac{1}{r_0!}\d_\lambda^{r_0}\Delta_0(\omega_0,1,0)}&< 0\ \Big),
\end{align*}
then the corresponding wave is spectrally unstable.
\epr

\begin{proof}
This follows from the the Taylor expansion of $\Delta_0$ 
\begin{eqnarray*}
\Delta_0(\lambda,\iD\xi,\|\bfeta\|)&=&(\lambda-\iD\xi\omega_0)^{r_0}
\frac{\xi^{4-r_0}\partial_\lambda^{r_0}
\Delta_0(\iD\omega_0,\iD,0)}{r_0!}
+\|\bfeta\|^2\bigg(\delta_{1,2,1}(\iD\xi)^2+\delta_{2,1,1}\lambda\iD\xi+
\delta_{3,0,1}\lambda^2\bigg)\\
&&+\cO\left(|\lambda-\iD\xi\omega_0|^{r_0+1}+\|\bfeta\|^4\right).
\end{eqnarray*}
Hence from Theorem \ref{th:low-freq} and a continuity argument
there are $r_0$ roots of $D_\xi(\cdot,\|\bfeta\|)$ near $\iD\xi\omega_0$ that expand as
\[
\iD\xi\,\omega_0+\iD\xi\,\cZ\,\left(\frac{\|\bfeta\|}{|\xi|}\right)^{\frac{2}{r_0}}
\left(\frac{\delta_{(1,2,1)}+\delta_{(2,1,1)}\omega_0+\delta_{(3,0,1)}\omega_0^2}{\frac{1}{r_0!}\d_\lambda^{r_0}\Delta_0(\omega_0,1,0)}\right)^{\frac{1}{r_0}}
+\cO\left(|\xi|\left(\frac{\|\bfeta\|}{|\xi|}\right)^{\frac{2}{r_0}}
\left(\left(\frac{\|\bfeta\|}{|\xi|}\right)^{\frac{2}{r_0}}
+\frac{|\xi|^3}{\|\bfeta\|^2}\right)\right),
\]
in the limit 
\[
\left(|\xi|,\frac{\|\bfeta\|}{|\xi|},\frac{|\xi|^3}{\|\bfeta\|^2}\right)\to(0,0,0)\,,
\]
where $\cZ$ runs over the $r_0$th roots of unity.
\end{proof}

\bpr[Breaking of a multiple root near $\xi=0$] \label{pr:splitting-xi0}
Assume that $\delta_{(4,0,0)}\neq0$ and $\delta_{(3,0,1)}^2=4\,\delta_{(4,0,0)}\,\delta_{(2,0,2)}$. If
\begin{align*}
\delta_{(2,1,1)}&\,\delta_{(4,0,0)}-\frac12\delta_{(3,1,0)}\,\delta_{(3,0,1)}\neq0
\end{align*}
then the corresponding wave is spectrally unstable.
\epr

\begin{proof}
From Lemma \ref{l:xi=0} stability requires $\delta_{(4,0,0)}>0$, $\delta_{(2,0,2)}\geq0$ and $\delta_{(3,0,1)}=2\,\sqrt{\delta_{(4,0,0)}\,\delta_{(2,0,2)}}$ and we assume this from now on. The polynomial $\Delta_0$ is then
$$
\Delta_0(\lambda,\iD\xi,\zeta)=\delta_{(4,0,0)}\left(\lambda^2+
\sqrt{\frac{\delta_{(2,0,2)}}{\delta_{(4,0,0)}}}\zeta^2\right)^2
+\iD\xi(\delta_{(2,1,1)}\lambda\zeta^2+\delta_{(3,1,0)}\lambda^3)
+\cO\left(\xi^2(\lambda^2+\zeta^2)).
\right)
$$
We begin with the case $\delta_{(2,0,2)}\neq0$. To analyze it we introduce
\[
\omega_0:=\bigg(\frac{\delta_{(2,0,2)}}{\delta_{(4,0,0)}}\bigg)
^{1/4}
 =\sqrt{\frac{\delta_{(3,0,1)}}{2\delta_{(4,0,0)}}}\,.
\]
Then for each $\sigma\in\{-1,1\}$, when $(\|\bfeta\|,|\xi|/\|\bfeta\|,\|\bfeta\|^3/|\xi|^2)$ is sufficiently small, there are $2$ roots of $D_\xi(\cdot,\|\bfeta\|)$ near $\iD\sigma\|\bfeta\|\omega_0$ that expand as
$$
\iD\sigma \|\bfeta\|\omega_0\left(1\pm \frac{1}{2\delta_{(4,0,0)}}
\sqrt{\frac{\sigma \xi}{\|\bfeta\|}\left(\delta_{(2,1,1)}
\delta_{(4,0,0)}-\frac{1}{2}\delta_{(3,1,0)}\delta_{(3,0,1)}
\right)
}
+\cO\left(\|\bfeta\|\sqrt{\frac{|\xi|}{\|\bfeta\|}}\left(
\frac{|\xi|}{\|\bfeta\|}+\frac{\|\bfeta\|^2}{|\xi|}
\right)\right)
\right),
$$
in the limit $\di \left(\|\bfeta\|,\,\frac{|\xi|}{\|\bfeta\|},
\, \frac{\|\bfeta\|^2}{|\xi|}\right)\to 0$.
\\
When $\delta_{(2,0,2)}=0$ --- thus also $\delta_{(3,0,1)}=0$ --- and $(\|\bfeta\|,|\xi|/\|\bfeta\|,\|\bfeta\|^2/|\xi|)$ is sufficiently small, $3$ of the $4$ roots of $D_\xi(\cdot,\|\bfeta\|)$ near $0$ expand as
\[
\iD\|\bfeta\|^{\frac23}\xi^{\frac13}\,\cZ
\,\left(\frac{\delta_{(2,1,1)}}{\delta_{(4,0,0)}}\right)^{\frac13}
+\cO\left(\|\bfeta\|^{\frac23}|\xi|^{\frac13}\left(
\left(\frac{|\xi|}{\|\bfeta\|}\right)^{\frac23}
+\frac{\|\bfeta\|^2}{|\xi|}\right)\right)
\]
where $\cZ$ runs over the $3$rd roots of unity,
in the limit 
\[
\left(\|\bfeta\|,\frac{|\xi|}{\|\bfeta\|},\frac{\|\bfeta\|^2}{|\xi|}\right)\to(0,0,0)\,,
\]

\end{proof}

\subsection{Large-period regime}\label{s:homoclinic}

We now examine consequences of Corollary~\ref{c:transverse} in asymptotic regimes described in from Section~\ref{s:asymp}. Since the large-period regime turns out to be significantly simpler to analyze, we begin the solitary-wave limit.

We prove the following theorem.

\bt\label{th:homoclinic_asymptotics}
When $d\geq2$, if $\d_{\cx}^2\Theta_{(s)}(\ucx^{(0)},\urho^{(0)},\ukp^{(0)})\neq0$ then, in the large period regime near $(\ucx^{(0)},\urho^{(0)},\ukp^{(0)})$, System~\ref{e:W-more} fails to be weakly hyperbolic and waves are spectrally exponentially unstable to transversally-slow longitudinally-co-periodic perturbations.
\et

Before turning to the proof of Theorem~\ref{th:homoclinic_asymptotics}, we would like to add one comment. It is natural to wonder whether the proved instability correspond to an instability of the limiting solitary-wave and even to expect that one could argue the other way around by proving the instability of solitary waves and deduce periodic-wave instability in the large-period regime by a spectral perturbation argument. When $\d_{\cx}^2\Theta_{(s)}(\ucx^{(0)},\urho^{(0)},\ukp^{(0)})<0$, this is a well-known fact. We expect this to be true under the assumptions of Theorem~\ref{th:homoclinic_asymptotics}. Yet so far general results for solitary-wave instabilities \cite{Benzoni-transverse,Rousset-Tzvetkov-linear} have been proven only for semilinear cases. More precisely, it has been proven for very specific forms of Schr\"odinger equations and for larger classes of Euler--Korteweg systems, sufficiently general to include all our semilinear cases. In the semilinear case, a different proof of Theorem~\ref{th:homoclinic_asymptotics} could thus be obtained by applying results from \cite{Benzoni-transverse,Rousset-Tzvetkov-linear} to solitary waves of the associated Euler--Korteweg systems, transferring those to large-period periodic waves of the same Euler--Korteweg systems through a suitable spectral perturbation theorem in the spirit of \cite{Gardner-large-period,Sandstede-Scheel-large-period,Yang-Zumbrun} and passing the latter to the Schr\"odinger systems by the results of Section~\ref{s:linearized-EK}. 

Let us stress that instead our proof goes by examining the large-period limit of a periodic-wave criterion. Incidentally we point out that our periodic-wave criterion is orthogonal to the arguments in \cite{Rousset-Tzvetkov-linear} but share some similarities with those in \cite{Benzoni-transverse}. For the adaptation as a periodic-wave criterion of the arguments of the former the reader is referred to \cite{HSS}.

One of the advantages in the way we have chosen is that it offers a symmetric treatment of both limits of interest, whereas the spectral perturbation argument fails in the harmonic limit. Another one is that we prove that the instability is of modulation type, being associated with failure of weak hyperbolicity of System~\ref{e:W-more}.

The rest of the present section is devoted to the proof of Theorem~\ref{th:homoclinic_asymptotics}. This section and the next one about the harmonic regime use formulation \eqref{e:ABC} and builds upon intermediate\footnote{As opposed to main results.} results from \cite{BMR2-I} and \cite{BMR2-II} on systems of Korteweg type. Indeed, in the solitary-wave limit, once relevant asymptotic expansions have been recalled, the proof shall be quite straightforward.

To ease notational translations, it is useful to recall that in Section~\ref{s:EK} we have derived for \eqref{e:ab} the
hydrodynamic formulation \eqref{e:absEK}
\be\label{e:absEK-bis}
\d_t\begin{pmatrix}\rho\\\bfv\end{pmatrix}=\cJ\,\delta H_0[(\rho,\bfv)]\,,
\ee
with $\bfv$ curl-free, where
\[
\cJ=\begin{pmatrix}
 0 & \Div\\
\n & 0
\end{pmatrix}
\]
and
\begin{equation*}
H_0[(\rho,\bfv)]=
\kappa(2\,\rho)\,\rho\,\|\bfv\|^2
+\frac{\kappa(2\,\rho)}{4\,\rho}\|\nabla_\bfx\rho\|^2\,.
\end{equation*}
In turn, the Hamiltonian problems studied in \cite{BMR2-I,BMR2-II} include systems\footnote{Restricted to the one-dimensional case.} in the form \eqref{e:absEK-bis} but with a larger class of Hamiltonian densities, given in original notation from \cite{BMR2-I,BMR2-II} as
\[
\cH[(\rho,\bfu))]=\frac{1}{2}\tau(v)\|\bfu\|^2+\frac{1}{2}\kappa(v)\|\nabla_\bfx v\|^2+f(v).
\]
Thus, when importing results from \cite{BMR2-I,BMR2-II}, we shall keep in mind the notational correspondence
\begin{align*}
(v,\,u)&\to (\rho,\,v)\,,&
\kappa(v)&\to \frac{1}{2\rho}\kappa(2\rho)\,,&
\tau(v)&\to 2\rho\,\kappa(2\rho)\,,&
f(v)&\to W(2\rho)\,.
\end{align*}

To derive expansions for $\sigma_1$, $\sigma_2$ and $\sigma_3$, it is convenient to use the profile equation \eqref{e:profile-rho} so as to write them as
\begin{align*}
\sigma_1&=
\int_{\rhomin(\mux)}^{\rhomax(\mux)} 
\frac{f_1(\rho)}{\sqrt{\mux-\cW_\rho(\rho))}}
\sqrt{\frac{2\,\kappa(2\,\rho)}{2\,\rho}}\,\dd \rho\,,\\
\sigma_2&=
\int_{\rhomin(\mux)}^{\rhomax(\mux)} 
\frac{f_2(\rho)}{\sqrt{\mux-\cW_\rho(\rho))}}
\sqrt{\frac{2\,\kappa(2\,\rho)}{2\,\rho}}\,\dd \rho\,,\\
\sigma_3&=
\int_{\rhomin(\mux)}^{\rhomax(\mux)} 
\frac{f_3(\rho)}{\sqrt{\mux-\cW_\rho(\rho))}}
\sqrt{\frac{2\,\kappa(2\,\rho)}{2\,\rho}}\,\dd \rho
+\,\int_{\rhomin(\mux)}^{\rhomax(\mux)} 
\sqrt{\mux-\cW_\rho(\rho))}\sqrt{\frac{2\,\kappa(2\,\rho)
}{2\,\rho}}\,\dd \rho\,,
\end{align*}
with
\begin{align*}
f_1(\rho)&:=\kappa(2\,\rho)\,2\,\rho\,,\\
f_2(\rho;\cx,\mup)&:=\kappa(2\,\rho)\,2\,\rho\,\nu(\rho;\cx,\mup)
\,,\\
f_3(\rho;\cx,\mup)&:=\kappa(2\,\rho)\,2\,\rho\,(\nu(\rho;\cx,
\mup))^2\,.
\end{align*}
In this form, \cite[Proposition~C.3]{BMR2-I} is directly applicable and yields required expansions.

\phantomsection\label{e:fj}

In the above and from now on, we mostly keep the dependence on  $(\cx,\omp,\mup)$ implicit for the sake of readability. This is consistent with the fact that the limit is reached by holding $(\cx,\omp,\mup)$ fixed and taking $\mux$ sufficiently close to $\mux^{(0)}(\cx,\omp,\mup)$ (uniformly for  $(\cx,\omp,\mup)$ in a compact neighborhood of $(\ucx^{(0)},\uomp^{(0)},\umup^{(0)})$).

The solitary-wave expansions naturally involve a mass conjugated to 
the minimal\footnote{Recall that we have decided to focus only on 
the case when the endstate of the mass of the limiting solitary 
wave is also its infimum.} mass of periodic waves. Explicitly, in 
the large-period regime there exists 
$\rhodual=\rhodual(\mux;\cx,\omp,\mup)$ such that
\[
\mux=\cW_\rho(\rhodual)\,,
\]
with $\rhodual<\rho^{(0)}$ and $\mux-\cW_\rho(\cdot)$ does not vanish on $(\rhodual,\rhomin)$. In other words, $\rhodual$ is the first cancellation point of $\mu_x-\cW_\rho
(\cdot)$ at the left of $\rho^{(0)}$; see Figure~\ref{phasesol}. 

The following theorem gathers the relevant pieces of asymptotic expansions. Up to a slight extension to incorporate expansions of $\sigma_1$, $\sigma_2$, $\sigma_3$, it is the translation in our setting of results from \cite[Theorem~3.16 \& Lemma~4.1]{BMR2-I}. We borrow the statement and notation\footnote{Except for a few variations. We use $(\Xx^{(s)},\eps)$ instead of $(\Xi_s,\rho)$ and $(\rhodual,\rhomin,\rhomax)$ instead of $(v_1,v_2,v_3)$. The subscript $0$ was originally $s$.} from \cite[Proposition~4 \& Theorem~2]{BMR2-II} where relevant results from \cite{BMR2-I} are compactly summarized. 

\begin{theorem}[\cite{BMR2-I}]\label{thm:BMR2-I_soliton}
In the large-period regime there exist real numbers $\falpha_s$, $\fbeta_s$, a positive number $\feta_s$, a vector $\bX_s$ and a symmetric matrix $\OO_s$ --- depending smoothly on the parameters $(\cx,\omp,\mup)$ --- such that, with\footnote{The parameter $\eps$ goes to zero as $\sqrt{\mux^{(0)}(\cx,\omp,\mup)-\mux}$.} 
\begin{align*}
\eps(\mux)
&:=\frac{\rhomin(\mux)-\rhodual(\mux)}{\rhomax(\mux)-\rhomin(\mux)}\,,&
\fgamma_s&:=
-\frac{1}{2\d_\rho^2\cW_\rho(\rho^{(0)})}\,,
\end{align*}
\begin{equation}\label{eq:asgradsol}
\dfrac{\pi}{\Xx^{(s)}}\,\nabla\Theta\,=\,-\,\bV_0\,\ln\eps\,-\,\bX_s \,+\,\frac{\eps}{2}\,\bV_0\,-\,\dfrac{1}{2\feta_s}\,(\falpha_s\,\bV_0\,+\,\fbeta_s\,\bW_0\,+\,\fgamma_s\,\bZ_0)\,\eps^2\ln\eps\,+\,{\mathcal O}\big(\eps^2\big)
\end{equation}
\begin{equation}\label{eq:ashesssol}
\frac{\pi}{\Xx^{(s)}}\,
\Hess\Theta
\begin{array}[t]{l}\displaystyle\,=\,
\feta_s\,\frac{1+\eps}{\eps^2}\;\bV_0\otimes \,\bV_0
\,+\,\left( \falpha_s\,\bV_0\otimes \,\bV_0\,+\,\fbeta_s\,\left(\bV_0\otimes \bW_0\,+\,\bW_0\otimes \bV_0\right)\right)\,\ln\eps\\ [15pt]
\displaystyle\,+\left(\bT_0\otimes \bT_0\,+\,2\fgamma_s\,\bW_0\otimes \bW_0
\,+\,\fgamma_s\,(\bZ_0\otimes \bV_0\,+\,\bV_0\otimes \bZ_0)\right) \ln\eps\\ [15pt]
\displaystyle\,+\;\OO_s\,+\, {\mathcal O}\big(\eps\ln\eps\big)
\end{array}
\end{equation}
and, for $j=1,2,3$,
\begin{align*}
\dfrac{\pi}{\Xx^{(s)}}\,\sigma_j\,=\,-\,f_j(\rho^{(0)})\,\ln\eps
\,+\,{\mathcal O}(1)\,,
\end{align*}
where 
\[
\Xx^{(s)}:=\sqrt{\frac{-\kappa(\rho^{(0)})}{\d_\rho^2\cW_\rho(\rho^{(0)})}}\,,
\] 
\begin{align*}\label{eq:vectstwo}
\bV_0&:=\left(\begin{array}{c} 1\\ \fq(\rho^{(0)}) \\ \rho^{(0)} \\ \nu(\rho^{(0)}) \end{array}\right)\,,&
\bW_0&:=\left(\begin{array}{c} 0\\ \d_\rho\fq(\rho^{(0)}) \\ 1 \\ \d_\rho \nu(\rho^{(0)})\end{array}\right)\,,\\
\bZ_0&:=\left(\begin{array}{c} 0\\\d_\rho^2\fq(\rho^{(0)}) \\ 0 \\ \d_\rho^2\nu(\rho^{(0)})\end{array}\right)\,,&
\bT_0&:=\,
\frac{1}{\sqrt{\kappa(2\,\rho^{(0)})\,2\,\rho^{(0)}}}
\left(\begin{array}{c}
0 \\
\rho^{(0)}\\
0 \\
1
\end{array}\right)\,,
\end{align*}
with $\fq$ defined by
\[
\fq(\rho):=\rho\,\nu(\rho)\,.
\]
Moreover the vectors are such that
\[
\begin{array}{l}
\bV_0 \cdot \bB_0^{-1}\bA_0\,\bV_0\,=\,0\,,\ \ 
\bV_0\,\cdot \bB_0^{-1}\bA_0\,\bW_0\,=\,0\,,\ \ 
\bV_0\,\cdot \bB_0^{-1}\bA_0\,\bT_0\,=\,0\,,\\ [10pt]
\bV_0\,\cdot \bB_0^{-1}\bA_0\,\bZ_0\,=\,-\,\bW_0\cdot \bB_0^{-1}\bA_0\,\bW_0\,,\qquad
\bT_0 \cdot \bB_0^{-1}\bA_0\,\bT_0\,=\,0\,,\qquad
\bT_0 \cdot \bB_0^{-1}\bA_0\,\bZ_0\,=\,0\,,\\ [10pt]
\beD_1\cdot \bV_0\,=\,1\,,\qquad
\beD_1\cdot \bW_0\,=\,0\,,\qquad
\beD_1\cdot \bZ_0\,=\,0\,,\qquad
\beD_1\cdot \bT_0\,=\,0\,,
\end{array}
\]
and
\begin{align*}
\frac{\Xx^{(s)}}{\pi}\,(\bB_0^{-1}\bA_0\bV_0) \cdot \bX_s &\,=\,
-\d_{\cx}\Theta_{(s)}(\cx^{(0)},\rho^{(0)},\kp^{(0)})\,,\\ 
\frac{\Xx^{(s)}}{\pi}\,(\bB_0^{-1}\bA_0\bV_0) \cdot \OO_s\,\bB_0^{-1}\bA_0\bV_0&\,=\,
\d_{\cx}^2\Theta_{(s)}(\cx^{(0)},\rho^{(0)},\kp^{(0)})\,.
\end{align*}
\end{theorem}

As a consequence, after a few straightforward but tedious computations omitted here but detailed in the discussion preceding\footnote{In notation of \cite{BMR2-II}, $\bP_0=\transp{\PP_s}\SS$, $-\bB_0^{-1}\bA_0=\SS^{-1}$, $\bfy_s=\transp{(\bD_s^{-1})}\bfy_s$, $\Sigma_0^{-1}=\Sigma_s\bD_s^{-1}$ and
\[
\bp\sigma_s&0\\\frac{w_s}{2}&1\ep
=\bA_s\,\bB^{-1}\,.
\]
} \cite[Theorem~8]{BMR2-II} stems that
\begin{align*}
&\dfrac{\pi}{\Xx^{(s)}}\,
\bp1&0&0\\
0&\frac{\sqrt{1+\eps}}{\eps}&0\\
0&0&\I_2\ep
\bP_0
(-\bB_0^{-1}\bA_0)\Hess\Theta\,
\bP_0^{-1}
\bp1&0&0\\
0&\frac{\eps}{\sqrt{1+\eps}}&0\\
0&0&\I_2\ep\\
&\,=\,\bp
-\fgamma_s\,w_0\,\ln(\eps)+\cO(1)&\frac{\sqrt{1+\eps}}{\eps}\,\feta_s+\cO(\eps\,\ln(\eps))&\cO(\ln(\eps))\\
\frac{\sqrt{1+\eps}}{\eps}\,\dfrac{\Xx^{(s)}}{\pi}\d_{\cx}^2\Theta_{(s)}
+\cO(\ln(\eps))&-\fgamma_s\,w_0\,\ln(\eps)+\cO(1)
&\frac{\sqrt{1+\eps}}{\eps}\,\transp{\bfy_s}+\cO(\ln(\eps))\\
\cO(1)&\cO(\eps\,\ln(\eps))&
\Sigma_0^{-1}\,\ln(\eps)+\cO(1)
\ep
\end{align*}
with $\bfy_s$ some two-dimensional vector and
\begin{align}\nonumber
\bP_0&:=\transp{\bp
-\beD_2&\bV_0&\bT_0&\bW_0
\ep}\\
\label{e:defSigma0}
\Sigma_0&:=\bp\sigma_0&0\\\frac{w_0}{2}&1\ep\,
(\bB\Hess H^{(0)}(\rho^{(0)},\kp^{(0)})+\cx\,\I_2)
\,\bp\sigma_0&0\\\frac{w_0}{2}&1\ep^{-1}
\end{align}
where 
\begin{align*}
\sigma_0&:=-\bT_0\cdot \bB_0^{-1}\bA_0\bW_0=\frac{1}{\sqrt{\kappa(2\,\rho^{(0)})\,2\,\rho^{(0)}}}\\
w_0&:=-\bW_0\cdot \bB_0^{-1}\bA_0\bW_0=2\d_\rho\nu(\rho^{(0)})\,.
\end{align*}
Moreover
\[
\Sigma_0^{-1}
\,=\,\bp
0&2\fgamma_s \sigma_0\\
\sigma_0&2\fgamma_s w_0
\ep\,.
\]
We recall that $H^{(0)}$ is the zero dispersion limit of the Hamiltonian $H_0$ of the hydrodynamic formulation of the Schr\"odinger equation and $\bB$ is the self-adjoint matrix involved in this formulation; see \eqref{e:zerodisp}. 

Observe also that from the foregoing follows
\begin{align*}
\delta_{(4,0,0)}=
\det(\Hess\Theta)
&\,=\,
-\frac{(\ln(\eps))^2}{\eps^2}\,\dfrac{(\Xx^{(s)})^5}{\pi^5}\frac{\feta_s
\,\d_{\cx}^2\Theta_{(s)}}{\det(\bB\Hess H^{(0)}+\cx\,\I_2)}
+\cO\left(\frac{\ln(\eps)}{\eps^2}\right)\\
&\,=\,
\frac{(\ln(\eps))^2}{\eps^2}\,\dfrac{(\Xx^{(s)})^5}{\pi^5}2\fgamma_s\,\sigma_0^2\,\feta_s
\,\d_{\cx}^2\Theta_{(s)}
+\cO\left(\frac{\ln(\eps)}{\eps^2}\right)\\
&\,=\,
-\frac{(\ln(\eps))^2}{\eps^2}\,\dfrac{(\Xx^{(s)})^5}{\pi^5}\frac{\feta_s
\,\d_{\cx}^2\Theta_{(s)}}{\kappa(2\,\rho^{(0)})\,2\,\rho^{(0)}\,\d_\rho^2\cW_\rho(\rho^{(0)})}
+\cO\left(\frac{\ln(\eps)}{\eps^2}\right)\,.
\end{align*}

Likewise
\begin{align*}
\delta_{(0,4,0)}&=1\,,&
\delta_{(1,3,0)}&=\cO(1)\,,
\end{align*}
\begin{align*}
\delta_{(2,2,0)}&=-\frac{1}{\eps^2}\,\dfrac{(\Xx^{(s)})^3}{\pi^3}\,\feta_s\,\d_{\cx}^2\Theta_{(s)}+\cO\left(\frac{\ln(\eps)}{\eps}\right)\,,
\end{align*}
\begin{align*}
\delta_{(3,1,0)}&=-\frac{\ln(\eps)}{\eps^2}\,\dfrac{(\Xx^{(s)})^4}{\pi^4}\,2\fgamma_s\,w_0\,\feta_s\,\d_{\cx}^2\Theta_{(s)}\,
+\cO\left(\frac{1}{\eps^2}\right)\,.
\end{align*}

To go on we need to compute a similar expansion for
\[
\bp1&0&0\\
0&\frac{\sqrt{1+\eps}}{\eps}&0\\
0&0&\I_2\ep
\bP_0
(-\bB_0^{-1}\bC_0)\,
\bP_0^{-1}
\bp1&0&0\\
0&\frac{\eps}{\sqrt{1+\eps}}&0\\
0&0&\I_2\ep\,.
\]
A direct computation yields
\begin{align*}
-\bP_0\bB_0^{-1}\bC_0&=
\bp
0&0&0&0\\
0&\sigma_3+\nu(\rho^{(0)})\sigma_2&-(\sigma_2+\nu(\rho^{(0)})\sigma_1)&0\\
0&\sigma_0\sigma_2&-\sigma_0\sigma_1&0\\
0&\frac{w_0}{2}\sigma_2&-\frac{w_0}{2}\sigma_1&0
\ep
\end{align*}
and
\begin{align*}
\bP_0^{-1}
&=\bp
*&*&*&*\\
-1&0&0&0\\
\nu(\rho^{(0)})&0 &-\frac{w_0}{2\sigma_0} & 1\\
*&*&*&*
\ep
\end{align*}
so that
\begin{align*}
&-\bP_0\bB_0^{-1}\bC_0\bP_0^{-1}=
\bp
0&0&0&0\\
-(\sigma_3+2\nu(\rho^{(0)})\sigma_2+\nu(\rho^{(0)})^2\sigma_1)&0
&\frac{w_0}{2\sigma_0}(\sigma_2+\nu(\rho^{(0)})\sigma_1)
&-(\sigma_2+\nu(\rho^{(0)})\sigma_1)
\\
-\sigma_0(\sigma_2+\nu(\rho^{(0)})\sigma_1)&0
&\frac{w_0}{2}\sigma_1
&-\sigma_0\sigma_1
\\
-\frac{w_0}{2}(\sigma_2+\nu(\rho^{(0)})\sigma_1)&0
&
\frac{w_0^2}{4\sigma_0}\sigma_1
&-\frac{w_0}{2}\sigma_1
\ep\,.
\end{align*}
To ease computations and materialize both symmetry and size we introduce
\begin{align*}
\delta_1^{(s)}&:=\frac{\sigma_1}{-\ln(\eps)\Xx^{(s)}/\pi}\,,&
\delta_2^{(s)}&:=\frac{\sigma_2+\nu(\rho^{(0)})\sigma_1}{-\ln(\eps)\Xx^{(s)}/\pi}\,,&
\delta_3^{(s)}&:=\frac{\sigma_3+2\nu(\rho^{(0)})\sigma_2+\nu(\rho^{(0)})^2\sigma_1}{-\ln(\eps)\Xx^{(s)}/\pi}\,.
\end{align*}
Thus
\begin{align*}
-\bp1&0&0\\
0&\frac{\sqrt{1+\eps}}{\eps}&0\\
0&0&\I_2\ep
&\bP_0\bB_0^{-1}\bC_0\bP_0^{-1}
\bp1&0&0\\
0&\frac{\eps}{\sqrt{1+\eps}}&0\\
0&0&\I_2\ep\\
&=
\frac{\Xx^{(s)}}{\pi}
\bp
0&0&0&0\\
\delta_3^{(s)}\,\frac{\sqrt{1+\eps}}{\eps}\ln(\eps)&0
&-\frac{w_0}{2\sigma_0}\delta_2^{(s)}\,\frac{\sqrt{1+\eps}}{\eps}\ln(\eps)
&\delta_2^{(s)}\,\frac{\sqrt{1+\eps}}{\eps}\ln(\eps)
\\
\sigma_0\delta_2^{(s)}\,\ln(\eps)&0
&-\frac{w_0}{2}\delta_1^{(s)}\,\ln(\eps)
&\sigma_0\delta_1^{(s)}\,\ln(\eps)
\\
\frac{w_0}{2}\delta_2^{(s)}\,\ln(\eps)&0
&
-\frac{w_0^2}{4\sigma_0}\delta_1^{(s)}\,\ln(\eps)
&\frac{w_0}{2}\delta_1^{(s)}\,\ln(\eps)
\ep
\end{align*}
with
\begin{align*}
\delta_1^{(s)}&=\,f_1(\rho^{(0)})+\cO\left(\frac{1}{\ln(\eps)}\right)\,,\\
\delta_2^{(s)}&=\,2\,\nu(\rho^{(0)})\,f_1(\rho^{(0)})+\cO\left(\frac{1}{\ln(\eps)}\right)\,,\\
\delta_3^{(s)}&=\,4\,\nu(\rho^{(0)})^2\,f_1(\rho^{(0)})+\cO\left(\frac{1}{\ln(\eps)}\right)\,.
\end{align*}
At main order the matrix has rank one, with this observation
we find
\[
\delta_{(2,0,2)}
\,=\,\cO\left(\frac{(\ln(\eps))^2}{\eps^2}\right)\,,
\]
\[
\delta_{(3,0,1)}
\,=\,
\frac{(\ln(\eps))^3}{\eps^2}\,\dfrac{(\Xx^{(s)})^4}{\pi^4}
2\,\fgamma_s\,\sigma_0^2\,\feta_s\delta_3^{(s)}
+\cO\left(\frac{(\ln(\eps))^2}{\eps^2}\right)\,.
\]
Note that this already yields the instability condition
in the large period regime:
\begin{align*}
\d_{\cx}^2\Theta_{(s)}&>0\,,
\end{align*}
This may be derived, for instance, by examining the limit $\eps\to 0$ of the rescaled
\begin{eqnarray*}
\Delta_0\left(\frac{\sqrt{\eps}\,\lambda}{\sqrt{|\ln(\eps)|}},0,\frac{\sqrt{\eps}\,\zeta}{\ln(\eps)}\right)
&=&\left(
-\dfrac{(\Xx^{(s)})^5}{\pi^5}\frac{\feta_s
\,\d_{\cx}^2\Theta_{(s)}}{\kappa^{(0)}\,2\,\rho^{(0)}\,\d_\rho^2\cW_\rho^{(0)}}
+\cO\left(\frac{1}{\ln \eps}\right)\right)\lambda^4
+\cO\left(\frac{\zeta^4}{\ln^2\eps}\right)\\
&&+\left(\dfrac{(\Xx^{(s)})^4}{\pi^4}
2\,\fgamma_s\,\sigma_0^2\,\feta_s4(\nu^{(0)})^2
2\rho^{(0)}\kappa^{(0)}
+\cO\left(\frac{1}{\ln \eps}\right)\right)\lambda^2
\zeta^2
\,.
\end{eqnarray*}

According to both Theorem~\ref{th:coperiodic_asymptotics} and Theorem~\ref{th:side-band_asymptotics}, $\d_{\cx}^2\Theta_{(s)}<0$ also yields instability. Applying Theorem~\ref{th:coperiodic_asymptotics} shows that the instability may be obtained with $\xi=0$, whereas applying Theorem~\ref{th:side-band_asymptotics} shows that it also corresponds to a failure of weak hyperbolicity of the modulated system. 

This achieves the proof of Theorem~\ref{th:homoclinic_asymptotics}.

\subsection{Small-amplitude regime}\label{s:harmonic}

We now turn to the small-amplitude limit. Our goal is to prove the following theorem.

\bt\label{th:harmonic_asymptotics}
In the small amplitude regime near a $(\ucx^{(0)},\urho^{(0)},\ukp^{(0)})$ such that 
\[
\d_\rho\nu(\urho^{(0)};\ucx^{(0)},\umup^{(0)})\neq0\,,
\]
and\footnote{Where $\delta_{hyp}$, $\delta_{BF}$ are as in \eqref{def:sideindex}-\eqref{def:sideindex_II}.}
\[
\delta_{hyp}\times\delta_{BF}(\ucx^{(0)},\uomp^{(0)},\umup^{(0)})\neq0,
\]
waves are spectrally exponentially unstable to transversally-slow longitudinally-side-band perturbations.
\et

Let us stress that, as proved in Appendix~\ref{s:constant}, the limiting constant state is spectrally stable if and only if $\delta_{hyp}<0$ so that, when $\delta_{hyp}>0$, the result is nontrivial from the point of view of spectral perturbation. The overall proof strategy is the same than in the large-period regime but the final argument is significantly more cumbersome. In particular it involves essentially all coefficients of $\Delta_0$.

To begin with, we gather relevant expansions. The following is a straightforward extension of \cite[Theorem~3.14 \& Lemma~4.1]{BMR2-I} with notation taken from \cite[Theorem~2 \& Proposition~4]{BMR2-II}.

\begin{theorem}[\cite{BMR2-I}]\label{thm:BMR2-I_harmonic}
In the small-amplitude regime there exist a real number $\fbeta_0$ and a positive number $\fgamma_0$ --- depending smoothly on the parameters $(\cx,\omp,\mup)$ --- such that, with $\falpha_0$ given by \eqref{eq:falpha} and\footnote{The parameter $\eps$ goes to zero as $\sqrt{\mux-\mux^{(0)}(\cx,\omp,\mup)}$.} 
\begin{align*}
\eps(\mux)
&:=\frac{\rhomax(\mux)-\rhomin(\mux)}{2}
\frac{\rhomin(\mux)-\rhodual(\mux)}{\rhomax(\mux)-\rhomin(\mux)}\,,&
\fgamma_0&:=\frac{1}{2\d_\rho^2\cW_\rho(\rho^{(0)})}\,,
\end{align*}
we have
\begin{equation}\label{eq:asgradharm}
\frac{4\fgamma_0}{\Xx^{(0)}}\,\nabla\Theta\,=\,4\fgamma_0\,\bV_0\,+\,(\falpha_0\,\bV_0\,+\,\fbeta_0\,\bW_0\,+\,\fgamma_0\,\bZ_0)\,\eps^2
\,+\,{\mathcal O}(\eps^4)\,,
\end{equation}
\begin{equation}\label{eq:ashessharm}
\frac{1}{\Xx^{(0)}}\,\Hess\Theta\,=\,\begin{array}[t]{l}\falpha_0\,\bV_0\otimes \bV_0\,+\,\fbeta_0\,(\bV_0\otimes\bW_0\,+\,\bW_0\otimes\bV_0)\,-\,\bT_0\otimes\bT_0\\ [10pt]
\,+\,2\,\fgamma_0\,\bW_0\otimes\bW_0\,+\,
\fgamma_0\,(\bV_0\otimes\bZ_0\,+\,\bZ_0\otimes\bV_0)
\,+\,{\mathcal O}(\eps^2)\,,
\\ [10pt]
\end{array}
\end{equation}
and, for $j=1,2,3$,
\begin{align*}
\frac{4\fgamma_0}{\Xx^{(0)}}\,\sigma_j\,=\,
4\fgamma_0\,f_j(\rho^{(0)})
&\,+\,{\mathcal O}(\eps^2)\,,
\end{align*}
where $\Xx^{(0)}$ denotes the harmonic period \eqref{e:harm_period} and the other quantities are as in Theorem~\ref{thm:BMR2-I_soliton}.
\end{theorem}

Our starting point is 
\begin{align*}
&\dfrac{1}{\Xx^{(0)}}\,
\bp1&0&0\\
0&\frac{1}{\eps}&0\\
0&0&\I_2\ep
\widetilde{\bP}_0^{-1}
\bP_0
(-\bB_0^{-1}\bA_0)\Hess\Theta\,
\bP_0^{-1}\widetilde{\bP}_0
\bp1&0&0\\
0&\eps&0\\
0&0&\I_2\ep\\
&\,=\,\bp
\frac{1}{\lambda_\eps}&\fdelta_0\,\eps+\cO(\eps^3)&\cO(\eps^2)\\
\fepsilon_0\,\eps
+\cO(\eps^2)&\frac{1}{\lambda_\eps}
&\cO(\eps)\\
\cO(\eps^2)&\cO(\eps^3)&
\Sigma_\eps^{-1}
\ep
\end{align*}
with 
\begin{align*}
\lambda_\eps&=-\frac{1}{\fgamma_0\,w_0}\,+\cO(\eps^2)\,,&
\Sigma_\eps&=\Sigma_0+\cO(\eps^2)\,,
\end{align*}
$\fdelta_0$ and $\fepsilon_0$ having the sign\footnote{In notation of \cite{BMR2-II}, with $(b,g)\to(1,\nu)$,
\begin{align*}
\delta_{BF}&=\frac{1}{16}\,\frac{4\tau(\d_v g)^5}{b^3k_0}
\frac{\d_v^2\cW+3\tau\,(\d_v g)^2}{-\d_v^2\cW+3\tau\,(\d_v g)^2}\times\Delta_{MI}\,,&
\delta_{hyp}&=\frac{1}{4}\d_v^2\Ham
\,=\,\frac{1}{4}(\tau\,(\d_v g)^2-\d_v^2\cW)\,,\\
2\rho^{(0)}\,\kappa(2\rho^{(0)})\,w_0^2+8\delta_{hyp}
&=2(-\d_v^2\cW+3\tau\,(\d_v g)^2)\,,&
w_0&=w_0=\frac{2\d_vg}{b}\,,\\
\fdelta_0&=
\frac{\tau(\d_v g)^5}{b^3\,k_0\,(\d_v^2\cW)^3}(-\d_v^2\cW+3\tau\,(\d_v g)^2)\,\Delta_{MI}\,,&
\fdelta_0\,\fepsilon_0&=
\frac{\fgamma_0^3\,w_0^5}{4\,k_0}\,\Delta_{MI}\,.
\end{align*}} respectively of $\delta_{BF}$ and of $2\rho^{(0)}\,\kappa(2\rho^{(0)})\,w_0^2+8\delta_{hyp}$,
\begin{align*}
\widetilde{\bP}_0&=
\begin{pmatrix}
1&0&\transp{\ell}\\
0&1&0\\
0&r&\I_2
\end{pmatrix}\,,&
\widetilde{\bP}_0^{-1}&=
\begin{pmatrix}
1&\transp{\ell}r&-\transp{\ell}\\
0&1&0\\
0&-r&\I_2
\end{pmatrix}\,,
\end{align*}
where
\begin{align*}
r&=-(\Sigma_0^{-1}+\fgamma_0\,w_0)^{-1}
\begin{pmatrix}\fbeta_0\,\sigma_0\\
\fbeta_0\,w_0+\fgamma_0\,\zeta_0\end{pmatrix}\,,&
\ell&=-\begin{pmatrix}0&\sigma_0\\
\sigma_0&w_0\end{pmatrix}^{-1}\,r\,.
\end{align*}
Note that 
\begin{align*}
\Sigma_0^{-1}
&\,=\,\bp
0&-2\fgamma_0 \sigma_0\\
\sigma_0&-2\fgamma_0 w_0
\ep\,,&
\delta_{hyp}&=\frac14\left(\frac{w_0^2}{4\sigma_0^2}-\frac{1}{2\fgamma_0}\right)\,,\\
\det(\Sigma_0^{-1}+\fgamma_0\,w_0)&=
-16\,\fgamma_0^2\,\sigma_0^2\,\delta_{hyp}\,,&
(\Tr(\Sigma_0^{-1}))^2-4\,\det(\Sigma_0^{-1})
&=\,64\,\fgamma_0^2\,\sigma_0^2\,\delta_{hyp}\,.
\end{align*}

This yields $\delta_{(0,4,0)}=1$,
\begin{align*}
\delta_{(4,0,0)}&=
\det(\Hess\Theta)
\,=\,(\Xx^{(0)})^4\,\left(\frac{1}{\lambda_\eps^2}-\eps^2\,\fdelta_0\,\fepsilon_0\right)
\,\det(\Sigma_\eps^{-1})
+\cO\left(\eps^4\right)\,,\\
\delta_{(3,1,0)}&=(\Xx^{(0)})^3\,\left(
\frac{2}{\lambda_\eps}\,\det(\Sigma_\eps^{-1})
+\left(\frac{1}{\lambda_\eps^2}-\eps^2\,\fdelta_0\,\fepsilon_0\right)
\,\Tr(\Sigma_\eps^{-1})\right)
+\cO\left(\eps^4\right)\,,\\
\delta_{(2,2,0)}&=(\Xx^{(0)})^2\,\left(
\det(\Sigma_\eps^{-1})
+\left(\frac{1}{\lambda_\eps^2}-\eps^2\,\fdelta_0\,\fepsilon_0\right)
+\frac{2}{\lambda_\eps}\,\,\Tr(\Sigma_\eps^{-1})\right)
+\cO\left(\eps^4\right)\,,\\
\delta_{(1,3,0)}&=\Xx^{(0)}\,\left(
\frac{2}{\lambda_\eps}\,+\,\Tr(\Sigma_\eps^{-1})\right)\,,
\end{align*}
which is equivalent to
\begin{align*}
\Delta_0\left(\frac{\lambda}{\Xx^{(0)}},z,0\right)
&=\,\left(\left(\frac{\lambda}{\lambda_\eps}+z\right)^2-\lambda^2\eps^2\fdelta_0\,\epsilon_0\right)\,
\det(\lambda\Sigma_\eps^{-1}+z\,\I_2)\,+\,
\cO(\epsilon^4\,\lambda^2\,(|\lambda|^2+|z|^2))\,,\\
\Delta_0\left(\frac{\lambda}{\Xx^{(0)}},z-\frac{\lambda}{\lambda_\eps},0\right)
&=\,\left(z^2-\lambda^2\eps^2\fdelta_0\,\epsilon_0\right)\,
\det\left(\lambda\,\left(\Sigma_\eps^{-1}-\frac{1}{\lambda_\eps}\I_2\right)+z\,\I_2\right)\,+\,
\cO(\epsilon^4\,\lambda^2\,(|\lambda|^2+|z|^2))\,,
\end{align*}
One recovers the instability criteria on $\delta_{BF}$ and $\delta_{hyp}$ in the form that both from $\fdelta_0\,\epsilon_0<0$ and from $\left(\Tr(\Sigma_0^{-1})\right)^2-4\det(\Sigma_0^{-1})<0$ stems instability.

This motivates a first shift to
\[
\tilde{\Delta}_0(\lambda,z,\zeta):=
\Delta_0\left(\frac{\lambda}{\Xx^{(0)}},z-\frac{\lambda}{\lambda_\eps},\frac{\zeta}{\Xx^{(0)}}\right)\,.
\]
Note that we still have an expansion in the form
\[
\tilde{\Delta}_0(\lambda,z,\zeta)
\,=\,\sum_{\substack{0\leq m,n,p \leq 4\\m+n+p=4\\p\leq m}}
\tilde{\delta}_{(m,n,p)}\lambda^{m-p}\,z^n\,\zeta^{2p}\,,
\]
but that expressing instability criteria in terms of $\tilde{\Delta}_0$ is not obvious and will require some care. We already  know that $\tilde{\delta}_{(0,4,0)}=1$, then
\begin{align*}
\tilde{\delta}_{(4,0,0)}&=
-\eps^2\,\fdelta_0\,\fepsilon_0
\,\det\left(\Sigma_\eps^{-1}-\frac{1}{\lambda_\eps}\I_2\right)
+\cO\left(\eps^4\right)\,,\\
\tilde{\delta}_{(3,1,0)}&=-\eps^2\,\fdelta_0\,\fepsilon_0
\,\Tr\left(\Sigma_\eps^{-1}-\frac{1}{\lambda_\eps}\I_2\right)
+\cO\left(\eps^4\right)\,,\\
\tilde{\delta}_{(2,2,0)}&=
\det\left(\Sigma_\eps^{-1}-\frac{1}{\lambda_\eps}\I_2\right)
-\eps^2\,\fdelta_0\,\fepsilon_0
+\cO\left(\eps^4\right)\,,\\
\tilde{\delta}_{(1,3,0)}&=\Tr\left(\Sigma_\eps^{-1}-\frac{1}{\lambda_\eps}\I_2\right)=\cO(\eps^2)\,.
\end{align*}

Now, as in the solitary wave limit, we introduce 
\begin{align*}
\delta_1^{(0)}&:=\frac{\sigma_1}{\Xx^{(0)}}\,,&
\delta_2^{(0)}&:=\frac{\sigma_2+\nu(\rho^{(0)})\sigma_1}{\Xx^{(0)}}\,,&
\delta_3^{(0)}&:=\frac{\sigma_3+2\nu(\rho^{(0)})\sigma_2+\nu(\rho^{(0)})^2\sigma_1}{\Xx^{(0)}}\,,
\end{align*}
so that
\begin{align*}
-\bP_0\bB_0^{-1}\bC_0\bP_0^{-1}
&=
-\Xx^{(0)}
\bp
0&0&0&0\\[0.25em]
\delta_3^{(0)}&0
&-\frac{w_0}{2\sigma_0}\delta_2^{(0)}
&\delta_2^{(0)}
\\[0.25em]
\sigma_0\delta_2^{(0)}&0
&-\frac{w_0}{2}\delta_1^{(0)}
&\sigma_0\delta_1^{(0)}
\\[0.25em]
\frac{w_0}{2}\delta_2^{(0)}&0
&-\frac{w_0^2}{4\sigma_0}\delta_1^{(0)}
&\frac{w_0}{2}\delta_1^{(0)}
\ep
\end{align*}
with
\begin{align*}
\delta_1^{(0)}&=\,f_1(\rho^{(0)})+\cO\left(\eps^2\right)\,,\\
\delta_2^{(0)}&=\,2\,\nu(\rho^{(0)})\,f_1(\rho^{(0)})+\cO\left(\eps^2\right)\,,\\
\delta_3^{(0)}&=\,4\,\nu(\rho^{(0)})^2\,f_1(\rho^{(0)})+\cO\left(\eps^2\right)\,.
\end{align*}
Then we compute that
\begin{align*}
&-\dfrac{1}{\Xx^{(0)}}\,
\bp1&0&0\\
0&\frac{1}{\eps}&0\\
0&0&\I_2\ep
\widetilde{\bP}_0^{-1}
\bP_0
\bB_0^{-1}\bC_0\,
\bP_0^{-1}\widetilde{\bP}_0
\bp1&0&0\\
0&\eps&0\\
0&0&\I_2\ep\\
&=
\bp\cO(1)&\cO(\eps)&\cO(1)&\cO(1)\\
\frac{\delta_3^{(0)}}{\eps}&\cO(1)&\cO\left(\eps^{-1}\right)&\cO\left(\eps^{-1}\right)\\
\cO(1)&\cO(\eps)&\cO(1)&\cO(1)\\
\cO(1)&\cO(\eps)&\cO(1)&\cO(1)
\ep
\end{align*}
and observe that the latter matrix takes the form 
\[
\textrm{matrix of rank }1\quad\times(\,\I_4+\cO(\eps^2)\,)\,.
\] 
As a result
\begin{align*}
\tilde{\delta}_{(2,0,2)}&=\cO(\eps^2)\,,&
\tilde{\delta}_{(3,0,1)}&=
-\delta_3^{(0)}\,\fdelta_0\,\det\left(\Sigma_\eps^{-1}-\frac{1}{\lambda_\eps}\I_2\right)
+\,\cO(\eps^2)\,,\\
\tilde{\delta}_{(2,1,1)}&=\cO(1)\,,&
\tilde{\delta}_{(1,2,1)}&=\cO(1)\,.
\end{align*}
Taking the limit $\eps\to0$ in
\[
\Delta_0\left(\frac{\lambda_\eps}{\Xx^{(0)}}\left(-\eps^{-\frac14} Z+\eps^{\frac14}\Lambda\right),\eps^{-\frac14} Z,\eps^{\frac14}\frac{\Gamma}{\Xx^{(0)}}\right)
\]
yields the limiting
\[
\Lambda^2\,Z^2\,\det\left(\Sigma_0^{-1}-\frac{1}{\lambda_0}\I_2\right)
-\delta_3^{(0)}\,\fdelta_0\,\det\left(\Sigma_0^{-1}-\frac{1}{\lambda_0}\I_2\right)\lambda_0^2\,Z^2\,\Gamma^2\,
\,=\,0\,,
\]
where we recall $\det\left(\Sigma_0^{-1}-\frac{1}{\lambda_0}\I_2\right)
=-16\fgamma_0^2\sigma_0^2\delta_{hyp}$. From this, one deduces that when $\delta_{hyp}\neq0$, $\delta_{BF}>0$ gives instability since $\fdelta_0$ has the sign of $\delta_{BF}$. 

Since Theorem~\ref{th:side-band_asymptotics} already concludes instability from $\delta_{BF}>0$, this concludes the proof of Theorem~\ref{th:harmonic_asymptotics}.


\appendix

\section{Symmetries and conservation laws}\label{s:Noether}

In the present paper, including the current section we only consider functional densities depending of derivatives up to order $1$. In particular
\begin{align*}
L \cA[U](\bJ\delta \cB [U])
\,=\,
-L \cB[U](\bJ\delta \cA [U])
+\sum_{\ell}
\d_{\ell}\left(
\nabla_{U_{x_\ell}}\cA[U]\cdot\bJ\delta \cB[U]
+\nabla_{U_{x_\ell}}\cB[U]\cdot\bJ\delta \cA[U]
\right)\,.
\end{align*}
As a consequence, if $U_t=\bJ\delta\cH[U]$ then
\begin{equation}\label{e:abstract-cl}
(\cG[U])_t\,=\,
-L \cH[U](\bJ\delta \cG [U])
+\sum_{\ell}
\d_{\ell}\left(
\nabla_{U_{x_\ell}}\cG[U]\cdot\bJ\delta \cH[U]
+\nabla_{U_{x_\ell}}\cH[U]\cdot\bJ\delta \cG[U]\right)
\,.
\end{equation}

The main point in concrete uses of the latter equality is that $U\mapsto L \cH[U](\bJ\delta \cG [U])$ encodes the variation of $\cH$ under the action of the group generated by $\cG$. Here we consider two kinds of invariance by the action of a group generated by a functional density:
\begin{itemize}
\item stationarity of the density functional $\cH$ under the action of the group generated by $\cG$ encoded by
\[
L \cH[U](\bJ\delta \cG [U])\equiv 0
\]
in which case \eqref{e:abstract-cl} reduces to
\[
(\cG[U])_t\,=\,
\sum_{\ell}
\d_{\ell}\left(
\nabla_{U_{x_\ell}}\cG[U]\cdot\bJ\delta \cH[U]
+\nabla_{U_{x_\ell}}\cH[U]\cdot\bJ\delta \cG[U]\right)
\,;
\]
\item commutation of the density functional $\cH$ with the action of the group generated by $\cG$ encoded by
\[
L \cH[U](\bJ\delta \cG [U])\,=\, \bJ\delta \cG [\cH[U]]
\]
in which case \eqref{e:abstract-cl} reduces to
\[
(\cG[U])_t\,=\,
-\bJ\delta \cG [\cH[U]]
+\sum_{\ell}
\d_{\ell}\left(
\nabla_{U_{x_\ell}}\cG[U]\cdot\bJ\delta \cH[U]
+\nabla_{U_{x_\ell}}\cH[U]\cdot\bJ\delta \cG[U]\right)
\,.
\]
\end{itemize}
Note that in the latter case if the group is a group of translations then the latter equation is still a conservation law.

Specializing the first case to $\cG=\Mass$ gives
\[
(\Mass[U])_t\,=\,
\sum_{\ell}
\d_{\ell}\left(
\nabla_{U_{x_\ell}}\cH[U]\cdot\bJ U\right)
\]
whereas a specialization of the second case respectively to $\cG=\Impulse_j$ and to $\cG=\cH$ gives respectively
\[
(\Impulse_j[U])_t\,=\,
\d_j(\nabla_{U_{x_j}}\Impulse_j[U]\cdot\bJ\delta \cH[U]-\cH[U])
+\sum_{\ell}
\d_{\ell}\left(
\nabla_{U_{x_\ell}}\cH[U]\cdot U_{x_j}\right)
\]
and 
\[
(\cH[U])_t\,=\,
\sum_{\ell}
\d_{\ell}\left(
\nabla_{U_{x_\ell}}\cH[U]\cdot\bJ\delta \cH[U]\right)
\,.
\]
To compute how these conservation laws are transformed when going to uniformly moving frames, we also record the following simple but useful relations 
\begin{align*}
L \Mass[U](\bJ\delta\Mass[U])&=0\,,&
L \Mass[U](\bJ\delta\Impulse_\ell[U])&=\d_\ell(\Mass[U])\,,\\
L \Impulse_j[U](\bJ\delta \Mass[U])&=-\d_j(\Mass[U])\,,&
L \Impulse_j[U](\bJ\delta \Impulse_\ell[U])&=\d_\ell(\Impulse_j[U])\,,\\
L \cH[U](\bJ\delta\Mass[U])&=0\,,&
L \cH[U](\bJ\delta\Impulse_\ell[U])&=\d_\ell(\cH[U])\,.
\end{align*}
Among the foregoing identities only the third one is not a simple expression of invariances of $\Mass$, $\Impulse_j$ and $\cH$ but it may be deduced from the second one.

For our purposes, it is also crucial to derive linearized versions of the algebraic relations expounded above. Let us define $\cF_{\cG,\cH}$ by $\cF_{\cG,\cH}[U]=L \cG[U](\bJ\delta\cH[U])$ so that $U_t=\bJ\delta\cH[U]$ implies $(\cG(U))_t=\cF_{\cG,\cH}[U]$. Now note that if $\uU$ is such that $\delta\cH[\uU]=0$ then $L\cF_{\cG,\cH}[\uU](V)=L \cG[\uU](\bJ L\delta\cH[\uU](V))$. In particular if $\uU_t=0$ and $\delta\cH[\uU]=0$ then $V_t=\bJ L\delta\cH[\uU](V)$ implies 
\[
(L\cG[\uU]V)_t\,=\,L\cF_{\cG,\cH}[\uU](V)\,.
\]
The latter computations also provide similar conclusions for the associated spectral problems.

\section{Spectral stability of constant states}\label{s:constant}

In the present section we study the spectral stability of constant solutions to \eqref{e:ab}. By constant solutions we mean solutions that are constant up to the symmetries, thus solutions in the form\footnote{The action of spatial translations is redundant with the action of rotations for this class of solutions.}
\be\label{def:constant}
\bU(t,\bfx)=\eD^{(\bkp\cdot\bfx+\omp\,t)\bJ}\bU^{(0)}\,,
\ee
with $\bU^{(0)}$ a constant vector of $\R^2$, $\bkp\in\R^d$ and $\omp\in\R$. Since it is almost costless and will be useful in Appendix~\ref{s:more-equations}, we consider a yet more general class of Hamiltonian equations
\begin{align}\label{e:ab-more}
\d_t\bU&=\bJ\,\delta \Ham_0[\bU]\,,&
\qquad\text{with}\qquad&&
\Ham_0\left[\bU\right]&=\tfrac12 \nabla_\bfx \bU\cdot\bD(\|\bU\|^2)\nabla_\bfx \bU
+W(\|\bU\|^2)\,,
\end{align}
where $\bD$ is valued in real symmetric $d\times d$-matrices. That $\bU$ from \eqref{def:constant} solves \eqref{e:ab-more} reduces to either $\bU^{(0)}$ is zero or
\be\label{e:constant}
\omp\,=\,2\,W'(\|\bU^{(0)}\|^2)
+\bkp\cdot\bD(\|\bU^{(0)}\|^2)\,\bkp
+\|\bU^{(0)}\|^2\,\bkp\cdot\bD'(\|\bU^{(0)}\|^2)\,\bkp\,.
\ee

For the sake of concision and comparison, it is expedient to introduce the dispersionless hydrodynamic Hamiltonian
\[
H^{(0)}(\rho,\bfv):=\rho\,\bfv\cdot\bD(2\,\rho)\bfv
+W(2\,\rho)\,.
\]
As a first instance, note that with $\rho^{(0)}:=\Mass(\bU^{(0)})$, \eqref{e:constant} takes the concise form $\omp=\d_\rho H^{(0)}(\rho^{(0)},\bkp)$.

Changing frame through $\bU(t,\bfx)=\eD^{(\bkp\cdot\bfx+\omp\,t)\bJ}\bV(t,\bfx)$, linearizing and using the Fourier transform bring the spectral stability question under consideration to the question of knowing whether for any $\bfxi\in\R^d$, the linear operator on $\C^2$
\begin{align*}
\bV\mapsto&
\left[\,\bU^{(0)}\cdot\bV\times\d_\rho^2H^{(0)}(\rho^{(0)},\bkp)
\,+\,2\,\bJ\bU^{(0)}\cdot\bV\times\bkp\cdot \bD'(2\,\rho^{(0)})\iD\bfxi\,\right]\,\bJ\bU^{(0)}\\
&+\left[2\,\bU^{(0)}\cdot\bV\times\bkp\cdot \bD'(2\,\rho^{(0)})\iD\bfxi\right]\,\bU^{(0)}
+\left[\bfxi\cdot \bD(2\,\rho^{(0)})\bfxi\right]\,\bJ\bV
+\left[2\,\bkp\cdot \bD(2\,\rho^{(0)})\iD\bfxi\right]\,\bV
\end{align*}
has purely imaginary spectrum. If $\bU^{(0)}=0$ then, by diagonalizing $\bJ$, one gets that the latter spectrum is 
\[
\iD\,\left(\pm\bfxi\cdot \bD(2\,\rho^{(0)})\bfxi
+2\,\bkp\cdot \bD(2\,\rho^{(0)})\bfxi\right)\in\iD\,\R
\]
hence spectral stability holds. When $\bU^{(0)}\neq0$, we may use $\bV\mapsto (\bU^{(0)}\cdot\bV,\bJ\bU^{(0)}\cdot\bV)$ as a coordinate map in which the above operator's matrix is
\begin{align*}
\bp 
2\,\d_\rho\nabla_\bfv H^{(0)}(\rho^{(0)},\bkp)\cdot\iD\bfxi
&-\bfxi\cdot \bD(2\,\rho^{(0)})\bfxi\\
2\,\rho^{(0)}\,\d_\rho^2H^{(0)}(\rho^{(0)},\bkp)+\bfxi\cdot \bD(2\,\rho^{(0)})\bfxi
&
2\,\d_\rho\nabla_\bfv H^{(0)}(\rho^{(0)},\bkp)\cdot\iD\bfxi
\ep\,.
\end{align*}
Thus, when $\bU^{(0)}\neq0$, spectral stability holds if and only if for any $\bfxi\in\R^d$, the solutions in $\lambda$ of
\[
\left(\lambda-2\iD\,\d_\rho\nabla_\bfv H^{(0)}(\rho^{(0)},\bkp)\cdot\bfxi\right)^2
+\bfxi\cdot \bD(2\,\rho^{(0)})\bfxi\,\left(2\,\rho^{(0)}\,\d_\rho^2H^{(0)}(\rho^{(0)},\bkp)+\bfxi\cdot \bD(2\,\rho^{(0)})\bfxi\right)\,=\,0
\]
are purely imaginary, that is, if and only if, for any $\bfxi\in\R^d$,
\[
\bfxi\cdot \bD(2\,\rho^{(0)})\bfxi\,\left(2\,\rho^{(0)}\,\d_\rho^2H^{(0)}(\rho^{(0)},\bkp)+\bfxi\cdot \bD(2\,\rho^{(0)})\bfxi\right)
\,\geq\,0\,.
\]

As a conclusion, spectral stability holds if and only if, for any unitary $\beD\in\R^d$,
\[
\beD\cdot \bD(2\,\rho^{(0)})\,\beD
\times\,\rho^{(0)}\,\d_\rho^2H^{(0)}(\rho^{(0)},\bkp)
\,\geq\,0\,.
\]
For comparison, let us point out that $\d_\rho^2H^{(0)}(\rho^{(0)},\bkp)\,=\,4\,\delta_{hyp}$. 

\bl\label{l:constant}
Let $\bU^{(0)}$ be a constant profile for a solution to \eqref{e:ab-more} in the sense of \eqref{def:constant}. Then, with $\rho^{(0)}:=\Mass(\bU^{(0)})$, the corresponding solution is spectrally exponentially unstable if and only if $\bD(2\,\rho^{(0)})$ is not the zero matrix, $\rho^{(0)}\,\delta_{hyp}\neq0$ and one of two following possibilities hold
\begin{enumerate}
\item there exists $\beD_+$ such that $\beD_+\cdot \bD(2\,\rho^{(0)})\,\beD_+>0$ and $\beD_-$ such that $\beD_-\cdot \bD(2\,\rho^{(0)})\,\beD_-<0$;
\item $\bD(2\,\rho^{(0)})$ is nonnegative (respectively nonpositive) and $\delta_{hyp}<0$ (resp. $\delta_{hyp}>0$).
\end{enumerate}
\el

\br\label{rk:delta-hyp} In the present contribution we are interested in constant solutions only as far as they are reachable either as the constant limit in the small-amplitude regime or as the limiting solitary-wave endstate in the large-period regime. Extending \cite[Appendix~A]{BMR2-I}, let us point out that when $\bD$ has a sign (either nonnegative or nonpositive), constant states associated with a large-period regime are always spectrally stable. We prove here this claim for equations of type \eqref{e:ab}. Let us recall that with $\nu$ and $\cW_\rho$ defined through
\begin{align*}
\cW_\rho(\rho)
&=-H^{(0)}(\rho,\nu(\rho)\ex+\tkp)
+\omp\,\rho+\mup\,\nu(\rho)-\cx \rho\,\nu(\rho)\,,\\
0&=-\ex\cdot\,\nabla_\bfv H^{(0)}(\rho,\nu(\rho)\ex+\tkp)
+\mup-\cx\,\rho\,,
\end{align*}
this means that we focus on the case when $\d_\rho^2\cW_\rho(\rho^{(0)})<0$. By differentiating the foregoing identities, one deduces that
\begin{align*}
\d_\rho\cW_\rho(\rho)
&=-\d_\rho H^{(0)}(\rho,\nu(\rho)\ex+\tkp)
+\omp-\cx\nu(\rho)\,,\\
2\,\rho\,\kappa(2\,\rho)\,\nu'(\rho)&=-\ex\cdot\,\d_\rho\nabla_\bfv H^{(0)}(\rho,\nu(\rho)\ex+\tkp)
-\cx\,,\\
\d_\rho^2\cW_\rho(\rho)
&=-\d_\rho^2 H^{(0)}(\rho,\nu(\rho)\ex+\tkp)
+\frac{(\nu'(\rho))^2}{2\,\rho\,\kappa(2\,\rho)}\,.
\end{align*}
From this stems that, for $(\rho^{(0)},\bkp)=(\rho^{(0)},\nu(\rho^{(0)})\ex+\tkp)$, the saddle condition $\d_\rho^2\cW_\rho(\rho^{(0)})<0$ implies $\d_\rho^2 H^{(0)}(\rho,\bkp)>0$ \emph{i.e.} $\delta_{hyp}>0$, as claimed.
\er

\section{Anisotropic equations}\label{s:more-equations}

In the present section, we show how to generalize most of our results from systems of the form~\eqref{e:ab} to systems of the form~\eqref{e:ab-more}, namely
\begin{align*}
\d_t\bU&=\bJ\,\delta \Ham_0[\bU]\,,&
\qquad\text{with}\qquad&&
\Ham_0\left[\bU\right]&=\tfrac12 \nabla_\bfx \bU\cdot\bD(\|\bU\|^2)\nabla_\bfx \bU
+W(\|\bU\|^2)\,,
\end{align*}
where $\bD$ is valued in real symmetric $d\times d$-matrices. As in Appendix~\ref{s:more-waves}, our goal is not to transfer our methodology (with possibly different outcomes), but to point out what is readily accessible by simple changes in notation.

Consistently with the rest of the present paper, we shall discuss explicitly only waves in the form \eqref{def:wave}. Yet let us anticipate from Appendix~\ref{s:more-waves} that even for System~\eqref{e:ab-more} as considered here all longitudinal results apply equally well to waves of the form~\eqref{def:wave-general} and that instability results about general perturbations also generalize when either $\bD$ is constant (semilinear case) or when $d\geq3$ and, for any $\alpha$, $\tkp$ is an eigenvector of $\bD(\alpha)$.

The restriction on generality we make here is that we consider waves of type \eqref{def:wave} propagating in a direction that is a principal direction for the dispersion of \eqref{e:ab-more}. We assume that, for any $\alpha$, $\uex$ is an eigenvector of $\bD(\alpha)$ for a non-zero eigenvalue. This includes the case, considered in \cite{LBJM}, when $d=2$, waves propagate in the direction $\beD_1$ and 
\[
\bD\equiv\bp 1&0\\0&\pm 1\ep\,.
\]
It follows readily from the principal-direction restriction that all longitudinal results still hold with
\be\label{e:new-kappa}
\kappa(\alpha)\,:=\,\uex\cdot\bD(\alpha)\,\uex\,.
\ee
and we recall that at this stage there is no loss in generality in assuming $\kappa$ positive-valued. Unfortunately, in genuinely anisotropic cases, the principal-direction restriction is essentially incompatible with modulation of the direction $\ex$ and thus rules out any hope for a modulational interpretation in the spirit of Section~\ref{s:WKB}.

Therefore, under this assumption, we focus on extending instability results of Section~\ref{s:general}. As far as this goal is concerned, it is sufficient to deal with the case when $d=2$, $\uex=\beD_1$ and
\be\label{e:nondefHam}
\Ham_0\left[\bU\right]=\tfrac12\kappa(\|\bU\|^2)\|\d_x \bU\|^2+W(\|\bU\|^2)
+\tfrac12\tkappa(\|\bU\|^2)\|\d_y \bU\|^2
\ee
with $\kappa$ as in \eqref{e:new-kappa} and $\tkappa$ ranging over all the functions $\tkappa$ given by 
\[
\tkappa(\alpha)\,:=\,\beD\cdot\bD(\alpha)\,\beD
\]
where $\beD$ is a unitary vector orthogonal to $\uex$. Note that this reduction hinges on the obvious facts that there is no loss in taking $\bfeta$ under the form $\|\bfeta\|\,\beD$ with $\beD$ as above, and that, for any $\alpha$, the space of vectors orthogonal to $\uex$ is stable under the action of $\bD(\alpha)$.

Up to minor changes that we detail below, Corollary~\ref{c:transverse} and results of Section~\ref{s:criteria} extend readily to the case \eqref{e:nondefHam} with $\uex=\beD_1$. Indeed, the changes required in Theorem~\ref{th:low-freq} and its proof are purely notational and in the statement the only place where $\kappa$ should be replaced with $\tkappa$ is in the definition of $\Sigma_{\bfy}$, or, in other words, in the definition of $\sigma_1$, $\sigma_2$ and $\sigma_3$. Explicitly, 
\begin{align*}
\sigma_1&:=\int_0^{\uXx}\tkappa(\|\ucV\|^2)\,\|\ucV\|^2\,,&
\sigma_2&:=\int_0^{\uXx}\tkappa(\|\ucV\|^2)\,\bJ\ucV\cdot\ucV_x\,,&
\sigma_3&:=\int_0^{\uXx}\tkappa(\|\ucV\|^2)\,\|\ucV_x\|^2\,.
\end{align*}
The proof of Proposition~\ref{p:hig-freq} requires more significant changes but all of them are elementary. The upshot is that in Proposition~\ref{p:hig-freq} the ellipticity condition $|\lambda|+\|\bfeta\|^2\geq R_0$ should be replaced with $|\lambda|\geq R_0\,(1+\|\bfeta\|^2)$. This weaker conclusion is still sufficient to derive Corollary~\ref{c:transverse}. Once the above-mentioned change in $\Sigma_{\bfy}$ has been performed, all the results of Section~\ref{s:criteria} hold unchanged. 

Note for instance that in Lemma~\ref{l:xi=0}, only the coefficients $\delta_{(m,n,p)}$ with $p\neq0$ depend on the choice of the transverse coefficient $\tkappa$. This stems from the fact that the wave profiles are independent of this coefficient. Note moreover that the dependence of $\delta_{(m,n,p)}$ on $\tkappa$ has the parity of $p$. Thus it follows from Lemma~\ref{l:xi=0} that waves cannot be spectrally stable to perturbations that are longitudinally co-periodic for both $\tkappa$ and $-\tkappa$ except possibly if $\delta_{4,0,0}=\delta_{3,0,1}=\delta_{2,0,2}=0$. Note that in the latter degenerate case, in particular, $0$ has algebraic multiplicity larger than $4$ as an eigenvalue of $\cL_{0,0}$. Moreover it follows from an inspection of the coefficients of $\Sigma_t$ and $\Sigma_{\bfy}$ and a Cauchy-Schwarz argument that this latter degenerate case cannot occur when $\tkappa$ has a definite sign (either positive or negative).

Now we turn to the generalization of asymptotic results in Sections~\ref{s:homoclinic} and~\ref{s:harmonic}. It is important to track there how the replacement of $\kappa$ with $\tkappa$ at some places impacts the proof. In the integral representations of $\sigma_1$, $\sigma_2$ and $\sigma_3$, $\kappa$ should be replaced with $\tkappa$ in definitions of $f_1$, $f_2$ and $f_3$ and the formula for $\sigma_3$ should be modified as 
\begin{align*}
\sigma_3&=
\int_{\rhomin(\mux)}^{\rhomax(\mux)} 
\frac{f_3(\rho)}{\sqrt{\mux-\cW_\rho(\rho))}}
\sqrt{\frac{2\,\kappa(2\,\rho)}{2\,\rho}}\,\dd \rho
+\,\int_{\rhomin(\mux)}^{\rhomax(\mux)} 
\frac{\tkappa(2\,\rho)}{\kappa(2\,\rho)}
\sqrt{\mux-\cW_\rho(\rho))}\sqrt{\frac{2\,\kappa(2\,\rho)
}{2\,\rho}}\,\dd \rho\,.
\end{align*}
Theses changes appear in proofs of Theorems~\ref{th:homoclinic_asymptotics} and~\ref{th:harmonic_asymptotics} only through the value $f_1(2\,\rho^{(0)})$. When $\tkappa(2\,\rho^{(0)})>0$, the arguments still apply so that instability still occurs.

Let us now focus on the case when $\tkappa(2\,\rho^{(0)})<0$. Recall that we have normalized signs to ensure $\kappa(2\,\rho^{(0)})>0$. Thus it follows from Lemma~\ref{l:constant} that if $\delta_{hyp}\neq0$ the limiting constant state is spectrally unstable. Note that, as pointed out in Remark~\ref{rk:delta-hyp}, the condition $\delta_{hyp}\neq0$ holds systematically at the large-period limit. So we only need to explain how to transfer spectral instability from limiting constant states to nearby periodic waves. 
 
In the small-amplitude regimes, the transfer follows from a direct standard perturbation argument for isolated eigenvalues of finite multiplicity, considering the constant-coefficient operators obtained by linearizing about the constant-state as a periodic operator. To perform this comparison, it is important, even at the constant limit, to choose a frame adapted to the harmonic limit. Indeed let us observe that the choice of a frame --- among those in which the reference solution is stationary --- does impact the spectrum of the linearized operator, yet without altering the instable character of this spectrum. We omit details of the standard argument and again refer the reader to \cite{Kato} for background on spectral perturbation theory.

The large-period regime is trickier to analyze but is covered by \cite{Yang-Zumbrun}.

To summarize, we have obtained that under the principal-direction assumption, when $\bD(2\,\rho^{(0)})$ is non trivial --- in the sense that there exists $\beD$ orthogonal to $\uex$ such that $\bD(2\,\rho^{(0)})\beD$ is not zero ---, spectral instability holds
\begin{enumerate}
\item in the large-period regime when at the limiting solitary wave $\d_{\cx}^2\Theta_{(s)}\neq0$;
\item in the small-amplitude regime when at the limiting constant $\d_\rho\nu\neq0$ and $\delta_{hyp}\,\delta_{BF}\neq0$.
\end{enumerate} 

\section{General plane waves}\label{s:more-waves}

In the present section, we show how our general analysis may be applied to more general plane waves in the form \eqref{def:wave-general}. Whereas we believe that our methodology could be applied to all these waves (with possibly different outcomes), our aim here is merely to point out what is readily accessible by a simple change in frame or notation.

As already highlighted in Section~\ref{s:profile-general}, all our longitudinal results apply as they are once one has replaced $W$ with $W_{\tkp}$ defined through
\[
W_{\tkp}(\alpha):=W(\alpha)\,+\,\frac12\,\alpha\,\kappa(\alpha)\,\|\tkp\|^2
=H^{(0)}\left(\frac{\alpha}{2},\tkp\right)\,.
\]
Thus we only need to discuss our results on general perturbations, focused on proving spectral exponential instability. 

\medskip

\noindent {\bf Dimension larger than $2$.} A simple but efficient observation is that when one restricts to perturbations that are constant in the direction of $\tkp$ all transverse contributions due to the fact that $\tkp$ is non-zero do disappear. As a direct consequence, when $\tkp\neq0$ but $d\geq3$, all the spectral instability results hold as they are (up to the change $W\to W_{\tkp}$) and a modulational interpretation is available for the spectral expansion of $D_\xi(\lambda,\bfeta)$ when $(\lambda,\xi,\bfeta)\to (0,0,0)$ under the condition $\bfeta\cdot\tkp=0$. In particular, when $d\geq3$, in non-degenerate cases, spectral instability occur in both small-amplitude and large-period regimes. Except for the generalization of the modulational interpretation, this argument also applies to the more general form of equations considered in Appendix~\ref{s:more-equations} under the assumption that, for any $\alpha$, $\tkp$ is an eigenvector of $\bD(\alpha)$.

\medskip

\noindent {\bf Semilinear case.} In the semilinear case, one may go further by using a form of Galilean invariance. Let us consider System~\eqref{e:ab-more} with $\bD\equiv\bD_0$. Then for any vector $\tkp$ if $\bU$ solves \eqref{e:ab-more} so does 
\[
(t,\bfx)\mapsto 
\eD^{\left(\tkp\cdot\bD_0\tkp\,t+\tkp\cdot\bfx\right)\bJ}\bU(t,\bfx+t\,2\,\bD_0\,\tkp)\,.
\]
The foregoing transformation preserves (in)stability properties and brings waves of type \eqref{def:wave-general} into waves of type \eqref{def:wave}. Thus in the semilinear case there is absolutely no loss in generality in assuming the form \eqref{def:wave}.

\section{Table of symbols}\label{s:index}

We gather here page numbers of main definitions for symbols that are used recurrently throughout the text. Pieces of notation specific to a subsection are not indexed here. For groups of symbols introduced simultaneously, the definitions may run over a few pages. We recall that underlining is used throughout to denote specialization at a specific background wave.

\begin{align*}
\delta,\,\Hess,\,L,\,\otimes,\,\Div\,,&\ \pageref{notation}&
\bJ,\,\kappa\,,W,\,\Ham_0,\,\Mass,\,\bImpulse,\,\Impulse_j\,,&\ \pageref{e:ab}&
\kx,\,\kp,\,\cx,\,\omx,\,\omp,\,\Ham,&\ \pageref{def:wave}\\
\cL,\,\cL^x,\,\cL^\bfy,\,\Ham^x,\,\Ham^\bfy,&\ \pageref{def:L}&
\cL_{\xi,\bfeta},\,\xi,\,\bfeta,\,\cL^x_\xi,\,\cB,\,\check{},\,\cF,\,\widehat{ },&\ \pageref{s:Bloch}&
\Theta,\,\rhomin,\,\rhomax,&\ \pageref{s:action}\\
\mux,\,\mup,\,\Hamp,&\ \pageref{e:unscaled-profile}&
D_\xi(\lambda,\bfeta),\,R(x,x_0;\lambda,\bfeta),&\ \pageref{s:Evans-def}&
\bA_0,\,\bB_0,\,\mass,\,\impulse,&\ \pageref{e:W-intro}\\
\bkx,\,\ex,\,\bkp,\,\tkp,\,\bimpulse,&\ \pageref{def:wave-general}&
\bC_0,\,\tau_0,\,\tau_1,\,\tau_2,\,\tau_3,\,\sigma_1,\,\sigma_2,\,\sigma_3,&\ \pageref{e:W-intro-more}&
\Xx,\,\xip,\,\vphip,\,\vphix,&\ \pageref{s:jumps}\\
\cU(\rho,\theta),\,\cJ,\,H_0,\,Q_j,\,\Hp,&\ \pageref{s:EK}&
\HEK,\,\nu,\,\cW_\rho,&\ \pageref{eq:EK}&
{}^{(0)},\,{}^{(s)},&\ \pageref{s:asymp}\\
\Theta_{(s)},\,\bB\,,H^{(0)},&\ \pageref{def:action-s}&
\Ham_{\tkp},\,\Ham_{0,\tkp},\,W_{\tkp},&\ \pageref{s:profile-general}&
\Sigma_t,\,\Sigma_\bfy,&\ \pageref{th:low-freq}\\
\delta_{hyp},\,\delta_{BF},&\ \pageref{def:sideindex}&
\Delta_0,\,&\ \pageref{def:Delta}&
\delta_{(m,n,p)},&\ \pageref{eq:poly}\\
f_1,\,f_2,\,f_3,&\ \pageref{e:fj}
\end{align*}



\label{biblio}
\newcommand{\etalchar}[1]{$^{#1}$}
\newcommand{\SortNoop}[1]{}\def\cprime{$'$}

\end{document}